\documentclass[11pt,a4paper]{article}
\usepackage[cp1251]{inputenc}
\usepackage{amsmath,amsthm}
\usepackage{amsfonts,amstext,amssymb,verbatim,epsfig,color}
\usepackage{dsfont}
\usepackage{enumerate}
\usepackage{psfrag}
\usepackage{graphicx}
\usepackage{hyperref}

\def\R{{\mathbb{R}}}

\def\N{{\mathbb{N}}}
\def\Z{{\mathbb{Z}}}

\newtheorem{theorem}{Theorem}[section]
\newtheorem{corollary}[theorem]{Corollary}
\newtheorem{lemma}[theorem]{Lemma}
\newtheorem{proposition}[theorem]{Proposition}
\newtheorem{definition}[theorem]{Definition}
\newtheorem{remark}[theorem]{Remark}
\newtheorem{claim}[theorem]{Claim}

\newtheoremstyle{likedef}
  {}%
  {}%
  {}%
  {\parindent}%
  {\bfseries}%
  {.}%
  {.5em}%
  {}%

\theoremstyle{likedef}

\numberwithin{equation}{section}

\begin{document}
\title{Percolation on the stationary distributions of the voter model}\author{Bal\'azs R\'ath\textsuperscript{1}, Daniel Valesin\textsuperscript{2}}

\footnotetext[1]{MTA-BME Stochastics Research Group, Hungary. E-mail: \url{rathb@math.bme.hu}}
\footnotetext[2]{University of Groningen, The Netherlands. E-mail: \url{d.rodrigues.valesin@rug.nl}}

\maketitle
\begin{abstract}

The voter model on $\Z^d$ is a particle system that serves as a rough model for changes of opinions among social agents or, alternatively, competition between biological species occupying space. When $d \geq 3$, the set of (extremal) stationary distributions is a family of measures $\mu_\alpha$, for $\alpha$ between 0 and 1. 
A configuration sampled from $\mu_\alpha$ is a strongly correlated field of 0's and 1's on $\Z^d$ in which the density of 1's is $\alpha$. 
We consider such a configuration as a site percolation model on $\Z^d$. We prove that if $d \geq 5$, the probability of existence of an infinite percolation cluster of 1's exhibits a phase transition in $\alpha$. If the voter model is allowed to have sufficiently spread-out interactions, we prove the same result for $d \geq 3$.

\bigskip

\medskip


\end{abstract}

\section{Introduction}
\label{section:intro}

\subsection{Model and results}

Given integers $d \ge 1$ and $R \ge 1$, the voter model with range $R$ on the $d$-dimensional lattice $\Z^d$  is a Markov process, denoted here by $(\xi_t)_{t \ge 0}$, with configuration space $\{0,1\}^{\Z^d}$ and stochastic dynamics described informally as follows.
 Each vertex (or site) $x$ of $\Z^d$ updates its current state $\xi_t(x) \in \{0,1\}$
 at rate one by copying the state $\xi_t(y)$ of a vertex $y$ that is chosen uniformly among all vertices at ($\ell^1$-norm) distance at most $R$ from $x$. 

In Section \ref{section_voter_graphical_construction} we give the formal definition of the model and recall some of its relevant properties. In this Introduction, we will only very briefly present the concepts that are needed to state our main results.

The voter model was introduced independently by Clifford and Sudbury in \cite{cliff} and Holley and Liggett in \cite{holl}. In the interpretation of the latter pair of authors, each site of $\Z^d$ represents a voter which can have one of two possible opinions (corresponding to the states 0 and 1). The model thus represents the evolution of the opinions among the population. Clifford and Sudbury gave a biological interpretation for the model: there are two competing species, denoted 0 and 1, and each site is a region of space that can be occupied by an individual of one of the two species.

The set of stationary distributions of the voter model on $\Z^d$ has been thoroughly studied; the following is a summary of known results. For fixed $d \ge 3$, $R \ge 1$ and $\alpha \in [0,1]$, one defines a probability measure $\mu_\alpha$ on $\{0,1\}^{\Z^d}$ as the distributional limit (which is shown to exist), as time is taken to infinity, of the voter model with the random initial configuration in which the states of all sites are independent and Bernoulli($\alpha$). $\mu_\alpha$ is then stationary for the voter model dynamics. Moreover, it is shown that the set of stationary distributions for the voter model dynamics that are \textit{extremal} -- i.e., that cannot be expressed as non-trivial convex combinations of other stationary distributions -- is precisely the family
$$\{ \, \mu_\alpha: \alpha \in [0,1] \, \}.$$
We note that this property of the voter model is rather delicate and small perturbations of the dynamics
can result in an interacting particle system which has only one non-trivial stationary distribution, see 
\cite{cp14}.

The measures $\mu_\alpha$ can be obtained in a more constructive way with the aid  of coalescing random walks.  A realization of a system of coalescing random walks with range $R$ on $\Z^d$ induces a partition of $\Z^d$ into \textit{coalescence classes}: we say that $x$ and $y$ are in the same class if the walkers started at $x$ and $y$ are eventually joined. We then assign 0's or 1's to the coalescence classes independently with probabilities $1-\alpha$ and $\alpha$, respectively, and the resulting configuration $\xi \in \{0,1\}^{\Z^d}$ has law $\mu_\alpha$. (Again, the sentences in this paragraph will be given a precise meaning in Section \ref{section_voter_graphical_construction}).

With the aid of this construction, it is not difficult to show that each $\mu_\alpha$ is invariant and ergodic with respect to translations of $\Z^d$ (see \cite[Theorem 2.5 of Chapter V, Corollary 4.14 of Chapter I]{lig85})
 and satisfies $\mu_\alpha[\, \xi(0) = 1 \, ] = \alpha$, so that $\alpha$ is equal to the density of 1's. Moreover, the family $\{\mu_\alpha\}$ is stochastically increasing: in the partial order on $\{0,1\}^{\Z^d}$ induced by the order $0<1$ on the coordinates, we have that $\mu_\alpha$ is stochastically dominated by $\mu_{\alpha'}$ when $\alpha < \alpha'$.

\medskip

The objective of this paper is to show that the measures $\mu_\alpha$ exhibit a non-trivial percolation phase transition.
 Loosely speaking, we want to show that if $\alpha$ is close to zero then the set of 1's only contains finite connected components and if $\alpha$ is close to one then the set of 1's contains an infinite component.
  Let us explain this concept more precisely.
   We define the event $\mathrm{Perc} \subseteq \{0,1\}^{\Z^d}$ which consists of those voter configurations
    $\xi$  for which the 
subgraph of the nearest-neighbour lattice $\Z^d$ spanned by the set of 
occupied sites $\{x: \xi(x) = 1\}$  has an infinite connected component.
 By ergodicity, $\mu_\alpha(\mathrm{Perc} )$ is either 0 or 1. If $\mathrm{Perc}$ occurs, we say that
 the set $\{x: \xi(x) = 1\}$ percolates.
  We can then define $\alpha_c$ as the supremum of all the values of $\alpha$ for which $\mu_\alpha(\mathrm{Perc}) = 0$.
  By the stochastic ordering mentioned in the previous paragraph, $\mu_\alpha(\mathrm{Perc})$ is non-decreasing in $\alpha$.
  Thus for any $\alpha<\alpha_c$ we have $\mu_\alpha(\mathrm{Perc}) = 0$ and for any $\alpha>\alpha_c$ we have
   $\mu_\alpha(\mathrm{Perc}) = 1$.
    Our aim is to show that the family of measures $\{\mu_\alpha: 0 \leq \alpha \leq 1\}$ 
    exhibits a non-trivial percolation percolation phase transition, i.e., that $0<\alpha_c<1$. Our main results are
    
\begin{theorem}\label{thm:nearest_neighbour}
If $d \geq 5$ and $R \geq 1$, then the family of stationary distributions of the voter model exhibits a non-trivial percolation phase transition.
\end{theorem}
\begin{theorem}\label{thm:spread_out}
If $d = 3$ or $4$ then there exists $R_0=R_0(d) \in \N$ such that if $R \geq R_0$ then the family of stationary distributions of the voter model exhibits a non-trivial percolation phase transition.
\end{theorem}

\subsection{Context}

Although it may at first seem intuitively clear that, similarly to the case of Bernoulli percolation, $\xi$ should be non-percolative if $\alpha$ is  close to zero, this
 statement is not obvious. As the dynamics of the voter model favours that voters synchronize their opinions, the measures $\mu_\alpha$ present long-range dependences. In fact, it follows from \eqref{eq:stationary_mean_corr} below that
  for any $\alpha \in (0,1)$, the configuration $\xi$ under the law $\mu_\alpha$ has covariances given by 
\begin{equation}\label{eq:intro_corr}
c(\alpha,d,R)\cdot |x-y|^{2-d} 
\leq
 \text{Cov}_{\mu_{\alpha}}(\xi(x),\xi(y)) 
 \leq
  C(\alpha, d,R)\cdot |x-y|^{2-d}, \qquad x \neq y \in \Z^d.
\end{equation}
It is \textit{a priori} possible that percolation models with strong correlations present no phase transition. It is easy to build artificial examples, but let us recall an example that arises ``naturally''. 
The random interlacement set $\mathcal{I}^u$ at level $u > 0$, introduced in \cite{sznitman_interlacement} 
is a random subset of $\Z^d$: 
(a) the law of $\mathcal{I}^u$ is stochastically dominated by the law of $\mathcal{I}^{u'}$ when $u < u'$, 
  (b) the correlations of $\mathcal{I}^u$ decay like \eqref{eq:intro_corr} 
(see \cite[(1.68)]{sznitman_interlacement}) and (c) the density of $\mathcal{I}^u$ can be taken arbitrarily small by 
making $u$ small (see \cite[(1.58)]{sznitman_interlacement}), yet the set $\mathcal{I}^u$ is connected for any $u>0$, (see \cite[(2.21)]{sznitman_interlacement}).

On the other hand, in case one attempts to prove that phase transition \textit{does} occur, 
then the slowly decaying correlations \eqref{eq:intro_corr}  pose a challenge, 
as many of the well-known tools that are used for Bernoulli percolation are not applicable.
 Additionally, since general criteria are lacking and (as mentioned above) phase transition may in principle fail to occur, one needs to envisage strategies of proof that are model-specific.
  The proof of non-degeneracy of the percolation threshold has been carried out for the vacant set
    $\mathcal{V}^u = \Z^d \setminus \mathcal{I}^u$ of random interlacements in \cite{sznitman_interlacement, SidoraviciusSznitman_RI} and the excursion sets of the Gaussian free field in \cite{BLM87} (for $d = 3$) and \cite{RS13} ($d \geq 3$). 
    Both of these percolation models exhibit a decay of correlations described by \eqref{eq:intro_corr}.

\smallskip  

In the case of the voter model, the question of percolation has been considered before, in \cite{LS86}, \cite{BLM87}, \cite{ML06} and \cite{Ma07}.  The main focus of these works is on the  case where $d=3$ and $R=1$.
 Through simulations and numerical studies, the first, third and fourth of these references argue that there should be a non-trivial phase transition and that the predictions of 
 \cite{Halperin_Weinrib, Weinrib} regarding the critical behaviour of percolation models with correlations described by 
 \eqref{eq:intro_corr} should be correct.
 However,  the problem of finding a rigorous proof of the non-triviality of the percolation phase
  transition of the stationary state of the voter model remained open.
  This problem is (partially) settled by our Theorems \ref{thm:nearest_neighbour} and \ref{thm:spread_out}.

Another investigation of geometric properties of the stationary distribution of the voter model has recently been carried out in \cite{HMN14}. The object of interest there is the voter model on a finite rhombus of the triangular lattice; the boundary of the rhombus, composed of four segments, is frozen so that two adjacent segments are always in state 0 and the other two in state 1. In this finite setting, there is only one stationary distribution, which can be constructed with the aid of coalescing random walks and the resulting coalescence classes, similarly to the $\mu_\alpha$'s on $\Z^d$. The authors study the volume of the coalescing classes and the interface curve that appears as a consequence of the opposing boundaries.

\smallskip

Questions regarding percolation of the stationary distributions of interacting particle systems other than the voter model have also been investigated. It is proved in \cite{LS06} that the
upper invariant measure $\overline{\nu}_\lambda$ of the contact process with infection rate $\lambda$ on $\Z^d,\, d \geq 2$
 percolates if $\lambda \geq 6.25$. To the best of our knowledge,
 Question 2 of Section 8 of \cite{LS06} is still open, i.e., it is not known whether there exists $\lambda>\lambda_c$
 for which $\overline{\nu}_\lambda$ is non-percolative. 
  However, it is proved in \cite{rob_sharp} that for $d=2$ the percolation phase transition of $\overline{\nu}_\lambda$ is
  sharp. This result is extended to more complex versions of the contact process in  \cite{markus_jakob_rob_sharp}.
Let us note here that the stationary distribution $\mu_\alpha$ of the voter model is rather different from the
upper invariant measure $\overline{\nu}_\lambda$ of the contact process, e.g.,
$\text{Cov}_{\overline{\nu}_\lambda}(\xi(x),\xi(y)) $ decays exponentially as $|x-y| \to \infty$
 for any value of $\lambda$ (see, e.g., \cite[Lemma 2.2]{rob_sharp}), as opposed to the polynomial decay exhibited by $\mu_\alpha$ in \eqref{eq:intro_corr}.
  
\smallskip  
  
Let us also point out that 
the scaling limit of the voter model is super-Brownian motion 
(see \cite{cdp, voter_cluster_scaling_limit_super}), and, despite the fact that continuum scaling limits
do not explicitly appear in the calculations that we are about to present, our intuition was guided by the question of
 the disconnecedness of the support of super-Brownian motion, as we discuss in Remark \ref{remark_conjecture_one_block}.

\subsection{Ideas and structure of proof}

Let us now explain how the paper is organized and also the contents of each section. 

\smallskip

In Section \ref{s:prelim}, we give a notation summary and also collect some facts regarding martingales and random walks
  that are needed in the rest of the exposition.
  
\smallskip  
  
   Section \ref{section_voter_graphical_construction} contains an introduction to the voter model on $\Z^d$, including its graphical construction, duality properties and the construction of the extremal stationary distributions using a family of coalescing random walks.
  
\smallskip

We begin to prove  our main results in Section \ref{s:first}. 
Our goal is to show (see \eqref{eq:noperc}) that for sufficiently small values of $\alpha$,
 the probability that a large annulus is crossed by a $*$-connected path of 1's in $\xi$ is smaller than a stretched exponential function of the radius of the annulus. The condition \eqref{eq:noperc} is then shown to imply  $0<\alpha_c<1$. 
It is self-evident that if \eqref{eq:noperc} holds, then there is no percolation for small enough $\alpha$. We also show, through a classical argument using planar duality, that \eqref{eq:noperc} implies that if $\alpha$ is close enough to 1, then there is percolation. 

We were able to establish \eqref{eq:noperc} for the two sets of assumptions that appear in our main theorems (namely: first for $d \geq 5,\; R\geq 1$  and second for $d\geq 3$ and $R$ large enough).
 We prove both cases using a renormalization scheme inspired by Sections 2 and 3 of \cite{sznitman_decoupling}, 
 which involves embeddings of binary trees into $\Z^d$ that are ``spread-out on all scales".
  In Section \ref{subsection:renorm}, we present this renormalization scheme and  some of its properties.

\smallskip
 
In Section \ref{section_spread_out} we establish \eqref{eq:noperc} for $d \geq 3$ and $R$ large, and in Section \ref{section_nearest_neighbour} we establish it for $d \geq 5$ and $R \geq 1$. For simplicity of notation, Section \ref{section_nearest_neighbour} only treats explicitly $d \geq 5$ and $R=1$ (i.e., the case of nearest neighbour interactions), but it will be easy to see that the proof given there applies for any value of $R$. In fact, the proof of Section \ref{section_nearest_neighbour} could also be adapted to cover the case of $d \geq 3$ and $R$ large enough, so that Section \ref{section_spread_out} is (strictly speaking) redundant. We have nevertheless chosen to include it for three reasons: first, because it is quite short; second, because the method might find other applications; and third,
 the contents of Section \ref{section_spread_out}  may be helpful for the reader to grasp the more involved arguments of Section \ref{section_nearest_neighbour}.

\smallskip
 
A common point in the proofs of Section \ref{section_spread_out} and \ref{section_nearest_neighbour} is the need to provide an upper bound for probabilities of the form
\begin{equation}
\label{eq:heuristics_intro}
\mu_\alpha\left[\, \xi(x) = 1 \text{ for all } x \in \mathcal{X}\, \right]
\end{equation}
for certain finite sets $\mathcal{X} \subseteq \Z^d$ that appear at the ``bottom'' scale of the renormalization construction. An immediate consequence (as we will explain in Section \ref{section_voter_graphical_construction}, up to equation \eqref{eq:dualityinf}) of the construction of $\mu_\alpha$ through ``coalescence classes'' is that \eqref{eq:heuristics_intro} is equal to $\mathbb{E}\left[\alpha^{\mathcal{N}_\infty(\mathcal{X})}\right]$, where $\mathcal{N}_\infty(\mathcal{X})$ is the (random) terminal number of random walkers in a system of coalescing random walks started from the configuration in which there is one walker in each vertex of $\mathcal{X}$.
 Hence, in order to give a good upper bound for \eqref{eq:heuristics_intro}, one needs to argue that $\mathcal{N}_\infty(\mathcal{X})$ is comparable to  $|\mathcal{X}|$ (the cardinality of $\mathcal{X}$). It is worth noting that $\alpha^{|\mathcal{X}|}$ is the probability of the event in \eqref{eq:heuristics_intro} for independent, $\mathrm{Bernoulli}(\alpha)$ percolation.

Our renormalization construction ensures that the set $\mathcal{X}$ under consideration here is
 ``sparse on all scales". 
 Hence, one expects that walkers started from the vertices of $\mathcal{X}$ tend to avoid other walkers, and 
  the amount of loss due to coalescence, $|\mathcal{X}| - \mathcal{N}_\infty(\mathcal{X})$, is far from 
  $|\mathcal{X}|$ with overwhelming probability. In order to make this precise, we use different strategies in Sections \ref{section_spread_out} and \ref{section_nearest_neighbour}. Both of these techniques are novel.
  \begin{itemize}
\item ($d \geq 3$, $R \gg 1$) In Section \ref{section_spread_out}, we replace the system of coalescing random walks with a system of \textit{annihilating} random walks and observe that annihilation events are ``negatively correlated". 
This allows us to derive a useful explicit bound on
\eqref{eq:heuristics_intro} which is particularly effective if the range $R$ of the walkers is big enough to guarantee
that the expected number of annihilations is sufficiently small.

 \item ($d \geq 5$, $R=1$) The proof of Section \ref{section_nearest_neighbour}  involves two important ideas. 
 First, it turns out that under some carefully constructed circumstances one can 
 run the walkers for some period of time independently from each other (i.e., without coalescence), which 
 allows them to ``wander away" from each other before they start to coalesce.
 Second, we reveal the paths of random walkers one by one and pre-emptively throw away those future walks
  that are too likely to coalesce with the ones already revealed. We can then control 
\begin{enumerate}[(a)]
  \item the number of walkers that we throw away and 
  \item the number of coalescences occurring between the remaining walkers
\end{enumerate}  
in such a way that the sum of these two numbers (which is greater than or equal to $|\mathcal{X}| - \mathcal{N}_\infty(\mathcal{X})$) 
 is not too big compared to $|\mathcal{X}|$.
 \end{itemize}

To state the obvious, Theorems \ref{thm:nearest_neighbour} and \ref{thm:spread_out} leave open the cases of
 dimension 3 and 4 and range $R$ small, even though, as mentioned above, simulations and numerical work suggest that non-trivial phase transition should also occur in these cases. 
 In our final Section \ref{section_remarks}, we
 give an heuristic explanation to the ineffectiveness of the method of Section \ref{section_nearest_neighbour} in treating $d = 3, 4$ and $R = 1$ (see Remark \ref{remark_why_no_d_3_4}).
 In Remarks \ref{remark_why_no_L_equals_1_in_nearest_neighbour_case} and \ref{remark:why_no_annih_in_nearest_neighbour}  we
   explain why the tricks of Section \ref{section_spread_out} are
    insufficient to prove Theorem \ref{thm:nearest_neighbour}, so that we could not do without the more involved method of Section \ref{section_nearest_neighbour}. In Remark \ref{remark_conjecture_one_block}
   we heuristically explain how voter model percolation is related
 to the question of disconnectedness of the closed support of super-Brownian
  motion.

\section{Notation and preliminary facts}
\label{s:prelim}

\subsection{Summary of notation}
Given a set or event $A$, we denote by $\mathds{1}_A$ its indicator function and by $|A|$ its cardinality. 
 
Given a vertex $x\in\Z^d$, we denote by $|x|$ its $\ell^\infty$ norm and by $|x|_1$ its $\ell^1$ norm. We then write
\begin{equation}\label{eq:def_balls}
\begin{array}{ll}
B(L) = \{x\in\Z^d: |x|\leq L\},& B(x,L) = x + B(L);\\
B_1(L) = \{x\in\Z^d: |x|_1\leq L\},& B_1(x,L) = x + B_1(L);\\
S(L) = \{x\in\Z^d: |x| =  L\},& S(x,L) = x + S(L).\end{array}
\end{equation}
If for $x,y\in\Z^d$ we have $|x-y|_1 =1$, then these points are said to be neighbors, and we abbreviate this by $x\sim y$. They are $*$-neighbors if $|x-y|=1$.
For sets $A, B \subset \Z^d$, $\mathrm{dist}(A, B) =\min\{|x-y|: x\in A,\;y \in B\}$. The expression $A \subset \subset \Z^d$ indicates that $A$ is a finite subset of $\Z^d$.

A nearest-neighbor path in $\Z^d$ is a (finite or infinite) sequence $\gamma(0), \gamma(1),\ldots$ so that $\gamma(i+1) \sim \gamma(i)$ for each $i$. A $*$-connected path is a sequence $\gamma(0), \gamma(1),\ldots$ so that $\gamma(i+1)$ and  $\gamma(i)$ are $*$-neighbors for each $i$. We denote by $\{\gamma\}$ the set $\{\gamma(0), \gamma(1),\ldots\}$.

\begin{definition}\label{def_connections_crossings}
Let $\xi \in \{0,1\}^{\Z^d}$ and let $A$ and $B$ denote two disjoint subsets of $\Z^d$.
\begin{enumerate}[(a)]
\item \label{crossing_nearest_neighbour}
 We say $A$ and $B$ are connected by an open path in $\xi$ (and write $A \stackrel{\xi}{\leftrightarrow} B$) if there exists a nearest-neighbor path $\gamma(0),\ldots,\gamma(n)$ such that $\gamma(0)$ is the neighbor of a point of $A$, $\gamma(n)$ is a neighbor of a point of $B$ and $\xi(\gamma(i)) = 1$ for each $i$.
\item \label{crossing_star_neighbour}
Similarly, we write $A \stackrel{*\xi}{\longleftrightarrow} B$ if there exists a $*$-connected path $\gamma(0),\ldots,\gamma(n)$ so that $\gamma(0)$ is the $*$-neighbor of a point of $A$, $\gamma(n)$ is the $*$-neighbor of a point of $B$ and $\xi(\gamma(i)) = 1$ for each $i$.
\end{enumerate}
\end{definition}


\subsection{Martingale facts}
We will need a concentration inequality involving continuous-time martingales. We start recalling two definitions. Consider a probability space with a filtration $(\mathcal{F}_t)_{t \ge 0}$.

\begin{definition}
\label{def_predictable_process}
 A process $(X_t)_{t \geq 0}$ is predictable with respect to $(\mathcal{F}_t)$ if $$X_t \in \mathcal{F}_{t-} = \sigma\left(\cup_{s < t} \mathcal{F}_s \right)\qquad \text{for all $t$}.$$
 \end{definition}
Note that if $(X_t)$ is continuous and adapted to $(\mathcal{F}_t)$, then it is predictable with respect to $(\mathcal{F}_t)$.

\begin{definition}\label{def_q_variation}
Let $(N_t)_{t \ge 0}$ be a square-integrable c\`adl\`ag  martingale with respect to $(\mathcal{F}_t)_{t \ge 0}$. The predictable quadratic variation of $(N_t)$ is the predictable process $(\langle N \rangle_t)_{t \ge 0}$ such that $(N^2_t-\langle N \rangle_t)_{t \ge 0}$ is a martingale with respect to $(\mathcal{F}_t)$.
\end{definition}

The almost sure uniqueness of the predictable quadratic variation follows from Doob-Meyer-Dol\'eans decomposition (\cite[Theorem 25.5]{kallenberg}) applied to the submartingale $(N_t^2)$. Note that $\langle N \rangle_t$ is a non-decreasing function of $t$.
 We refer the reader to \cite[Proposition 26.1]{kallenberg} for elementary properties of $\langle N\rangle$. 
 The result we will need, which follows from \cite[Theorem 26.17]{kallenberg}, is: 

\begin{theorem}
\label{theorem_kallenberg}
Let $S \in [0, +\infty]$.
Let $(N_t)$ be a square-integrable c\`adl\`ag martingale with
$\langle N \rangle_{S} \leq \sigma^2$ almost surely for some $\sigma^2 \in (0, +\infty)$.
Assume that the jumps of $N$ are almost surely bounded by $\Delta \in (0, \sigma]$. 
Then we have
\begin{equation}\label{eq:ineq_kallengerg}
\mathbb{P}\left( \max_{t \in [0, S]} N_{t}-N_0 \geq r \right) 
\leq
  \exp \left(-\frac12 \frac{r}{\Delta} \ln\left(1+\frac{r \Delta}{\sigma^2} \right) \right), \qquad r \geq 0.
\end{equation}
\end{theorem}

Note that \cite[Theorem 26.17]{kallenberg} is only stated for the $\sigma=1$ case; however, our version
\eqref{eq:ineq_kallengerg} follows from an application of that theorem to the martingale
$N_t / \sigma$.

\subsection{Random walk facts}
\label{subsection_random_walk_facts}

\begin{definition}
Given $R \in \N_+$, we say that $(X^z_t)_{t \ge 0}$ is an $R$-spread-out random walk on $\Z^d$ starting at $z \in \Z^d$
if $X^z_0=z$ and $(X^z_t)_{t \geq 0}$ is a continuous-time c\`adl\`ag \ Markov process on $\Z^d$ with infinitesimal generator
\begin{equation*}
(\mathcal{L}f)(x) = 
\sum_{\substack{y \in \mathbb{Z}^d: \\ 0 < |x - y|_1 \le R}}
\frac{  f(y) - f(x)}{|B_1(R)|-1},
\end{equation*}
where  $f: \Z^d \to \mathbb{R}$.
When $R=1$, then we call $(X^z_t)$ a (continuous-time) nearest-neighbour simple random walk on $\Z^d$.
\end{definition}
In words: the holding times between jumps are i.i.d.\ with $\mathrm{Exp}(1)$ distribution and if a jump occurs at time
$t$ and $X^z_{t-}=x$ then $X^z_t$ is uniformly distributed on $B_1(x,R) \setminus \{x\}$. If $R=1$, then
$X^z_t$ is uniformly distributed on the set of nearest neighbours of $x$.

Let us formulate a useful corollary of Theorem \ref{theorem_kallenberg} about random walks:

\begin{corollary}\label{lemma:srw_large_dev_estimate}
 Let $X_t$ denote a $d$-dimensional continuous-time nearest-neighbour 
simple random walk with jump rate $1$ starting at the origin. Then for any $S,r \geq 0$ we have
\begin{equation}
\label{eq:srw_large_dev_estimate}
\mathbb{P} \left[ \max_{0 \leq t \leq S} |X_t| > r \right] \leq 2d 
\exp \left( -\frac12 r \ln \left( 1+ \frac{d \cdot r}{S}  \right) \right).
\end{equation}
\end{corollary}
\begin{proof} 
The $d$ coordinates of $X_t$ are 1-dimensional simple random walks with jump rate $1/d$, hence after 
a union bound we only need to apply \eqref{eq:ineq_kallengerg} with  $\sigma^2=S/d$ and $\Delta=1$ to 
achieve \eqref{eq:srw_large_dev_estimate}.
\end{proof}

Let us define the transition kernel and the Green function of $R$-spread-out random walk on $\Z^d$ by
\begin{equation}\label{eq:def_transition_pdf_green}
p_{R,t}(x,y)=\mathbb{P} \left[ X^z_{s+t}=y \, | \, X^z_s=x \right], \qquad 
g_R(x,y)=\int_0^{\infty} p_{R,t}(x,y)\, \mathrm{d}t.
\end{equation}
If $R=1$ then we drop the $R$ from the subscript and simply denote $p_t(x,y)$ and $g(x,y)$.
We have
\begin{equation}
\label{eq:pdf_green_basic_facts}
\begin{array}{c}
p_{R,t}(x,y)=p_{R,t}(y,x), \quad p_{R,t}(x,y)=p_{R,t}(y-x,0), \\
g_R(x,y)=g_R(y,x), \quad g_R(x,y)=g_R(y-x,0), \quad g_R(x,x) \geq 1.
\end{array}
\end{equation}

It follows from the Chapman-Kolmogorov equations for $p_{R,t}(\cdot,\cdot)$ that we have
\begin{equation}\label{eq:chapman_kolmogorov}
\sum_{y \in \Z^d} p_{R,T}(x,y) \cdot g_R(y,z)=\int_T^{\infty} p_{R,t}(x,z)\, \mathrm{d}t.
\end{equation}

It follows from the Local Central Limit Theorem (see \cite[Section 1.2]{L96}) that for any $d \geq 3$ there exist constants $c=c(d,R)>0$ and $C=C(d,R)<+\infty$ such that
\begin{equation}\label{green_bounds}
\frac{ \int_T^{\infty} p_{R,t}(x,y)\, \mathrm{d}t  }{ (| x-y| \vee \sqrt{T}  +1)^{2-d} } \in [c,C],
 \qquad x,y \in \Z^d, \quad T \geq 0. 
\end{equation}

It follows from the strong Markov property of random walks that we have
\begin{equation}\label{eq:hitting_prob_green}
P[ \; \exists \; t \geq 0 \; : \; X^x_t=y \; ] = \frac{g_R(x,y)}{g_R(y,y)} 
\stackrel{ \eqref{eq:pdf_green_basic_facts} }{\leq} g_R(x,y).
\end{equation}
 The distributions of the increments of our random walks are symmetric, therefore 
 if the random walks $(X^x_t)$ and $(X^y_t)$ are independent, then 
\begin{equation}\label{eq:difference_of_RWs_is_RW}
\left(X^y_t - X^x_t\right)_{t \geq 0} \quad \text{has the same law as} \quad
\left(X^{y-x}_{2t}\right)_{t \geq 0}.
\end{equation}

Let us define
\begin{equation}\label{eq:def_spread_out_hitting_prob}
 h_R(x,y)=P[ \; \exists \; t \geq 0 \; : \; X^x_t=X^y_t \; ], \qquad x,y \in \Z^d
\end{equation}
the probability that two independent $R$-spread-out random walks started from $x$ and $y$ ever meet.
We have
\begin{equation}\label{eq:meet_green}
h_R(x,y)
\stackrel{ \eqref{eq:difference_of_RWs_is_RW}, \eqref{eq:hitting_prob_green}, \eqref{eq:pdf_green_basic_facts} }{=}
\frac{g_R(x,y) }{g_R(0,0)} \leq g_R(x,y).
\end{equation}

In Section \ref{section_spread_out} we will make use of the following claim about spread-out random walks:
\begin{claim} \label{claim_spread_out_green}
Given $d \geq 3$, there exists $f: \N \to \R_+$  such that
\begin{equation}\label{spread_out_hitting_bound}
\forall \; R \in \N, \; x \neq y \in \Z^d\; : \quad
 h_R(x,y) \leq f(R)\cdot |x-y|^{2-d},  \qquad \lim_{R \to \infty} f(R)=0.
\end{equation}\label{claim:spread_out}
\end{claim}
\begin{remark} 
Before proving Claim \ref{claim_spread_out_green}, we first observe that, for fixed $R$, the bound $\sup_{x \neq y \in \Z^d} h_R(x,y) \cdot |x-y|^{d-2} < \infty$ already follows from \eqref{eq:def_transition_pdf_green}, \eqref{green_bounds} and \eqref{eq:meet_green}. The bound \eqref{spread_out_hitting_bound}
 is more informative, as it gives us asymptotic information as $R \to \infty$.
\end{remark}
\begin{proof}[Proof of Claim \ref{claim:spread_out}]
The bound \eqref{spread_out_hitting_bound} follows from \eqref{eq:meet_green},
 \cite[Proposition 1.6]{vdHofstad_Hara_Slade} and the observation
that the Green function of a continuous-time random walk with jump rate $1$ is identical to
 the Green function of the corresponding discrete-time random walk.  To see how the mentioned result in \cite{vdHofstad_Hara_Slade} is applied, first note that their parameter $L$ translates to our parameter $R$ and their expression $S_1(x)$ is equal to our $g_R(0,x)$. Then, by letting their parameters $\alpha$ and $\mu$ both be equal to 1, their equation (1.36) yields that there exists $C < \infty$ such that (in our notation):
\begin{equation*}
 g_R(0,x) \leq CR^{-1} |x|^{2-d} \qquad \text{for $R$ large enough and all $x \in \Z^d,\;x\neq 0$}.
 \end{equation*}
Claim \ref{claim:spread_out} readily follows by \eqref{eq:meet_green}.
\end{proof}

We will also make use of the following bound on the difference of Green function values of nearest neighbour sites:
there exists a $C=C(d)$ such that
\begin{equation}\label{eq:green_diff}
\vert g(x,y) - g(x,y+e) \vert \le C \cdot \left(| x-y |+1\right)^{1-d}, \qquad  x,y \in \Z^d, \; e \sim 0.
\end{equation}
This bound follows from the much stronger \cite[Theorem 1.5.5]{L96}.

The following heat kernel bound follows from the Local Central Limit Theorem: there exist 
$C=C(d)<+\infty$ and $c=c(d)>0$ such that
\begin{equation}\label{eq:heat_kernel}
p_t(x,y) \leq C  t^{-\frac{d}{2}}  \exp\left(-c \frac{ |x-y|^2}{t}\right), \qquad   x,y \in \Z^d, \; t \geq 1.
\end{equation}

In Section \ref{section:kallenberg} we will make use of the following bound.
\begin{claim}\label{claim_green_diff_heat_conv_bound}
There exists  $C=C(d)$ such that
\begin{equation}
\label{eq:weighted_power_1_minus_d}
\sum_{w \in \Z^d} p_t(y,w) |g(w,v)-g(w,v+e)| \leq C t^{\frac12 - \frac{d}{2}}, \qquad   y,v \in \Z^d,\;e \sim 0, \; t \geq 1. 
\end{equation}
\end{claim}
\begin{proof} By  \eqref{eq:pdf_green_basic_facts} we may assume $y=0$ without loss of generality.
\begin{eqnarray*}
\sum_{w \in \Z^d} p_t(0,w) |g(w,v)-g(w,v+e)| &\stackrel{\eqref{eq:green_diff},\eqref{eq:heat_kernel}}{\leq}
C'  t^{-\frac{d}{2}} \sum_{w \in \Z^d}   \exp\left(-c \frac{ |w|^2}{t}\right)  \left(| v-w |+1\right)^{1-d}
\\&\hspace{-2.5cm}\stackrel{(*)}{\leq} 
C'  t^{-\frac{d}{2}} \sum_{w \in \Z^d}  \exp\left(-c \frac{ |w|^2}{t}\right)  \left(| w |+1\right)^{1-d} \\
&\hspace{-.3cm}\leq C''  t^{-\frac{d}{2}} \sum_{n=0}^{\infty} \exp\left(-c \frac{ n^2}{t}\right)  \left(n+1\right)^{1-d} \cdot n^{d-1}
\leq C t^{\frac12 - \frac{d}{2}},
\end{eqnarray*}
where $(*)$ follows from the \emph{rearrangement inequality} \cite[Section 10.2, Theorem 368]{hlp52}.

\end{proof}

\section{Voter model: graphical construction, duality, stationary distributions}
\label{section_voter_graphical_construction}

In this section we define the voter model on $\mathbb{Z}^d$ and present some well-known facts about it. We refer the reader to \cite{lig85} for an introduction to the voter model and proofs of all the statements that we make in this section.

Fix $d,\;R\in \mathbb{N}$. The voter model on $\mathbb{Z}^d$ with range $R$, denoted by $(\xi_t)_{t \geq 0}$, is the Markov process with state space $\{0,1\}^{\mathbb{Z}^d}$ and infinitesimal generator given by
\begin{equation}\label{infinitesimal_generator}
(\mathcal{L}f)(\xi) = 
\sum_{\substack{x,y \in \mathbb{Z}^d:\\0 < |x - y|_1 \le R}}
\frac{  f(\xi^{y \to x}) - f(\xi)}{|B_1(R)|-1},
\end{equation}
where  $f: \{0,1\}^{\mathbb{Z}^d} \to \mathbb{R}$ is any function that only depends on finitely many coordinates, $\xi \in \{0,1\}^{\mathbb{Z}^d}$ and 
$$
\xi^{y \to x}(z) = 
\left\{ 
\begin{array}{ll}\xi(z)&\text{if } z \neq x, \\
 \xi(y) &\text{if } z = x. 
 \end{array}
 \right.
$$
In words, each site $x \in \Z^d$ updates its state $\xi(x)$ with rate $1$ by uniformly choosing a
 site $y \in B_1(x,R) \setminus \{x\}$ and adopting the state $\xi(y)$ of $y$. 
In case $R = 1$, we say that the model is nearest-neighbour.

Given $\xi \in \{0,1\}^{\mathbb{Z}^d}$, we denote by $P_\xi$ a probability measure under which $(\xi_t)_{t\geq 0}$ is defined and satisfies $P_\xi\left[\xi_0 = \xi\right] = 1$. Likewise, given a probability distribution $\nu$ on $\{0,1\}^{\mathbb{Z}^d}$, we write $P_\nu = \int P_\xi\; \mathsf{d}\nu(\xi)$.

\medskip

The process $(\xi_t)$ satisfies a duality relation with respect to a system of coalescing random walks. We will now explain what is meant by this -- or rather, we will give a particularly simple formulation of duality that will be sufficient for our purposes.

For each $x,y\in \mathbb{Z}^d$ with $0 < |x-y|_1\le R$, let $(D^{(x,y)}_t)_{t \ge 0}$ be a Poisson process with rate 
$(|B_1(R)|-1)^{-1}$ on $[0, \infty)$, so that $D^{(x,y)}_0 =0$ and $D^{(x,y)}_t - D^{(x,y)}_{t-}$ is equal to 0 or 1 for all $t$.
One pictures $D^{(x,y)}_t - D^{(x,y)}_{t-}=1$ as an arrow pointing from $x$ to $y$ at time $t$.
 We denote by $\mathbb{P}$ a probability measure under which all these processes are defined and are independent. 
 For each $x \in \mathbb{Z}^d$, we then define (on this same probability space) $(Y^x_t)_{t\ge 0}$ as the unique $\mathbb{Z}^d$-valued process which is right-continuous with left limits and satisfies
\begin{equation}\label{graphical_walk_poisson}
Y^x_0 = x,\qquad Y^x_t = Y^x_{t-} + \sum_{z \in B_1(R)}z \cdot \left( D^{(Y^x_{t-}, Y^x_{t-}+z)}_t - D^{(Y^x_{t-}, Y^x_{t-}+z)}_{t-}\right).
\end{equation}
One pictures $Y^x_t$ as it moves along the time axis and follows the arrows that it encounters.
The collection of processes $\{(Y^x_t)_{t \ge 0}: x\in \mathbb{Z}^d\}$ is what we refer to as a system of coalescing random walks. 
This terminology makes sense because, as is clear from the above definition, each $(Y^x_t)_{t \geq 0}$ is a continuous-time random walk on $\Z^d$ with rate $1$ which jumps to a uniformly distributed location in 
$Y^x_{t-} + (B_1(R)\setminus \{0\})$ and moreover, these walks move independently until they meet, after which they coalesce and remain together.

Now, for any fixed $\xi \in \{0,1\}^{\mathbb{Z}^d}$, $A \subset \subset \mathbb{Z}^d$ and $t \geq 0$, we have
\begin{equation}\label{eq:first_duality}
\text{law of } \left( \xi(Y^x_t):x \in A\right) \text{ under } \mathbb{P} 
\quad = \quad 
 \text{law of } \left( \xi_t(x): x \in A \right) \text{ under } P_\xi.
\end{equation} 
See  Figure \ref{fig:graphical}  for an illustration.
\begin{figure}[htb]
\begin{center}
\setlength\fboxsep{0pt}
\setlength\fboxrule{0pt}
\fbox{\includegraphics[width = 0.7\textwidth]{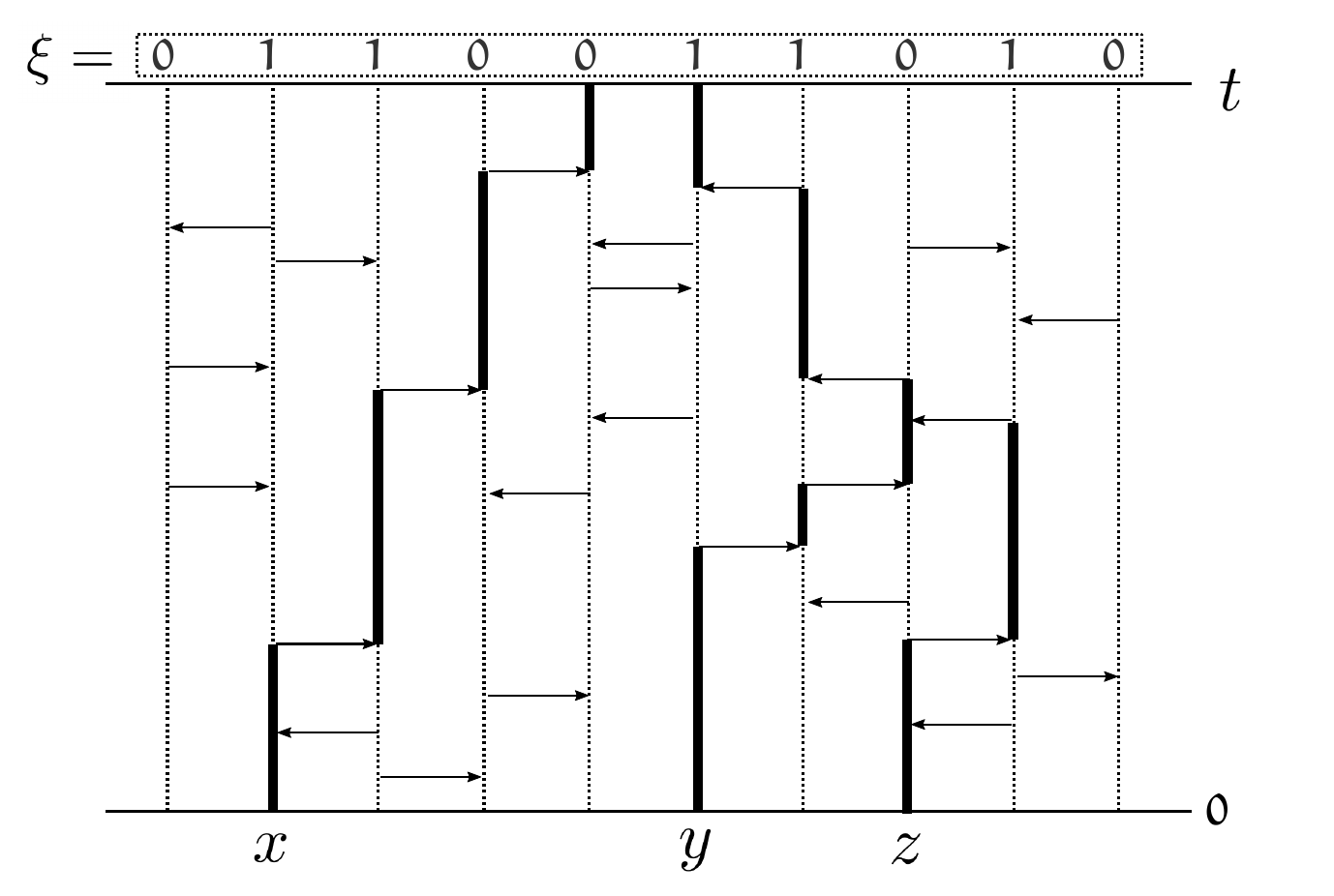}}
\end{center}
\caption{\label{fig:graphical}  Illustration of the system of coalescing random walks to which the voter model is dual. The horizontal axis represents space (which is one-dimensional in this picture) and the vertical axis represents time; arrows are plotted as described in the text. On top, a configuration $\xi \in \{0,1\}^{\mathbb{Z}^d}$ is represented. The thick vertical paths represent the trajectories of $(Y^x_s), (Y^y_s), (Y^z_s)$ for $0 \le s \le t$. In the situation illustrated we have $\xi(Y^x_t) = 0$ and $\xi(Y^y_t) = \xi(Y^z_t) = 1$.}
\end{figure}

As a consequence of \eqref{eq:first_duality}, we obtain the following \textit{duality equation} for the voter model: for any $A \subset \subset \mathbb{Z}^d$, $t \geq 0$ and probability measure $\nu$ on $\{0,1\}^{\mathbb{Z}^d}$,
\begin{equation}
\label{eq:duality} 
P_\nu\left[\xi_t(x) = 1 \text{ for all } x \in A\right] = 
\int \mathbb{P}\left[\{Y^x_t: x \in A\} \subset \{y: \xi(y) = 1\}\right]\, \mathsf{d}\nu(\xi).
\end{equation}
Note that by inclusion-exclusion the equation \eqref{eq:duality}
 characterizes the distribution of $\xi_t$ for the process started with distribution $\nu$. Of particular interest is the case when $\nu$ is equal to
$$\pi_\alpha := ( \alpha \delta_{\{1\}} + (1-\alpha) \delta_{\{0\}})^{\otimes \Z^d},$$
the product measure of $\mathrm{Bernoulli}(\alpha)$ on $\Z^d$, for $\alpha \in [0,1]$. In order to discuss this case, let us introduce some notation. For $A \subset \Z^d$, we let
\begin{equation}
\label{eq:defN}
\mathcal{N}_t(A) = |\{Y^x_t: x \in A\}|, \;\; t \geq 0 \quad \text{ and }
\quad \mathcal{N}_\infty(A) = \lim_{t \to \infty} \mathcal{N}_t(A);
\end{equation}
the limit exists because $\mathcal{N}_t(A)$ decreases with $t$. Denoting by $\mathbb{E}$ the expectation operator associated with $\mathbb{P}$, we can then rewrite (\ref{eq:duality}) as
\begin{equation*}
P_{\pi_\alpha}\left[\xi_t(x) = 1 \text{ for all } x \in A\right] = \mathbb{E}\left[\alpha^{\mathcal{N}_t(A)}\right].
\end{equation*}
By taking the limit on the right-hand side as $t \to \infty$, we can conclude that, under $P_{\pi_\alpha}$,  as $t\to \infty$, $\xi_t$ converges in distribution to a measure $\mu_\alpha$ on $\{0,1\}^{\Z^d}$ characterized by
\begin{equation}
 \label{eq:dualityinf}
 \mu_\alpha\left[\xi(x) = 1 \text{ for all } x \in A\right] =
  \mathbb{E}\left[\alpha^{\mathcal{N}_\infty(A)}\right]
 \end{equation}
for every finite $A \subset \Z^d$. 
The measures $\mu_\alpha$ are invariant and ergodic with respect to translations on $\Z^d$ and satisfy
\begin{equation}\label{eq:stationary_mean_corr}
 \mu_\alpha[\, \xi(x) = 1\, ]= \alpha, \qquad 
 \mathrm{Corr}_{\mu_{\alpha}} \left( \xi(x), \xi(y) \right)\stackrel{ \eqref{eq:def_spread_out_hitting_prob} }{=}
 h_R(x,y), 
 \qquad x,y \in \Z^d,
 \end{equation}
thus \eqref{eq:intro_corr} indeed holds by \eqref{green_bounds} and \eqref{eq:meet_green}.
We also note that
\begin{equation}\label{fkg}
  \mu_{\alpha} \left[\xi(x) = 1 \text{ for all } x \in A \right] 
\stackrel{\eqref{eq:defN},\eqref{eq:dualityinf}}{\geq} \mathbb{E}\left[\alpha^{\mathcal{N}_0(A)}\right]
 =\alpha^{|A|}. 
\end{equation}

As the measures $\mu_\alpha$ are obtained as distributional limits of $(\xi_t)$, they are also stationary with respect to the dynamics of the voter model. In fact, in Section V.1 of \cite{lig85} it is shown that
\begin{itemize}
\item if $d \ge 3$, then the set of extremal stationary distributions of the voter model is equal to $\{\mu_\alpha: \alpha \in [0,1]\}$. Here, a measure is said to be extremal if it cannot be written as a nontrivial convex combination of other stationary distributions.
\item if $d = 1$ or $2$, then there are only two extremal stationary distributions, namely the point masses on the constant configurations $\xi \equiv 1$ and $\xi \equiv 0$. (If $d=1$ or $2$, by recurrence of the random walk we have $\mathcal{N}_\infty(A) = 1$ almost surely for any finite and non-empty $A$. We can then see from (\ref{eq:dualityinf}) that $\mu_\alpha$ is a convex combination with weight $\alpha$ of the point masses on the constant configurations).
\end{itemize}

Finally, we give a useful construction, jointly on the same probability space, of the system of coalescing random walks and for each $\alpha \in [0,1]$, a random $\xi^{(\alpha)} \in \{0,1\}^{\Z^d}$ distributed as $\mu_\alpha$.
 To this end, we take the probability space in which the aforementioned measure 
 $\mathbb{P}$ and the processes $\left( (Y^x_t)_{t \geq 0}:x \in \Z^d \right)$ are defined,
  and enlarge it so that a sequence of random variables $\mathcal{U}_n, \, n \in \mathbb{N}$, all independent and 
uniformly distributed on $[0,1]$, are also defined (and are independent of the $Y^x_t$'s).
 Next, fix an arbitrary enumeration $x_1, x_2, \ldots$ of $\Z^d$. For any $n \geq 1$, define the random variables
\begin{equation}
\label{eq:jointc1}
\eta(n) = \min\{m \; :\; m \leq n \text{ and }
  Y^{x_m}_t = Y^{x_n}_t \text{ for some } t \ge 0\}
  \end{equation}
and then set 
\begin{equation}
\label{eq:jointc2}
\xi^{(\alpha)}(x_n)= 
\mathds{1}_{\{ \mathcal{U}_{\eta(n)}\leq \alpha\}},\qquad n \in \mathbb{N},\;\alpha \in [0,1].
\end{equation}
It is then straightforward to check that $\xi^{(\alpha)}$ has law $\mu_\alpha$, as defined in (\ref{eq:dualityinf}), and moreover it satisfies
\begin{equation*}
\xi^{(\alpha)}(x)= \xi^{(\alpha)}(y) 
\quad \text{ if } \quad
Y^x_t = Y^y_t \text{ for some } t. 
\end{equation*}
Moreover, it follows from this construction that if $\alpha \leq \alpha'$, then 
$\xi^{(\alpha)}(x) \leq \xi^{(\alpha')}(x)$ for each $x \in \Z^d$, therefore
 $\mu_\alpha$ is stochastically dominated by $\mu_{\alpha'}$, that is, if $f:\{0,1\}^{\Z^d} \to \mathbb{R}$
 is increasing (with respect to the partial order on $\{0,1\}^{\Z^d}$ that is induced by the order $0<1$ on the coordinates), then 
\begin{equation}
\label{eq:stoch_mon} 
 \int f \;\mathrm{d}\mu_\alpha \leq \int f \;\mathrm{d}\mu_{\alpha'}.
 \end{equation}
 We will also need the following consequence of the joint construction:
\begin{equation}
\label{eq:one_minus}
\text{ the law of $1-\xi$ under $\mu_\alpha$ is the same as law of $\xi$ under $\mu_{1-\alpha}$.}
\end{equation}

\section{First facts about voter model percolation, $\boldsymbol{d\ge 3}$}
\label{s:first}

In this section we collect the definitions and facts that are common to the proofs of Theorems  \ref{thm:nearest_neighbour} and \ref{thm:spread_out}.
Throughout this section, we fix $d \geq 3$, $R \geq 1$ (see \eqref{infinitesimal_generator}) and $\alpha \in [0,1]$.
$\xi$ denotes an element of $\{0,1\}^{\Z^d}$ and $\mu_\alpha$ denotes the extremal stationary distribution of the voter model with density $\alpha$, as described in Section \ref{section_voter_graphical_construction}.

In Section \ref{subsection:sufficient} we  state the key inequality \eqref{eq:noperc} and deduce
Theorems  \ref{thm:nearest_neighbour} and \ref{thm:spread_out} from it. 
In Section \ref{subsection:renorm} we set up the multi-scale renormalization scheme that
we will employ to prove \eqref{eq:noperc}.

\subsection{A sufficient condition for percolation phase transition}
\label{subsection:sufficient}

 We will show that there exist $\alpha_0 > 0$ and  a sequence $(L_N)_{N \geq 0}$ of form
\begin{equation}\label{eq:form_of_L_N} L_N= L \cdot \ell^N, \end{equation} 
where $\ell \geq 6$, $L \geq 1$ such that
\begin{equation} 
\label{eq:noperc} 
 \mu_{\alpha_0}\left[B(L_N - 2) \stackrel{*\xi}{\longleftrightarrow} B(2L_N)^c\right] \leq 2^{-2^N}, \qquad N \geq 1.
\end{equation}
In words: the probability under $\mu_{\alpha_0}$ that an annulus with inner radius $L_N-2$ and outer radius $2 L_N$ is crossed by a $*$-connected path of 1's in $\xi$ is less than or equal to $2^{-2^N}$. 
Note that \eqref{eq:noperc} implies that the crossing probability of the annulus $B(2M)\setminus B(M)$
 decays as a stretched exponential function of $M$ as $M \to \infty$:
\[\mu_{\alpha_0}\left[B(M) \stackrel{*\xi}{\longleftrightarrow} B(2M)^c\right] \leq C e^{- M^{\kappa}}
 \quad \text{ for some } C<+\infty \text{ and } \kappa>0.  \]
We will prove \eqref{eq:noperc} for  $d \geq 3$ and $R \gg 1$
 in Section \ref{section_spread_out}  and for $d \geq 5$ and $R = 1$ in Section \ref{section_nearest_neighbour}.
Let us now deduce the main results of this paper from \eqref{eq:noperc}.

\begin{proof}[Proof of Theorems  \ref{thm:nearest_neighbour} and \ref{thm:spread_out}]

As we have already discussed in Section \ref{section_voter_graphical_construction}, the measure $\mu_{\alpha}$ is invariant and ergodic under
 spatial shifts of $\Z^d$. Therefore the probability under $\mu_\alpha$ of the event
\begin{equation}\label{eq:def_perc_event}
\mathrm{Perc} = \left\{ \; \{x:\xi(x) = 1\} \text{ has an infinite connected component} \; \right\}
\end{equation} 
 can only be zero or one for any $\alpha$. 
 Also, since the event in \eqref{eq:def_perc_event} is increasing, by \eqref{eq:stoch_mon} there indeed exists 
 $0 \leq \alpha_c \leq 1$ such that $\mu_\alpha\left( \mathrm{Perc} \right) =0$ for any $\alpha<\alpha_c$ and
 $\mu_\alpha\left( \mathrm{Perc} \right) =1$ for any $\alpha>\alpha_c$. 
 Our aim is to to prove that $0<\alpha_c<1$.

Let us now explain how \eqref{eq:noperc} implies
\[\alpha_0 \leq \alpha_c \leq 1-\alpha_0.\] 
As soon as we prove these inequalities, the statements  of 
Theorems  \ref{thm:nearest_neighbour} and \ref{thm:spread_out} will follow.

\medskip

First,  we will prove $\alpha_0 \leq \alpha_c$ by showing that $\mu_{\alpha_0}\left( \mathrm{Perc} \right)=0$.
Denote by $\{ x \stackrel{\xi}{\longleftrightarrow} \infty \}$ the event that
there exists a nearest-neighbor path $\gamma(0),\gamma(1),\ldots,$ such that $\gamma(0)=x$, 
$\lim_{n\to \infty} |\gamma(n)|=\infty$ and $\xi(\gamma(n)) = 1$ for each $n$.
Since 
\[\mathrm{Perc}= \bigcup_{x \in \Z^d} \{ x \stackrel{\xi}{\longleftrightarrow} \infty \} \]
and $\mu_{\alpha_0}$ is invariant under translations of $\Z^d$, it is enough to prove that
$\mu_{\alpha_0}[ 0 \stackrel{\xi}{\longleftrightarrow} \infty  ]=0$. 
This follows from  \eqref{eq:noperc} and the inclusions 
\[\{ 0 \stackrel{\xi}{\longleftrightarrow} \infty \} \subseteq
 \left\{ B(L_N - 2) \stackrel{\xi}{\longleftrightarrow} B(2L_N)^c \right\}
 \stackrel{(*)}{\subseteq}
 \left\{ B(L_N - 2) \stackrel{*\xi}{\longleftrightarrow} B(2L_N)^c \right\}, \quad N \geq 1,
  \]
 where $(*)$ holds since a nearest-neighbour path is also a $*$-path, see Definition \ref{def_connections_crossings}.

\medskip

Now we will prove $\alpha_c \leq 1-\alpha_0$ using a variant of the classical Peierls argument, see
 \cite{Pei36} and \cite[Section 1.4]{Gr99}.
Let us define the plane $\mathcal{P} \subset \mathbb{Z}^d$ by
$$\mathcal{P} = \{\; x = (x_1,\ldots, x_d) \in \mathbb{Z}^d: x_i = 0 \text{ for all } i \geq 3 \; \}.$$

For $\xi \in \{0,1\}^{\mathbb{Z}^d}$, we denote by $\bar{\xi}$ the restriction of $\xi$ to $\mathcal{P}$.
For any $N \geq 1$ we define the events
\begin{align*}
E_N &= \{ 
\text{ $B(L_N) \cap \mathcal{P}$ is not connected to $\infty$ by a nearest neighbour path of
 1's in $\bar{\xi}$ } \},\\
F_N &= \{ 
\text{ $B(L_N) \cap \mathcal{P}$ is surrounded by a $*$-connected 
cycle $\widetilde{\gamma}$ of 0's in $\bar{\xi}$ } \}.
\end{align*}

By planar duality (see  Definition 4, Definition 7 and Corollary 2.2 of \cite{Kesten82}), 
we have 
\[E_N=F_N.\]
If $F_N$ occurs, denote by $\widetilde{M}$ the smallest integer such that 
$B(L_{\widetilde{M}+1})\cap \{ \widetilde{\gamma}\} \neq \emptyset$.
 By the definition of $F_N$ we have $\widetilde{M} \geq N$.
 
  If $F_N$ occurs, we can pick an
 $\widetilde{x} \in ( L_{\widetilde{M}} \cdot \Z^d )\cap \mathcal{P}$ satisfying 
 $|\widetilde{x}| \leq  L_{\widetilde{M}+1}$ such that
 $B(\widetilde{x},L_{\widetilde{M}} - 1) \cap \{ \widetilde{\gamma} \} \neq \emptyset$. By the definition of $\widetilde{M}$ 
 the cycle $\widetilde{\gamma}$
  surrounds $B(L_{\widetilde{M}}) \cap \mathcal{P}$, 
 therefore by Definition \ref{def_connections_crossings}\eqref{crossing_star_neighbour} the annulus $B(\widetilde{x},2L_{\widetilde{M}})\setminus B(\widetilde{x},L_{\widetilde{M}}-2)$ 
 is crossed by $ \widetilde{ \gamma }$.
  Thus for any $N \in \N$ we can bound
\begin{align*}
\mu_{1-\alpha_0} [ \mathrm{Perc}^c] \leq  \mu_{1-\alpha_0}[E_N]&=
 \mu_{1-\alpha_0}[F_N]  \\&\leq
\sum_{M=N}^{\infty} \sum_{\substack{ x \in  L_M \cdot \Z^d \\ |x| \leq  L_{M+1} }}
 \mu_{1-\alpha_0}\left[B(x,L_M-2) \stackrel{*(1-\xi)}{\longleftrightarrow} B(x,2L_M)^c\right]
\\[.2cm]&\stackrel{ \eqref{eq:one_minus},\eqref{eq:form_of_L_N},\eqref{eq:noperc} }{\leq} 
\sum_{M=N}^{\infty} (2 \ell +1)^d \cdot 2^{-2^M},
\end{align*} 
from which $\mu_{1-\alpha_0}[\mathrm{Perc}]=1$ follows by letting $N \to \infty$.
This implies $\alpha_c \leq 1-\alpha_0$.
The proof of Theorems  \ref{thm:nearest_neighbour} and \ref{thm:spread_out} is complete, given \eqref{eq:noperc}.
\end{proof}

\subsection{Renormalization scheme for percolation, $\boldsymbol{d \ge 3}$}
\label{subsection:renorm}

We are going to use multi-scale renormalization.
Similar methods have been successfully employed to prove the percolation phase transition
of the vacant set of random interlacements (see \cite{SidoraviciusSznitman_RI, sznitman_interlacement})
and the excursion set of the Gaussian free field (see \cite{RS13}).
We will borrow the renormalization scheme of \cite{ra04}, which is in turn a variant of the 
method developed in Sections 2 and 3 of \cite{sznitman_decoupling}.

\medskip

Let us fix $d \geq 3$. We let $\ell$ and $L$ be two integers describing the scales of renormalization:
\begin{equation}\label{eq:renorm_scales_L_N}
 L_N= L \cdot \ell^N, \qquad N\geq 0.
\end{equation}
Using these scales we define the renormalized lattices
\begin{equation}
\label{eq:def_renom_lattice}
\mathcal{L}_N = L_N\cdot \Z^d,\qquad N \geq 0.
\end{equation}

\begin{remark}\label{remark_renormalization}
The basic idea behind the proof of \eqref{eq:noperc} is as follows. Denote by $p(N)$ the probability of the crossing event
 that appears on the left-hand side of \eqref{eq:noperc}. The crossing of an annulus of scale $L_N$ implies that two
  annuli of scale $L_{N-1}$ that are far enough from each other
   are also crossed (see Figure \ref{fig:renorm} below), so one naively hopes to
 upper bound $p(N)$ in terms of $p(N-1)^2$ and thus prove \eqref{eq:noperc} by induction on $N$.
  To make this idea rigorous, one needs to take into account the combinatorial
    term that counts the number of choices of the smaller annuli, and, more importantly, the strong positive correlation between
 the two crossing events on the smaller scale.

  We start our proof of \eqref{eq:noperc}  by repeating the above sketched renormalization step
 until we reach the bottom scale $L_0$.
 We encode the choices of the centers of these annulli as
 embeddings of the binary tree $T_N$ of depth $N$
into $\Z^d$ (see Definition \ref{def_proper_embedding_of_trees}) -- this way the proof of \eqref{eq:noperc}
boils down to bounding the probability of the joint occurrence of $2^N$ instances of a simple bottom-level event,
indexed by the leaves of $T_N$ (see Lemma \ref{lem:embpath}).
 \end{remark}

 Let $T_{(k)} = \{1,2\}^k$ for $k \geq 0$ (in particular, $T_{(0)} = \{\emptyset\}$) and then let
$$T_N = \cup_{k=0}^N T_{(k)}$$
be the binary tree of height $N$. If $0 \leq k < N$ and $m = (\eta_1, \ldots, \eta_k) \in T_{(k)}$, we let
\begin{equation}\label{children_of_m}
m_1 = (\eta_1,\ldots,\eta_k,1),\qquad m_2 = (\eta_1,\ldots,\eta_k,2)
\end{equation}
be the two children of $m$ in $T_{(k+1)}$.

\begin{definition}\label{def_proper_embedding_of_trees}
  $\mathcal{T}: T_N \to \Z^d$ is a proper embedding of  $T_N$ if

\begin{enumerate}
 \item $\mathcal{T}(\{\emptyset\})=0$;
\item for all $0 \leq k \leq N$ and $m \in T_{(k)}$ we have $\mathcal{T}(m) \in \mathcal{L}_{N-k}$;
\item for all $0 \leq k < N$ and $m \in T_{(k)}$ we have
\begin{equation}
\label{tree_children_spread_out}
|\mathcal{T}(m_1) -\mathcal{T}(m)|= L_{N-k}, \qquad |\mathcal{T}(m_2) -\mathcal{T}(m)|= 2 L_{N-k}.
\end{equation}
\end{enumerate}
We denote by $\Lambda_{N}$ the set of proper embeddings of  $T_N$ into $\Z^d$.
\end{definition}

We now collect a few facts from \cite{ra04} about these embeddings. 
Although the lemmas in \cite{ra04} that correspond to our 
Lemmas \ref{lemma:number_of_proper_embeddings}, \ref{lem:embpath} and \ref{lem:embsp} below
are stated for $\ell = 6$, their statements hold true 
 (and have the same proof) for any integer $\ell \geq 6$.

\begin{lemma}\label{lemma:number_of_proper_embeddings}
\begin{equation}
\label{eq:countemb_general}
|\Lambda_N| = \left[   (  (4\ell+1)^d-(4\ell-1)^d) \cdot ((2\ell + 1)^d - (2\ell-1)^d)\right]^{2^N-1}.
\end{equation}
\end{lemma}
\noindent This follows from \cite[Lemma 3.2]{ra04}. Informally, given $\mathcal{T}(m)$, there are
$(2\ell+1)^d-(2\ell-1)^d$ ways to choose $\mathcal{T}(m_1)$ and 
$(4\ell+1)^d-(4\ell-1)^d$ ways to choose $\mathcal{T}(m_2)$.

\smallskip

Next is the statement that, given a crossing of the $L_N$-scale annulus $B(2L_N) \backslash B(L_N)$, we can find a proper embedding $\mathcal{T}\in \Lambda_N$ so that all $L_0$-scale annuli 
$B(\mathcal{T}(m),2L_0) \setminus B(\mathcal{T}(m), L_0): m \in T_{(N)}$ are crossed. 
Recall the notion of $S(x,L)$ from \eqref{eq:def_balls}.

\begin{figure}[htb]
\begin{center}
\setlength\fboxsep{0pt}
\setlength\fboxrule{0pt}
\fbox{\includegraphics[width = 1\textwidth]{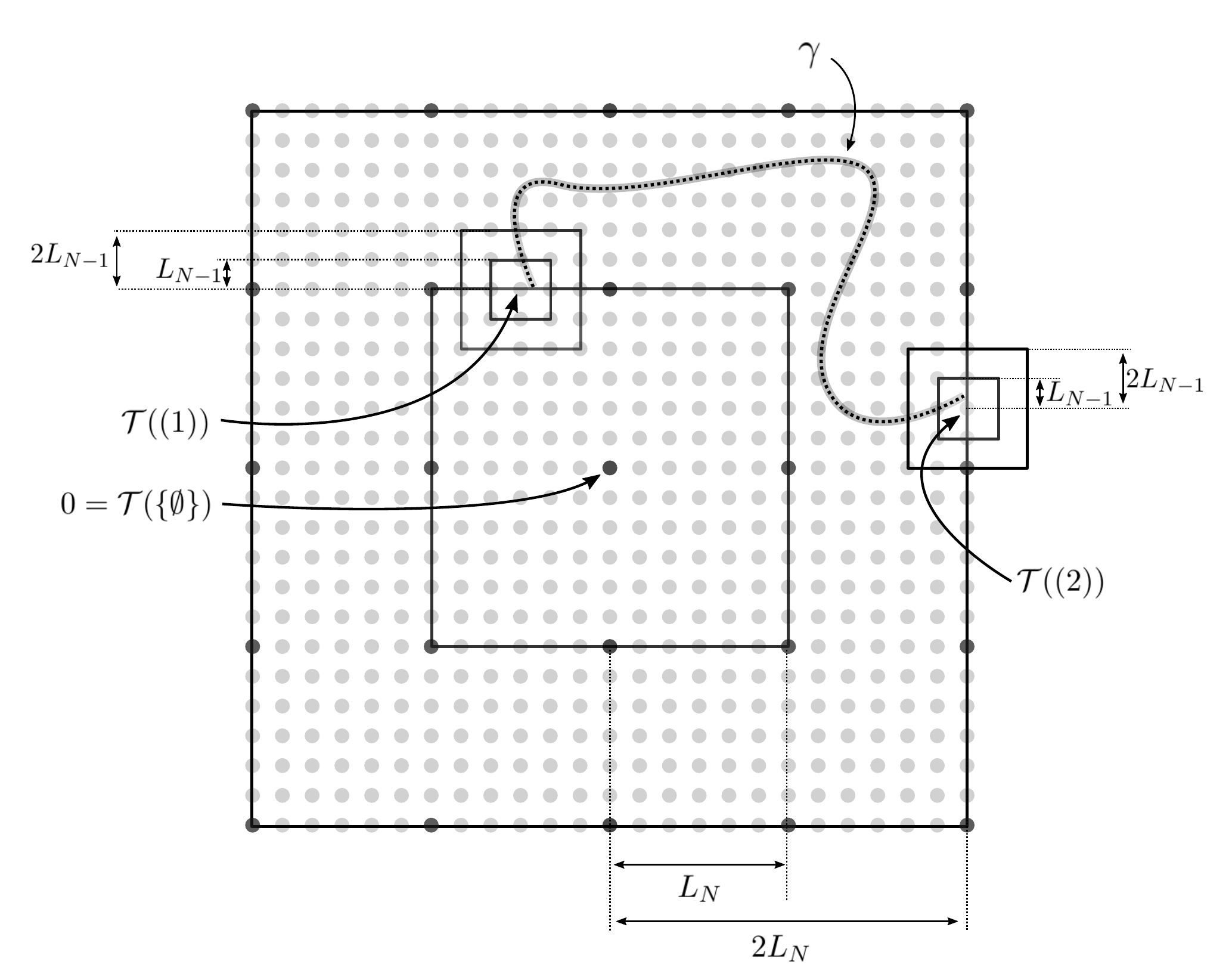}}
\end{center}
\caption{\label{fig:renorm} Illustration of the relation between the proper embedding $\mathcal{T}$ 
(see Definition \ref{def_proper_embedding_of_trees}) and 
the path $\gamma$ that appears in Lemma \ref{lem:embpath}. The light grey circles and dark grey circles represent points of the lattices $\mathcal{L}_{N-1}$ and $\mathcal{L}_N$, respectively. 
}\end{figure}

\begin{lemma}
\label{lem:embpath}
If $\gamma$ is a $*$-connected path in $\Z^d$ with
\begin{equation*}
\{\gamma\} \cap S(L_N-1) \neq \varnothing,\quad \{\gamma\} \cap S(2L_N) \neq \varnothing,
\end{equation*}
then there exists $\mathcal{T}\in \Lambda_N$ such that
\begin{equation}
\label{eq:rencons}
\{\gamma\} \cap S(\mathcal{T}(m), L_0-1) \neq \varnothing \text{ and }
 \{\gamma\} \cap S(\mathcal{T}(m),2L_0) \neq \varnothing \quad \text{ for all } m \in T_{(N)}.
\end{equation}
\end{lemma}
\noindent This is \cite[Lemma 3.3]{ra04} (in fact, the statement given here corresponds to equation 
(3.7) in the proof of that lemma). Informally, one recursively constructs a proper embedding
$\mathcal{T}$: if $\gamma$ crosses an annulus of scale $L_{N-k}$ centered at some
$\mathcal{T}(m) \in \mathcal{L}_{N-k}$ for
$m \in T_{(k)}$, $0 \leq k < N$, then  two ``children'' annuli of
scale $L_{N-k-1}$ centered at some $\mathcal{T}(m_1), \mathcal{T}(m_2)\in \mathcal{L}_{N-k-1}$
satisfying \eqref{tree_children_spread_out} will also be crossed by $\gamma$, see Figure \ref{fig:renorm}.

\smallskip

Finally,  given a proper embedding $\mathcal{T} \in \Lambda_N$, the set of images of the leaves $\{\mathcal{T}(m): m \in T_{(N)}\}$ is ``spread-out on all scales''.
\begin{lemma}
\label{lem:embsp}
For any $\mathcal{T}\in \Lambda_N$ and any $m_0 \in T_{(N)}$, we have
\begin{equation}
\label{eq:embsp}
\left|\left\{ m \in T_{(N)}: \mathrm{dist}\left( B(\mathcal{T}(m_0), 2L),\;B(\mathcal{T}(m), 2L) \right) \leq \ell^k L/2 \right\} \right| \leq 2^{k-1},\quad k\geq 1.
\end{equation}
\end{lemma}
\noindent This is a consequence of our assumption $\ell \geq 6$ and  \cite[Lemma 3.4]{ra04}.
 In particular, by  choosing $k=1$ in \eqref{eq:embsp} we obtain that the sets $B(\mathcal{T}(m), 2L)$ for $m \in T_{(N)}$ are disjoint.


\section{ Spread-out model, $\boldsymbol{d\ge 3}$}
\label{section_spread_out}

In this section we work with the voter model with range $R$, thus we will denote the stationary distribution (see \eqref{eq:dualityinf})  with density $\alpha$ by $\mu_{R,\alpha}$.
The goal of this section is to prove Theorem \ref{thm:spread_out}. More specifically, 
we will show that \eqref{eq:noperc} holds for any $d \geq 3$ if $R \geq R_0(d)$ for some large $R_0$ and
 some $\alpha_0=\alpha_0(d)>0$. 

\medskip

Recall the notion of $h_R(x,y)$ from \eqref{eq:def_spread_out_hitting_prob}.
 The key result in our proof of Theorem \ref{thm:spread_out} is the following decorrelation inequality
 which serves as a partial converse to \eqref{fkg}: 
 
\begin{lemma}\label{lemma:bound_on_occupied_using_annih}
For any $\mathcal{X}=\{x_1,\dots,x_{|\mathcal{X}|} \} \subset \Z^d $ we have
\begin{equation}\label{eq:bound_on_occupied_using_annih}
 \mu_{R,\alpha} \left[\xi(x) = 1 \text{ for all } x \in \mathcal{X} \right] \leq
 \alpha^{|\mathcal{X}|} 
 \prod_{1 \leq i < j \leq |\mathcal{X}|} \left( 1+ h_R(x_i,x_j)\left( \alpha^{-2}-1 \right)  \right).
 \end{equation}
\end{lemma}

Before we prove Lemma \ref{lemma:bound_on_occupied_using_annih}, let us see how it allows us to conclude.

\begin{proof}[Proof of \eqref{eq:noperc} for $d \geq 3$ and $R \gg 1$]
We use the renormalization scheme described in Section \ref{subsection:renorm}.
In this proof we choose $\ell=6$ and $L=1$ in \eqref{eq:renorm_scales_L_N}.
Given $\mathcal{T} \in \Lambda_N$, we denote
\begin{equation}\label{def_eq_tree_leaf_embedded_set}
\mathcal{X}_{\mathcal{T}} = \bigcup_{ m \in T_{(N)} } \{ \mathcal{T}(m)\} \stackrel{\eqref{eq:def_balls}}{=}
\bigcup_{ m \in T_{(N)} } S(\mathcal{T}(m), L_0-1).
\end{equation}
By Lemma \ref{lemma:number_of_proper_embeddings}, we have
\[ |\Lambda_N| \leq \widehat{C}^{2^N}, \text{ where } 
\widehat{C}=  (  (4\cdot 6+1)^d-(4\cdot 6-1)^d) \cdot ((2\cdot 6 + 1)^d - (2\cdot6-1)^d). \]
Combining Definition \ref{def_connections_crossings} and Lemma \ref{lem:embpath} in a union bound, we get, for any $N$,
\begin{equation}
\label{eq:binary_union_bound_spread_out} 
 \mu_{R,\alpha} \left[ B(L_N-2) \stackrel{*\xi}{\longleftrightarrow} B(2L_N)^c \right] \leq 
\widehat{C}^{2^N} \max_{ \mathcal{T} \in \Lambda_{N} } 
 \mu_{R,\alpha} \left[ \xi(x) = 1 \text{ for all } x \in \mathcal{X}_{\mathcal{T}} \right].
\end{equation}

Now we fix some $N$ and $\mathcal{T} \in \Lambda_N$ with the aim of bounding the probability on the right-hand side of
\eqref{eq:binary_union_bound_spread_out}. Note that by Lemma \ref{lem:embsp} we have 
$|\mathcal{X}_{\mathcal{T}}|=2^N$. Let us denote 
$\mathcal{X}_{\mathcal{T}}=\{x_1,\dots,x_{2^N} \} $. We have
\begin{multline}\label{eq:spread_out_occupied_bound_exp_appears}
\mu_{R,\alpha} \left[ \xi(x) = 1 \text{ for all } x \in \mathcal{X}_{\mathcal{T}} \right]
\stackrel{ \eqref{eq:bound_on_occupied_using_annih} }{\leq}\\
  \alpha^{2^N} 
 \prod_{1 \leq i < j \leq 2^N} 
 \left( 1+ h_R(x_i,x_j) \alpha^{-2}  \right) \leq
  \alpha^{2^N} \prod_{i=1}^{2^N-1} \exp\left( \alpha^{-2} \sum_{j > i} h_R(x_i,x_j) \right).
\end{multline}
For any $1\leq i < 2^N$ let us bound
\begin{equation}\label{eq:geo_spread_out}
\sum_{j > i} h_R(x_i,x_j)
\stackrel{\eqref{spread_out_hitting_bound}}{\leq}
\sum_{j > i} f(R) \cdot |x_i-x_j|^{2-d} 
\stackrel{ \eqref{eq:embsp} }{\leq}
\sum_{k=1}^{\infty} (2^{k-1}-1) \cdot \left( f(R)\cdot \left( 6^k /2 \right)^{2-d} \right).
\end{equation}
By $\lim_{R \to \infty} f(R)=0$  (see \eqref{spread_out_hitting_bound}) and 
\eqref{eq:geo_spread_out}, for any $\alpha>0$
 we can choose $R=R(\alpha)$ big enough so that for any $1\leq i < 2^N$ we have
\begin{equation} \label{spread_out_enough}
 \alpha^{-2} \sum_{j > i} h_R(x_i,x_j) \leq \ln(2).
 \end{equation}
Letting $\alpha_0=\frac14 \widehat{C}^{-1}$ 
 and $R=R(\alpha_0)$ 
we obtain the desired  \eqref{eq:noperc}:
\begin{multline*}
\mu_{R,{\alpha_0}} \left[ B(L_N-2) \stackrel{*\xi}{\longleftrightarrow} B(2L_N)^c \right]
\stackrel{ \eqref{eq:binary_union_bound_spread_out},\eqref{eq:spread_out_occupied_bound_exp_appears} }{\leq}
\widehat{C}^{2^N}  \alpha_0^{2^N} \prod_{i=1}^{2^N-1} \exp\left( \alpha_0^{-2} \sum_{j > i} h_R(x_i,x_j) \right)\\
\stackrel{ \eqref{spread_out_enough} }{\leq}
\widehat{C}^{2^N} \left(\frac14 \widehat{C}^{-1}\right)^{2^N} 2^{2^N}= 2^{-2^N}.
\end{multline*}
\end{proof}

The rest of this section is devoted to the proof of Lemma \ref{lemma:bound_on_occupied_using_annih}.

Recall the graphical construction of coalescing random walks $Y^x_t$, $x \in \Z^d$, $t \in \R_+$
defined on the probability space  of the Poisson point processes $(D^{(x,y)}_t)_{t \ge 0}$ from Section \ref{section_voter_graphical_construction}.
 Given $\mathcal{X} \subset \subset \Z^d$, define $\mathcal{X}_t=\{Y^x_t: x \in \mathcal{X} \}$, so that
$\mathcal{N}_t(\mathcal{X})=|\mathcal{X}_t|$. 
If $D^{(x,y)}_t - D^{(x,y)}_{t-}=1$ for some $x \in \mathcal{X}_{t-}$, $y \in \Z^d$ and $t \in \R_+$, then the
 graphical construction
 \eqref{graphical_walk_poisson} of coalescing random walks implies
 \begin{equation}\label{eq:coal_set_evol_graphical}
  \mathcal{X}_t= \left( \mathcal{X}_{t-} \setminus \{ x \} \right) \cup \{ y\}.
  \end{equation}

Let us introduce another set-valued stochastic process $\mathcal{X}'_t$, \emph{annihilating random walks}, 
 also defined on the probability space  of the Poisson point processes $(D^{(x,y)}_t)_{t \ge 0}$.
  Starting also from $\mathcal{X}'_0:=\mathcal{X}$ these particles also perform
independent $R$-spread-out continuous-time random walks until one of the walkers
 tries to jump on a site occupied by another walker, in which case both of them disappear immediately.
 The formal definition is as follows.
 If $D^{(x,y)}_t - D^{(x,y)}_{t-}=1$ for some $x \in \mathcal{X}'_{t-}$, $y \in \Z^d$ and $t \in \R_+$, then
 \begin{equation}\label{eq:annih_graphic_def}
  \mathcal{X}'_t= \left( \mathcal{X}'_{t-} \setminus \{ x \} \right) \Delta \{ y\},
  \end{equation}
where $A \Delta B$ denotes the symmetric difference of the sets $A$ and $B$.
 
Similarly to \eqref{eq:defN},  let us denote $\mathcal{N}'_t(\mathcal{X})=|\mathcal{X}'_t|$ and  $\mathcal{N}'_\infty(\mathcal{X})=\lim_{t \to \infty} \mathcal{N}'_t(\mathcal{X})$.

\begin{remark}\label{remark_annihilating}
  Annihilating random walks were introduced in \cite{EN74} (in a discrete-time version) and studied in the 70s and 80s; see for instance \cite{Sc76}, \cite{Gr78}, \cite{BG80} and \cite{Ar83}. We also mention that, as explained in Example 4.16 in Chapter III of \cite{lig85}, there is a duality relation between the voter model and annihilating random walks, which is of a different nature from the duality between the voter model and coalescing random walks. However, our use of annihilating walks is unrelated to this duality.
  
Our intuitive reason for switching from coalescing to annihilating walks is
the following: $\mathbb{E}\left[\alpha^{\mathcal{N}'_\infty( \mathcal{X} )}\right]$ is easier to bound than 
$\mathbb{E}\left[\alpha^{\mathcal{N}_\infty( \mathcal{X} )}\right]$, because in the case of coalescing walks, one ``ill-behaved'' walker can ``run around'' and cause many coalescence
events, but in the case of annihilating random walks, an ``ill-behaved'' walker will self-destruct at the moment of the first collision.  
\end{remark}

 \begin{lemma} \label{lemma:coalescing_dominates_annihilating}
For any $\mathcal{X} \subset \subset \Z^d$, $\alpha \in [0,1]$, $R \in \N$ and $t \geq 0$ we have
  \begin{equation}\label{eq:coalescing_dominates_annihilating}
  \mathbb{E}\left[\alpha^{\mathcal{N}_t( \mathcal{X} )}\right] \leq
   \mathbb{E}\left[\alpha^{\mathcal{N}'_t( \mathcal{X} )}\right].
  \end{equation}
 \end{lemma}
 \begin{proof}
 As soon as we show $\mathcal{X}'_t \subseteq \mathcal{X}_t$, the inequality \eqref{eq:coalescing_dominates_annihilating} will immediately follow.

Let us assume that $D^{(x,y)}_t - D^{(x,y)}_{t-}=1$ for some $x \in \mathcal{X}_{t-}$
and that $\mathcal{X}'_{t-} \subseteq \mathcal{X}_{t-}$ holds.
One can readily check using \eqref{eq:coal_set_evol_graphical} and \eqref{eq:annih_graphic_def}
 that we also have $\mathcal{X}'_{t} \subseteq \mathcal{X}_{t}$
by considering the cases 
\[
\text{(a) $x \in \mathcal{X}_{t-} \setminus \mathcal{X}'_{t-}$,} \quad 
\text{(b) $x \in \mathcal{X}'_{t-},\, y \notin \mathcal{X}'_{t-}$,} \quad
\text{(c) $x \in \mathcal{X}'_{t-},\, y \in \mathcal{X}'_{t-}$}
\] separately.
Since $\mathcal{X}_0=\mathcal{X}'_0=\mathcal{X}$, the inclusion $\mathcal{X}'_t \subseteq \mathcal{X}_t$ for all $t \in \R_+$
follows by induction.
 \end{proof}

Let us now give an alternative construction of $\mathcal{X}'_t$ on a different probability space.
Recall the notation $\mathcal{X}=\{x_1,\dots,x_{|\mathcal{X}|} \}$.
 Let $X^i_t$, $1\leq i \leq |\mathcal{X}|$ denote independent $R$-spread-out random walks with $X^i_0=x_i$.
 For $1\leq i<j \leq  |\mathcal{X}|$, we denote 
 \begin{equation}\label{eq:spread_out_meeting_time}
 \tau(i,j) = \inf\{\, t \, : \, X^i_t=X^j_t \, \}.
  \end{equation}
We also define the set-valued stochastic process 
$\mathcal{I}_t \subseteq \left[ 1,\dots,|\mathcal{X}| \right]$ and the stopping times $T_0, T_1, T_2,\dots$ 
 by letting $T_0=0$ and
$\mathcal{I}_{T_0} = \left[ 1,\dots,|\mathcal{X}| \right]$, and then inductively for $k\geq 1$ by
\begin{equation*}
T_{k}:= \inf\{ \, \tau(i,j) \; : \; i<j, \;\;\; i,j \in \mathcal{I}_{T_{k-1}} \, \}, \quad
T_{k}=\tau(i^*,j^*), \quad
 \mathcal{I}_{T_{k}}=\mathcal{I}_{T_{k-1}} \setminus \{i^*,j^*\}.
\end{equation*}
 In words, $T_{k}$ is the time of the $k$'th annihilation and $\mathcal{I}_{T_{k}}$ is the set of indices of those
 walkers that are still alive after the $k$'th annihilation.
  Of course if $T_{k}=+\infty$ for some $k \geq 1$ then we stop our inductive definition.
  We define 
 $\mathcal{I}_t=\mathcal{I}_{T_{k-1}}$ for any $T_{k-1} \leq t < T_{k}$. 
 
\begin{claim} 
 The set-valued process
  $\mathcal{X}'_t=\{ X^i_t \; : \; i \in \mathcal{I}_t \}$  has the same law as 
  the annihilating walks described in \eqref{eq:annih_graphic_def}. 
\end{claim}  
The proof of this claim is straightforward and we omit it.
  From now on we will use this new definition of annihilating walks. For $1 \leq i < j \leq |\mathcal{X}|$ we also define the indicators
  \begin{equation}\label{eq:def_annih_indicator_i_j}
   \eta_{i,j} = \mathds{1} \left[ \, \tau(i,j) <+\infty, \;  \tau(i,j)=T_k \text{ for some } k \, \right],
  \end{equation}
thus  $\eta_{i,j}$ is the indicator that the walkers indexed by $i$ and $j$ annihilate each other before any other walker annihilates either of them.
  Let us define 
  \begin{equation} \label{eq:def_number_of_annih}
  \mathcal{A}_\infty(\mathcal{X})=
  \sum_{i=1}^{|\mathcal{X}|} \sum_{j=i+1}^{|\mathcal{X}|} \eta_{i,j}
  \end{equation}
  the total number of annihilations that ever occurred. Now we have
  \begin{equation} \label{eq:annih_number_identity}
   \mathcal{N}'_\infty(\mathcal{X})=|\mathcal{X}| - 2 \mathcal{A}_\infty(\mathcal{X}),
   \end{equation}
  since each annihilation event kills two walkers. 
  By \eqref{eq:dualityinf}, Lemma \ref{lemma:coalescing_dominates_annihilating} and
    \eqref{eq:annih_number_identity} we only need to prove
  \begin{equation}\label{eq:number_of_annih_neg_corr}
  \mathbb{E}\left( \alpha^{-2 \mathcal{A}_\infty(\mathcal{X}) } \right) \leq 
  \prod_{1 \leq i < j \leq |\mathcal{X}|} \left( 1+ h_R(x_i,x_j)\left( \alpha^{-2}-1 \right)  \right) ,
  \qquad 0 <\alpha \leq 1 
  \end{equation}
  in order to complete the proof of Lemma \ref{lemma:bound_on_occupied_using_annih}.
 Let us introduce auxiliary Bernoulli random variables $\eta^*_{i,j}$, $1 \leq i < j \leq |\mathcal{X}|$ such that they are independent
and (recalling the definition of $h_R$ from \eqref{eq:def_spread_out_hitting_prob})
\begin{equation}
\label{eq:def_aux_eta_star_i_j}
 \mathbb{P} [\eta^*_{i,j}=1]= 1-\mathbb{P} [\eta^*_{i,j}=0] 
=
  h_R(x_i,x_j) 
 \stackrel{ \eqref{eq:spread_out_meeting_time} }{ =}
   \mathbb{P} [\tau(i,j) <+\infty]. 
   \end{equation}
Similarly to \eqref{eq:def_number_of_annih}, let us define
\begin{equation}\label{eq:def_auxiliary_indep_sum_spr_out}
\mathcal{A}^*_\infty(\mathcal{X})=
  \sum_{i=1}^{|\mathcal{X}|} \sum_{j=i+1}^{|\mathcal{X}|} \eta^*_{i,j}.
\end{equation}
Now the right-hand side of \eqref{eq:number_of_annih_neg_corr} is equal to 
$\mathbb{E}\left( \alpha^{-2 \mathcal{A}^*_\infty(\mathcal{X}) } \right)$,
 thus in order to prove \eqref{eq:number_of_annih_neg_corr}
  we only need to show that for any $\lambda \geq 0$ we have
  \begin{equation}
  \label{eq:mom_gen_ineq_decorr_spread_out}
   \mathbb{E}\left[ e^{ \lambda \mathcal{A}_\infty(\mathcal{X}) } \right] \leq
   \mathbb{E}\left[ e^{ \lambda \mathcal{A}^*_\infty(\mathcal{X}) }\right].
  \end{equation}
By taking the Taylor expansion of the above exponential functions about $\lambda=0$,
 we see that we only need to prove
\begin{equation*}
\mathbb{E}\left[ \left( \mathcal{A}_\infty(\mathcal{X}) \right)^k \right] 
\leq
\mathbb{E}\left[ \left( \mathcal{A}^*_\infty(\mathcal{X}) \right)^k \right] , \qquad
k \geq 0.
\end{equation*}
in order to achieve \eqref{eq:mom_gen_ineq_decorr_spread_out}.
By expanding the $k$'th power of the sums in the definitions of
$\mathcal{A}_\infty(\mathcal{X})$ (see \eqref{eq:def_number_of_annih})
and $\mathcal{A}^*_\infty(\mathcal{X})$ (see \eqref{eq:def_auxiliary_indep_sum_spr_out}), we see
that we only need to prove that the annihilation events are negatively correlated, i.e., that
\begin{equation}\label{eq:annih_ind_dom_ny_indep}
\mathbb{P} \left[ \eta_{i_1,j_1}=\dots=\eta_{i_k,j_k}=1 \right] \leq
\mathbb{P} \left[ \eta^*_{i_1,j_1}=\dots=\eta^*_{i_k,j_k}=1 \right]
\end{equation}
holds for any $k \geq 1$ and any $1 \leq i_l<j_l \leq |\mathcal{X}|$, $1\leq l \leq k$.
First, we may assume that the the list of pairs $\{i_1,j_1\}, \dots, \{i_k,j_k\}$ does not contain the 
same pair more than once, because we can throw out such duplicates and reduce the value of $k$ without changing the probabilities on either side of \eqref{eq:annih_ind_dom_ny_indep}. 
Second, we may also assume that the sets $\{i_1,j_1\}, \dots, \{i_k,j_k\}$ are disjoint, because if some of these sets have non-empty intersection, then the left-hand side of  $\eqref{eq:annih_ind_dom_ny_indep}$ is equal to zero by the definition of
the indicators $\eta_{i,j}$ (see \eqref{eq:def_annih_indicator_i_j}): a walker can only be annihilated once.
Now if the sets $\{i_1,j_1\}, \dots, \{i_k,j_k\}$ are disjoint, then
\begin{multline*}
\mathbb{P} \left[ \eta_{i_1,j_1}=\dots=\eta_{i_k,j_k}=1 \right]
\stackrel{ \eqref{eq:def_annih_indicator_i_j} }{\leq}  
\mathbb{P} \left[  \tau(i_1,j_1) <+\infty, \dots,  \tau(i_k,j_k) <+\infty \right]
\stackrel{(*)}{=} \\
\prod_{l=1}^k \mathbb{P} \left[  \tau(i_l,j_l) <+\infty \right]
\stackrel{ \eqref{eq:def_aux_eta_star_i_j} }{=}
\prod_{l=1}^k \mathbb{P} \left[  \eta^*_{i_l,j_l} =1 \right]
\stackrel{(**)}{=}  
\mathbb{P} \left[ \eta^*_{i_1,j_1}=\dots=\eta^*_{i_k,j_k}=1 \right],
\end{multline*}
where $(*)$ holds because the walkers $X^i_t$, $1\leq i \leq |\mathcal{X}|$ are independent and
the sets $\{i_1,j_1\}, \dots, \{i_k,j_k\}$ are disjoint, and $(**)$ holds because
$\eta^*_{i,j}$, $1 \leq i < j \leq |\mathcal{X}|$ are independent.
The proof of \eqref{eq:annih_ind_dom_ny_indep} and Lemma \ref{lemma:bound_on_occupied_using_annih} is complete.

\section{ Nearest-neighbour model, $\boldsymbol{d\ge 5}$}
\label{section_nearest_neighbour}

The goal of this section is to prove Theorem \ref{thm:nearest_neighbour}. More specifically, 
we will show that \eqref{eq:noperc} holds for any $d \geq 5$ and $R =1$ and
 some $\alpha_0=\alpha_0(d)>0$. Note that the same proof would work for any $R \geq 1$; the only reason we stick to
 the classical nearest-neighbour case is to ease notation. 
 We also note that a slight generalization of the method presented in this section would yield a proof
 of both Theorem \ref{thm:nearest_neighbour} and Theorem \ref{thm:spread_out}, however we chose to also present 
 in Section \ref{section_spread_out} a relatively short argument which only proves Theorem \ref{thm:spread_out}.

\medskip

We use the graphical construction of 
$\xi^{(\alpha)}$ distributed as $\mu_\alpha$ (see \eqref{eq:jointc2}). However, we will often drop the dependence on 
$\alpha$ from our notation, especially if a particular calculation works for any $\alpha \in (0,1)$. 

We will use the renormalization scheme of Section \ref{subsection:renorm}.
In order to specify the value of $\ell$ in \eqref{eq:renorm_scales_L_N} we fix the exponents
\begin{equation}\label{eq:def:epsilon_delta}
\varepsilon=\frac{1}{4d},  \qquad  \delta= \frac{\varepsilon}{d}.
\end{equation} 
The reasons for the choice of $\varepsilon$ and $\delta$ are discussed in Remark \ref{remark_about_varepsilon}
and Remark \ref{remark_martingale_green}.

 The following choice of $\ell$ in \eqref{eq:renorm_scales_L_N} will be suitable for our purposes:
 \begin{equation}
 \label{eq:choice_ell}
\ell = 3^{1/\delta}. 
\end{equation}
This choice of $\ell$ will be used in Section \ref{subsection:completion_of_renorm_proof} 
to guarantee the
convergence of certain geometric series which are similar in flavour to \eqref{eq:geo_spread_out}.

The choice of a large enough $L$ in \eqref{eq:renorm_scales_L_N} will be specified later in 
Section \ref{subsection:completion_of_renorm_proof}. In Remark \ref{remark_why_no_L_equals_1_in_nearest_neighbour_case}
we explain why $L=1$ is an insufficient choice in the $R=1$ case.

Choosing $\ell$ as in \eqref{eq:choice_ell} we have
\begin{equation}
\label{eq:countemb}
|\Lambda_N| \stackrel{ \eqref{eq:countemb_general}}{\leq}   C^{2^N} \quad \text{for some} \quad C=C(d).
\end{equation}

Combining Definition \ref{def_connections_crossings}, (\ref{eq:countemb}) and Lemma \ref{lem:embpath} in a union bound, we get, for any $N$,
\begin{equation}\label{bibary_tree_union_bound}\begin{split}
&\mathbb{P}[ B(L_N-2) \stackrel{*\xi}{\longleftrightarrow} B(2L_N)^c ]\\&\qquad \leq 
C^{2^N} \max_{ \mathcal{T} \in \Lambda_{N} } 
\mathbb{P} \left[ \bigcap_{m \in T_{(N)}} \{ B(\mathcal{T}(m),L) \stackrel{*\xi}{\longleftrightarrow} B(\mathcal{T}(m),2L)^c \} \right].\end{split}
\end{equation}
We will take a closer look at the crossing events that occur
on the right-hand side of \eqref{bibary_tree_union_bound} in Claim \ref{claim_crossing} below.
We discuss an open question related to crossing events in the low-dimensional setting in Remark \ref{remark_conjecture_one_block}.

We now fix $N$ and a proper embedding $\mathcal{T} \in \Lambda_N$ with the aim of bounding the probability on the right-hand side in \eqref{bibary_tree_union_bound} (see \eqref{eq:last_union_bound} below).
We recall the graphical construction \eqref{graphical_walk_poisson} of the coalescing random walks 
 $\left( Y_t^x \right)_{t \geq 0, x \in \Z^d}$, the construction \eqref{eq:jointc2} of
 the configuration $\left( \xi(x) \right)_{x \in \Z^d}$ 
 as well as the  definition of $\varepsilon$ from \eqref{eq:def:epsilon_delta} and let
\begin{equation} \label{eq:def_T}
 T  = L^{2-\varepsilon}
\end{equation}
and, for $x, y \in \Z^d$, we define the events
\begin{align}
\label{def_eq_E_x}
&E_x = \left\{ \max_{0 \leq t \leq T} | Y_t^x -x | > \frac14 L \right\},\medskip\\
\label{def_eq_E_xy}
&E_{x,y} = E_x^c \cap E_y^c \cap  \left\{ Y_t^x \neq Y_t^y, \, 0 \leq t \leq T\right\}, \medskip\\
\label{def_eq_F_xy}
&F_{x,y} = E_{x,y} \cap \{\xi(x) = \xi(y) = 1\}.
\end{align}

\begin{remark}\label{remark_about_varepsilon}
We defined $T \ll L^2$ in \eqref{eq:def_T} because we want $\mathbb{P}[E_x] \ll 1$, see \eqref{def_eq_beta} and \eqref{eq:bound_beta}.
We note that instead of defining $\varepsilon$ as in \eqref{eq:def:epsilon_delta}, we could in fact take
$\varepsilon$ as any positive constant which is small enough so that \[(2-\varepsilon)(1-\frac{d}{2} + \varepsilon)< 2-d+\frac14\] holds, see \eqref{eq:bounding_order_of_p} below.
\end{remark}

\begin{claim}\label{claim_crossing}
For any $z \in L \Z^d \stackrel{ \eqref{eq:def_renom_lattice} }{=} \mathcal{L}_0$, the following inclusion holds:
\begin{equation}\label{connection_inclusion_E_x_E_xy}
 \{ B(z,L) \stackrel{*\xi}{\longleftrightarrow} B(z,2L)^c \} \subseteq
\left( \bigcup_{x \in B(z,2L)} E_x \right) \cup
\left( \bigcup_{ \substack{ x , y \in B(z,2L) \\ |x - y| = 1 } } F_{x,y} \right).
\end{equation}
\end{claim}
\begin{proof}
Assume that the event on the left-hand side occurs. Then there exists a $*$-connected path $\gamma(1), ,\ldots, \gamma(k)$ with $|\gamma(1)| = L+1,\; |\gamma(k)| = 2L$ and $\xi(\gamma(i)) = 1$ for each $i$. For one such path, define
$$i^* = \max\{i \leq k: Y^{\gamma(i)}_t = Y^{\gamma(1)}_t \text{ for some } t \leq T\}.$$
If $i^* = k$, then $E_{\gamma(1)} \cup E_{\gamma(k)}$ occurs, since $|\gamma(1) - \gamma(k)| > L/2$. If $i^* < k$, then the walks $(Y^{\gamma(i^*)}_t)$ and $(Y^{\gamma(i^* + 1)}_t)$ do not meet before time $T$, so either $E_{\gamma(i^*)} \cup E_{\gamma(i^*+1)}$ or $F_{\gamma(i^*),\gamma(i^*+1)}$ occurs.
\end{proof}

With (\ref{connection_inclusion_E_x_E_xy}) in mind, given $\mathcal{T} \in \Lambda_N $ we choose two sets
 $\mathcal{X},\mathcal{Y} \subset \Z^d$. 
 
\begin{definition}\label{def:admissible}
The pair $(\mathcal{X}, \mathcal{Y})$,  $\mathcal{X},\mathcal{Y} \subset \cup_{m \in T_{(N)}} B( \mathcal{T}(m), 2L)$ is called admissible if
\begin{enumerate}[(i)]
\item \label{admissible_i}
 for any $m\in T_{(N)}$, 
$(|B(\mathcal{T}(m), 2L) \cap \mathcal{X}|, \;|B(\mathcal{T}(m), 2L) \cap \mathcal{Y}|)$ is either $(2,0)$ or $(0,1)$;
\item \label{admissible_ii}
if $  B(\mathcal{T}(m), 2L) \cap \mathcal{X}  = \{x,y\}$, then $|x -y| = 1$.
\end{enumerate}
The set of all admissible pairs $(\mathcal{X}, \mathcal{Y})$ associated to $\mathcal{T}$ is denoted $\mathcal{P}_\mathcal{T}$.
\end{definition}

\begin{lemma} Given $\mathcal{T} \in \Lambda_N $,
\begin{enumerate}
\item For any $(\mathcal{X}, \mathcal{Y}) \in \mathcal{P}_\mathcal{T}$ we have  
\begin{equation}\frac{1}{2}|\mathcal{X}| + |\mathcal{Y}| = 2^N. \label{calX_plus_calY} \end{equation}
\item There exists $C=C(d)$ such that the number of admissible pairs can be bounded by 
 \begin{equation}\label{eq:bound_on_number_of_admissible_pairs}
 |\mathcal{P}_\mathcal{T}| \leq (C  L^d)^{2^N}.
 \end{equation}
\item We have \begin{equation}\label{connection_sets_X_Y_union_bound}
\bigcap_{m \in T_{(m)}} \{B(\mathcal{T}(m), L) \stackrel{*\xi}{\longleftrightarrow} B(\mathcal{T}(m), 2L)^c\} \subset \bigcup_{(\mathcal{X},\mathcal{Y}) \in \mathcal{P}_{\mathcal{T}}} 
\left(\bigcap_{ \substack{ x,z \in \mathcal{X}\\ |x- z| =1 }} F_{x,z}\right) \cap \left( \bigcap_{y \in \mathcal{Y}} E_y \right).
\end{equation}
\item For every $x \in \mathcal{X}$, we have
\begin{equation}\label{eq:count_neighb_X}
\left| \mathcal{X}  \cap B(x, \ell^k L/2) \right| \leq 2^k, \qquad k \geq 1.
\end{equation}
\end{enumerate}
\end{lemma} 

\begin{proof}
Given an admissible pair $(\mathcal{X},\mathcal{Y})$ associated to $\mathcal{T}$, define
$$\begin{aligned}
&\mathcal{A}_{(\mathcal{X},\mathcal{Y})} = \{m \in T_{(N)}: (|B(\mathcal{T}(m),2L) \cap \mathcal{X}|,|B(\mathcal{T}(m),2L) \cap \mathcal{Y}|) = (2,0)\},
\end{aligned}$$
so that, by  Definition \ref{def:admissible} \eqref{admissible_i}, we have
$$\begin{aligned}
&T_{(N)}\backslash \mathcal{A}_{(\mathcal{X},\mathcal{Y})} = \{m \in T_{(N)}: (|B(\mathcal{T}(m),2L) \cap \mathcal{X}|,|B(\mathcal{T}(m),2L) \cap \mathcal{Y}|) = (0,1)\}
\end{aligned}$$
and thus \eqref{calX_plus_calY} holds:
$$\begin{aligned}2^N = |T_{(N)}| = |\mathcal{A}_{(\mathcal{X},\mathcal{Y})}| + |T_{(N)}\backslash  \mathcal{A}_{(\mathcal{X},\mathcal{Y})}| = \frac12|\mathcal{X}| + |\mathcal{Y}|.
\end{aligned}$$
Additionally, by Definition \ref{def:admissible}, the pair
 $(\mathcal{X},\mathcal{Y})$ is determined when we choose $\mathcal{A}_{(\mathcal{X},\mathcal{Y})}$ and then, for each $m\in \mathcal{A}_{(\mathcal{X},\mathcal{Y})}$, we choose two $*$-connected vertices in 
$B(\mathcal{T}(m), 2L)$ 
and for each $m\in T_{(N)}\backslash \mathcal{A}_{(\mathcal{X},\mathcal{Y})}$, 
we choose one vertex in 
$B(\mathcal{T}(m), 2L)$. Thus \eqref{eq:bound_on_number_of_admissible_pairs} indeed holds:
\begin{equation*}
 |\mathcal{P}_\mathcal{T}| \leq
 \sum_{A \subseteq T_{(N)}} \left( |B( 2L)|\cdot 3^d \right)^{|A|} \cdot  
  |B( 2L)| ^{2^N- |A|}
 \leq (C L^d)^{2^N}.
\end{equation*}

The inclusion (\ref{connection_sets_X_Y_union_bound}) is a consequence of (\ref{connection_inclusion_E_x_E_xy})
and Definition \ref{def:admissible}.

The bound \eqref{eq:count_neighb_X} follows from Lemma \ref{lem:embsp} and the fact that for each
$m \in T_{(N)}$ we have $|\mathcal{X} \cap B(\mathcal{T}(m),2L)| \leq 2$ by
 Definition \ref{def:admissible}\eqref{admissible_i}.

\end{proof}

Putting together \eqref{bibary_tree_union_bound}, \eqref{eq:bound_on_number_of_admissible_pairs} and \eqref{connection_sets_X_Y_union_bound}, we obtain
\begin{eqnarray}\nonumber
&&\mathbb{P}[ B(L_N-2) \stackrel{*\xi}{\longleftrightarrow} B(2L_N)^c ] \\&&\hspace{0.4cm}\label{eq:last_union_bound}\leq 
\left(C L^d \right)^{2^N} 
\max_{ \substack{ \mathcal{T} \in \Lambda_{N}  \\
(\mathcal{X},\mathcal{Y}) \in \mathcal{P}_{\mathcal{T}} } }
\mathbb{P} \left[
\left(\bigcap_{ \substack{ x,z \in \mathcal{X}\\ |x- z|=1 }} F_{x,z}\right) \cap 
\left( \bigcap_{y \in \mathcal{Y}} E_y \right) \right]
\end{eqnarray}
for some constant $C=C(d)$.

The main ingredient in the proof of \eqref{eq:noperc} is the following proposition.
\begin{proposition}\label{prop:final_max_bound}
For every $d \geq 5$, there exist  $L_{(0)} \geq 2 $ and $\mathcal{C}=\mathcal{C}(d)<+\infty$ 
such that for any
 $L \geq L_{(0)}$, any $\alpha \leq L^{2-d+1/4}$ and any $N \geq 1$ we have
\begin{equation}
\label{eq:prop_final_max_bound}
\max_{ \substack{ \mathcal{T} \in \Lambda_{N}  \\
(\mathcal{X},\mathcal{Y}) \in \mathcal{P}_{\mathcal{T}} } }
\mathbb{P} \left[
\left(\bigcap_{ \substack{ x,z \in \mathcal{X}\\ |x- z|=1 }} F_{x,z}\right) \cap 
\left( \bigcap_{y \in \mathcal{Y}} E_y \right) \right] \leq 
\left( \mathcal{C} L^{4 -2d + 1/2} \right)^{2^N}.
\end{equation}
\end{proposition}
\noindent Together with \eqref{eq:last_union_bound} and the assumption  $d \geq 5$, this proposition immediately yields the desired result \eqref{eq:noperc} if we choose $L$ large enough. We will explain why our
 method fails to prove \eqref{eq:noperc} if $d=3,4$ and $R=1$ in Remark \ref{remark_why_no_d_3_4}.

The rest of this section is devoted to the proof of Proposition \ref{prop:final_max_bound}. 

\subsection{Reduction to coalescing walks with initial period of no coalescence}
\label{subsection_separating_alpha_beta_using_graphical}

From now on, we fix not only $\mathcal{T} \in \Lambda_{N}$ (see Definition \ref{def_proper_embedding_of_trees}), but also $(\mathcal{X},\mathcal{Y}) \in \mathcal{P}_{\mathcal{T}} $ (see Definition \ref{def:admissible}). Recalling the definition of $E_x$ in \eqref{def_eq_beta}, let us define 
\begin{equation}\label{def_eq_beta}
 \beta=\beta(L,d) = \mathbb{P}[ E_x ] = \mathbb{P}[E_0].
\end{equation}

\begin{lemma}\label{lemma:separating_beta_alpha} We have
\begin{equation} \label{eq:betaalpha}
\mathbb{P} \left[
\left(\bigcap_{ \substack{ x,z \in \mathcal{X}\\ |x- z|=1 }} F_{x,z}\right) \cap 
\left( \bigcap_{y \in \mathcal{Y}} E_y \right) \right] \leq \beta^{|\mathcal{Y}| } \cdot
 \mathbb{E} \left[ \alpha^{\mathcal{N}_\infty(\mathcal{X})} \cdot
  \mathds{1}_{\left\{\mathcal{N}_T(\mathcal{X}) = |\mathcal{X}|  \right\}}\right].
\end{equation}
\end{lemma}
\begin{remark}
Recall the definition of $\mathcal{N}_{\infty}(\cdot)$ in (\ref{eq:defN}). 
The event $\left\{\mathcal{N}_T(\mathcal{X}) = |\mathcal{X}|  \right\}$  on the right-hand side of \eqref{eq:betaalpha} is simply the event that the walks started from the
 vertices of $\mathcal{X}$ do not coalesce with each other
 before time $T$.
\end{remark}

\begin{proof}

We will use the joint graphical construction of the system of coalescing walks and the
 configuration $\xi=\xi^{(\alpha)}$ described by equation (\ref{eq:jointc2}).
  Since our set $\mathcal{X}$ is fixed, we can and will assume that, in the enumeration of 
  $\Z^d$ that was needed for  (\ref{eq:jointc1}), the vertices in $\mathcal{X}$ come before all other vertices of $\Z^d$. 
  We can thus write 
\begin{equation} 
\label{eq:enumx}
\mathcal{X} = \{x_1, x_2, \ldots, x_{|\mathcal{X}|}\}.
\end{equation}









The occurrence of each event $E_y$, for $y \in \mathcal{Y}$, can be decided from the Poisson processes in the graphical construction in the space-time box 
$$\{z\in \Z^d: |z - y| \leq 1+L/4\} \times [0,T], \qquad y \in \mathcal{Y}.$$
The occurrence of $\cap_{x,z\in \mathcal{X},\; |x- z|=1 } F_{x,z}$ can be decided from the random variables $\mathcal{U}_n: 1 \leq n \leq |\mathcal{X}|$ and the Poisson processes in the graphical construction in the space-time set
$$\left(\{w \in \Z^d: \mathrm{dist}(\{w\},\mathcal{X}) \leq  1+ L/4\} \times [0,T] \right)\cup \left( \Z^d \times (T, \infty)\right).$$
Using (\ref{eq:embsp}), we see that these space-time sets are all disjoint, and thus

$$
\mathbb{P} \left[
\left(\bigcap_{ \substack{ x,z\in \mathcal{X}\\ |x- z|=1 }} F_{x,z}\right) \cap 
\left( \bigcap_{y \in \mathcal{Y}} E_y \right) \right] 
=
 \mathbb{P}\left[\bigcap_{ \substack{ x,z\in \mathcal{X}\\ |x- z|=1 }} F_{x,z} \right]
 \cdot
 \prod_{y\in \mathcal{Y}} \mathbb{P}\left[  E_y \right] 
\stackrel{ \eqref{def_eq_beta} }{=} 
\beta^{|\mathcal{Y}|} \cdot \mathbb{P}\left[\bigcap_{ \substack{ x,z\in \mathcal{X}\\ |x- z|=1 }} F_{x,z} \right] .
$$

We now define $\mathcal{M}_\mathcal{X} = \{\eta(x_k): 1\leq k \leq |\mathcal{X}|\}$, where $\eta$ is defined in (\ref{eq:jointc1}).
For every non-empty $A \subseteq \{1,\ldots, |\mathcal{X}|\}$ we have
\begin{equation}\label{eq_A_E_F_alpha}
\mathbb{P}\left[\{\mathcal{M}_\mathcal{X} = A\} \cap \bigcap_{ \substack{ x,z\in \mathcal{X}\\ |x- z|=1 }} F_{x,z} \right]
\stackrel{\eqref{eq:jointc2}, \eqref{def_eq_F_xy} }{=}
\alpha^{|A|} \cdot \mathbb{P} \left[\{ \mathcal{M}_\mathcal{X} = A\}  \cap 
\bigcap_{\substack{x,z \in \mathcal{X}:\\ |x - z|=1 }} E_{x,z}\right].
\end{equation}
 Note that, by \eqref{eq:defN}, \eqref{eq:jointc1} and \eqref{eq:enumx},  we have $|\mathcal{M}_\mathcal{X}| = \mathcal{N}_\infty(\mathcal{X})$. 
Therefore
\begin{multline*}
\mathbb{P}\left[\bigcap_{ \substack{ \{x,z\}\in \mathcal{X}\\ |x - z|=1 }} F_{x,z} \right]
\stackrel{ \eqref{eq_A_E_F_alpha} }{=}
\sum_{k=1}^{|\mathcal{X}|} \alpha^k \cdot \mathbb{P}\left[ \{ |\mathcal{M}_\mathcal{X}| = k\} \cap  \bigcap_{\substack{x,z \in \mathcal{X}:\\ |x - z|=1 }} E_{x,z}\right]\\
 =\mathbb{E}\left[\alpha^{\mathcal{N}_\infty(\mathcal{X})}\cdot 
 \mathds{1}\left\{\bigcap_{\substack{x,z \in \mathcal{X}:\\ |x - z|=1 }} E_{x,z}\right\} \right] 
 \stackrel{ \eqref{eq:embsp}, \eqref{def_eq_E_x}, \eqref{def_eq_E_xy} }{\leq} 
 \mathbb{E} \left[ \alpha^{\mathcal{N}_\infty(\mathcal{X})} \cdot \mathds{1}_{\left\{\mathcal{N}_T(\mathcal{X}) = |\mathcal{X}|\right\}}\right].
\end{multline*}
The proof of Lemma \ref{lemma:separating_beta_alpha} is complete.
\end{proof}

\subsection{Reduction to independent random walks}
\label{subsection_reduction_to_indep_walks}

Our next goal is to bound the expectation on the right-hand side of (\ref{eq:betaalpha}).
 We will need to take a close look at the coalescing walks $\{(Y^x_t)_{t \geq 0}: x \in \mathcal{X}\}$. 
For this, it will no longer be convenient to work with the graphical construction of the coalescing walks
 using the Poisson processes $(D^{(x,y)}_t)$ that we described in Section \ref{section_voter_graphical_construction}.
 Rather, we will switch to a new probability space, in which we will give a different
  representation of the system of coalescing walks.

The following construction will depend on the set $\mathcal{X}$ which has been fixed 
at the beginning of Section \ref{subsection_separating_alpha_beta_using_graphical}
 and also on the enumeration of $\mathcal{X}$ that was fixed in (\ref{eq:enumx}). Let $P$ denote a probability measure under which one defines a collection of processes $\{(X^x_t)_{t\ge 0}: x \in \mathcal{X}\}$ satisfying:
\begin{itemize} \item for each $x \in \mathcal{X}$, $(X^x_t)_{t\ge 0}$ is a continuous-time, nearest neighbor random walk on $\Z^d$ with jump rate 1 and $X^x_0 = x$;
\item these walks are all independent.
\end{itemize}
(We emphasize that this is \textit{not} a system of coalescing walks). 
The expectation operator associated to $P$ is denoted by $E$. We then define the processes:
\begin{itemize}
\item $\{(W^x_t)_{t \ge 0}: x\in \mathcal{X}\}$. They are defined by induction. Put $W^{x_1}_t = X^{x_1}_t$ for all $t$. Assume $W^{x_1},\ldots, W^{x_n}$ are defined and let $$\sigma = \inf\{t: X^{x_{n+1}}_t = W^{x_k}_t \text{ for some } k \leq n\}.$$ On $\{\sigma = \infty\}$, let $W^{x_{n+1}}_t = X^{x_{n+1}}_t$ for all $t$. On $\{\sigma < \infty\}$, 
let $K$ be the smallest index such that $X^{x_{n+1}}_\sigma = W^{x_K}_\sigma$. Put
$$W^{x_{n+1}}_t = \left\{\begin{array}{ll} X^{x_{n+1}}_t &\text{if } t \leq \sigma;\\W^{x_K}_t &\text{if } t > \sigma.\end{array}\right.$$
\item $\{(Z^x_t)_{t \ge 0}: x\in \mathcal{X}\}$. These are defined exactly as above, with the only difference that in the induction step, $\sigma$ is defined by
$$\sigma = \inf\{t \geq T: X^{x_{n+1}}_t = Z^{x_k}_t \text{ for some } k \leq n\}.$$
\end{itemize}

\begin{claim}\label{claim_two_constructions_of_coalescing_walks}
\begin{enumerate}[(i)]
\item $\{(W^x_t)_{t \geq 0}: x \in \mathcal{X}\}$ is a system of coalescing walks started from $\mathcal{X}$; in particular, its law under $P$ is the same as that of $\{(Y^x_t)_{t \geq 0}: x \in \mathcal{X}\}$ under $\mathbb{P}$.
\item  $\{(Z^x_t)_{t \ge 0}: x \in \mathcal{X}\}$ is a system of random walks that move independently (with no coalescence) up to time $T$ and after time $T$, behave as a system of coalescing walks.
\end{enumerate}
\end{claim}
The proof of this claim is straightforward and we omit it.

Similarly to \eqref{eq:defN}  we also define 
$$
\begin{aligned}&\mathcal{N}^W_t = |\{W^x_t: x \in \mathcal{X}\}|, \qquad \mathcal{N}^W_\infty = \lim_{t\to \infty} \mathcal{N}^W_t
 \\
&\mathcal{N}^Z_t = |\{Z^x_t: x \in \mathcal{X}\}|,
 \qquad \mathcal{N}^Z_\infty = \lim_{t\to \infty} \mathcal{N}^Z_t.
 \end{aligned}
$$
We now have
\begin{eqnarray}\nonumber \mathbb{E} \left[ \alpha^{\mathcal{N}_\infty(\mathcal{X})} \cdot \mathds{1}_{\left\{\mathcal{N}_T(\mathcal{X}) = |\mathcal{X}|\right\}}\right]=E \left[ \alpha^{\mathcal{N}^W_\infty} \cdot \mathds{1}_{\left\{X^x_t \neq X^y_t \text{ for all }x, y \in \mathcal{X},\;x \neq y \text{ and } t \leq T   \right\}}\right] \\
=E \left[ \alpha^{\mathcal{N}^Z_\infty} \cdot \mathds{1}_{\left\{X^x_t \neq X^y_t \text{ for all }x, y \in \mathcal{X},\;x \neq y \text{ and } t \leq T   \right\}}\right]
\leq E \left[ \alpha^{\mathcal{N}^Z_\infty}\right].\label{eq:bound_Zinf0}
\end{eqnarray}

At this point one might be tempted to apply Lemma \ref{lemma:coalescing_dominates_annihilating}, i.e., to switch from coalescing
to annihilating walks. 
In Remark \ref{remark:why_no_annih_in_nearest_neighbour} we explain
why this method cannot be used to prove Theorem \ref{thm:nearest_neighbour}.

\subsection{A stochastic domination result}
\label{section:main_stochastic_domination}

In this subsection we give definitions and state preliminary results (Lemma \ref{lem:V_star_dominates}, Lemma \ref{lem:U_n_rarely_one}, Proposition \ref{prop_near_martingale_deviation} and Proposition \ref{prop:stoch_dom}) which will put us in position to prove Proposition \ref{prop:final_max_bound} in Section \ref{subsection:completion_of_renorm_proof}. The following details the interdependence of these results.\\[-.4cm]
\begin{center}\fbox{
\begin{tabular}{lcl}&\underline{proved in Section}&\underline{needed for the proof of}\\
$\bullet\;$Lemma \ref{lem:V_star_dominates}&\ref{subsection:proof_of_lemmas_lem:V_star_dominates_and_lem:U_n_rarely_one}&Proposition \ref{prop:final_max_bound}, Proposition \ref{prop:stoch_dom}\\
$\bullet\;$Lemma \ref{lem:U_n_rarely_one}&\ref{subsection:proof_of_lemmas_lem:V_star_dominates_and_lem:U_n_rarely_one}&Proposition \ref{prop:stoch_dom}\\
$\bullet\;$Proposition \ref{prop:stoch_dom}&\ref{subsection:proof_of_prop:stoch_dom}&Proposition \ref{prop:final_max_bound}\\
$\bullet\;$Proposition \ref{prop_near_martingale_deviation}&\ref{section:kallenberg}&Lemma \ref{lem:V_star_dominates}
\end{tabular}}
\end{center}

Recall the notion of the enumeration $\mathcal{X} = \{x_1, x_2, \ldots, x_{|\mathcal{X}|}\}$
from \eqref{eq:enumx}.
Let us define 
\begin{equation}\label{def_eq_U_n_U}
U_n:= \mathds{1} \{ \exists\; k<n, \;  t \geq T \; : \;  X^{x_n}_t=X^{x_k}_t \}, \quad 1 \leq n \leq  |\mathcal{X}|, \qquad
U=\sum_{n=1}^{|\mathcal{X}|} U_n.
\end{equation}
In words: $U_n$ is the indicator of the event that the $n$'th walker hits any of the previous walkers after $T$.
Recalling the construction of Section \ref{subsection_reduction_to_indep_walks} we have
\begin{equation*}
\mathcal{N}^Z_\infty = 
|\mathcal{X}| - \sum_{n=1}^{|\mathcal{X}|} \mathds{1} \{ \exists\; k<n, \;  t \geq T \; : \;  Z^{x_n}_t=Z^{x_k}_t \} \geq
|\mathcal{X}| - U
\end{equation*}
and we can thus bound
\begin{equation} E\left[ \alpha^{\mathcal{N}^Z_\infty}\right] \leq 
\alpha^{|\mathcal{X}|} \cdot E\left[ \alpha^{-U}\right]. \label{eq:bound_Zinf}
\end{equation}

Let us now describe the main ideas of this subsection.
The indicator variables $U_n$, $1 \leq n \leq  |\mathcal{X}|$ are not independent; however,
in Proposition \ref{prop:stoch_dom} we will argue that their sum can be dominated by a sum of independent variables. Let us explain now the heuristics for this domination. 
Suppose we reveal the paths $(X^{x_n}_t)_{t \geq 0}, 1 \leq n \leq |\mathcal{X}|$
 one by one, starting with $(X^{x_1}_t)_{t \geq 0}$. We think of each path $n$ as a trial: a \textit{success} if it avoids all the previously revealed paths after time $T$ (that is, if $U_n = 0$), and a \textit{failure} otherwise. At the time of revealing path $n$, it should have a high probability of being a success
  (since the set $\{ X^x_T \; : \; x \in \mathcal{X} \}$ is very sparse), unless some  path of index $k < n$ behaved in an atypical manner that makes it exceptionally likely that
  $(X^{x_n}_t)_{t \geq T}$ meets $(X^{x_k}_t)_{t \ge T}$.
 In \eqref{def_eq_V_k_n} below we will introduce the variable $V_{k,n}$ as the indicator of this event that path $k$
  \emph{endangers} trial $n$. We then rely on two fundamental observations.
\begin{itemize}  
\item   First (see Lemma \ref{lem:V_star_dominates}): 
  since the random set $\{ X^x_T \; : \; x \in \mathcal{X} \}$ 
  is very sparse (as suggested by  \eqref{eq:count_neighb_X}), it is very unlikely that a path endangers a trial, so that the random variables $V_k = \sum_{n > k} V_{k,n}$, which represent the number of trials endangered by each path $k$, are equal to zero with high probability. 
  \item Second (see Lemma \ref{lem:U_n_rarely_one}): if trial $n$ is not endangered by any path of index $k < n$, then it is very likely to be successful.
  \end{itemize}

\medskip

For any $x,y \in \Z^d$ let us define the random variable
\begin{equation}\label{def_M_x_y_T_martingale}
M^{x,y,T}_\infty = \mathbb{P} [ \exists s \geq T \; : \; X^y_s=X^x_s \; | \; X^x_u \,:\, 0 \leq u < \infty ].
\end{equation}
(the reason for the $\infty$ symbol in $M^{x,y,T}_\infty$ will become clear in Section \ref{section:kallenberg}).

Recall the definition of $\varepsilon=\varepsilon(d)$ and $\delta=\delta(d)$ from \eqref{eq:def:epsilon_delta}.

\begin{proposition}\label{prop_near_martingale_deviation}
There exists $T_0=T_0(d)<+\infty$ and $D_0=D_0(d)<+\infty$ such that
\begin{align}
&P[ M^{x,y,T}_{\infty} > T^{1-\frac{d}{2}+\varepsilon} ] \leq e^{-T^{\delta}}, 
 & \quad x,y  \in \Z^d,& \quad
& T \geq T_0, 
  \label{eq:martingale_deviation} 
  \\
& P[ M^{x,y,T}_{\infty} > | x-y |^{2-d+\varepsilon} ]  \leq e^{-| x-y |^{\delta}}, 
& \quad x,y  \in \Z^d,& \; |x-y| \geq D_0, \quad 
& T \geq 0  . 
\label{eq:martingale_deviation_far}
\end{align}
\end{proposition}

\begin{remark}\label{remark_martingale_green}
 By \eqref{green_bounds} we have $E\left[ M^{x,y,T}_{\infty} \right] \asymp T^{1-d/2} \wedge |x-y|^{2-d}$, thus
 \eqref{eq:martingale_deviation} and \eqref{eq:martingale_deviation_far} are 
 bounds on the probability that the random variable $M^{x,y,T}_{\infty}$  deviates too much from its expectation.
  The reason for the choice of $\delta$ in \eqref{eq:def:epsilon_delta} as $\varepsilon/d$
  will become apparent in the proof of Proposition \ref{prop_near_martingale_deviation}.
  The bounds \eqref{eq:martingale_deviation} are \eqref{eq:martingale_deviation_far} 
   are sufficient for our purposes, but
 we do not claim that they are optimal.
 \end{remark}

\noindent The proof of Proposition \ref{prop_near_martingale_deviation} is postponed to Section \ref{section:kallenberg}.

We now fix $T_0$ and $D_0$ as in Proposition \ref{prop_near_martingale_deviation}.
 Given these choices, we may then assume that the renormalization constant $L$ satisfies
\begin{equation}\label{eq:choice_L_martingale_dev}
T \stackrel{ \eqref{eq:def_T} }{=} L^{2-\varepsilon} \geq T_0, \qquad L \geq D_0.
\end{equation}

We define for $1 \leq k < n \leq |\mathcal{X}|$ the random variables 
\begin{equation}\label{def_eq_V_k_n}
V_{k,n} = 
\begin{cases}
\mathds{1} \left\{ M^{x_k,x_n,T}_{\infty} > T^{1-\frac{d}{2}+\varepsilon} \right\}, & \text{ if } \; \; 
|x_k - x_n|=1 ,\\[.5cm]
\mathds{1} \left\{ M^{x_k,x_n,T}_{\infty} > |x_n- x_k |^{2-d+\varepsilon}\right\} , & \text{ otherwise. }
\end{cases}
\end{equation}
In words: $V_{k,n}$ is the indicator of the event that  $(X^{x_k}_t)_{t \geq 0}$ endangers 
$(X^{x_n}_t)_{t \geq 0}$. We also define
\begin{equation}\label{def_eq_V_k}
 \underline{V}_k=\left( V_{k,k+1},\dots,V_{k,|\mathcal{X}|} \right), \qquad  V_k = \sum_{n=k+1}^{|\mathcal{X}|} V_{k,n}.
 \end{equation}
Now by \eqref{def_M_x_y_T_martingale} and \eqref{def_eq_V_k_n},  for any  $1 \leq k < n \leq |\mathcal{X}|$
\begin{equation}
\label{V_k_n_measurable_wrt_kth_path}
V_{k,n} \quad \text{is measurable with respect to}\quad   \sigma(X^{x_k}_t \,:\, t \ge 0 ),
\end{equation}
 therefore 
\begin{align}
\label{ul_V_1_ul_V_2_etc_indep}
 \underline{V}_1,\dots, \underline{V}_{|\mathcal{X}|} \quad &\text{are independent,}\\
\label{V_1_V_2_etc_indep}
 V_1,\dots, V_{|\mathcal{X}|} \quad &\text{are independent.}
\end{align}
 
Now by \eqref{def_eq_V_k} for any $1 \leq n \leq |\mathcal{X}|$ the random variable $V_n$ is the number of trajectories that the trajectory $(X^{x_n}_t)$ endangers. Our stochastic domination result, Proposition \ref{prop:stoch_dom}, will involve the
total number of paths that either endanger others or are endangered by others; hence, as an intermediate step, in the next lemma we stochastically dominate the random variable $V_n + \mathds{1}\{ V_n>0\}$, which collects the 
$\sigma(X^{x_n}_t \,:\, t \ge 0 )$-measurable terms in the sum that counts the total number
 of paths that either endanger others or are endangered by others.

 \begin{lemma} \label{lem:V_star_dominates}
  If $e^{-T^{\delta}} + \sum_{k=1}^{\infty} 2^k \cdot e^{-(\frac12 \ell^{k} L)^{\delta} } \leq 1$ then
 for any $n \in \{1, \dots, |\mathcal{X}| \}$ the random variable $V_n + \mathds{1}\{ V_n>0\}$ is stochastically  dominated by a random variable $V^*_0$ with probability mass function $\mathsf{p}_{V^*_0}$ supported on the set of integers $\{0\}\cup\{2^k + 1: k \geq 0\}$ and given by
\begin{equation}\label{eq:def_V_0} \mathsf{p}_{V^*_0}(2) =e^{-T^{\delta}}, \quad 
\mathsf{p}_{V^*_0}(2^{k}+1)  =2^k\cdot e^{-(\frac12 \ell^{k} L)^{\delta} }, \; k \geq 1, \quad \mathsf{p}_{V^*_0}(0) = 1-\sum_{k > 0} \mathsf{p}_{V^*_0}(k).
\end{equation}
 In particular,
\begin{equation}
\label{V_star_0_vanishes_V_k_vanishes}
P[ V_n >0] \leq P[V^*_0>0]= 
e^{-T^{\delta}} + \sum_{k=1}^{\infty} 2^k \cdot e^{-(\frac12 \ell^{k} L)^{\delta} }.
\end{equation}
   \end{lemma}

The proof of Lemma \ref{lem:V_star_dominates} is postponed until Section \ref{subsection:proof_of_lemmas_lem:V_star_dominates_and_lem:U_n_rarely_one}.

Recall the definition of $U_n$ from \eqref{def_eq_U_n_U}.
In words, the next lemma states that  if a path is not endangered by any of the previous paths,
 then it is very likely to avoid all of them.

\begin{lemma}
\label{lem:U_n_rarely_one}
For any $n\in \{1,\ldots,|\mathcal{X}|\}$, 
\begin{align}
\label{eq:U_n_rarely_one}
&P\left[U_n = 1\;|\;X^{x_k}_t: 1 \leq k < n,\; t \ge 0\right]\cdot \mathds{1}\left\{\sum_{k=1}^{n-1} V_{k,n} = 0 \right\}\nonumber \\&\hspace{6cm}\leq T^{1-\frac{d}{2}+\varepsilon} + 
 \sum_{k=1}^{\infty} 2^k \left( \frac12 \ell^k L \right)^{2-d +\varepsilon}.
\end{align}
\end{lemma} 

The proof of Lemma \ref{lem:U_n_rarely_one} is postponed until Section \ref{subsection:proof_of_lemmas_lem:V_star_dominates_and_lem:U_n_rarely_one}.

In order to state the following proposition, and for the sake of clarity, we recapitulate some relevant definitions:
\begin{itemize}
\item $U_n$ (for $1 \leq n \leq |\mathcal{X}|$) and $U$ in \eqref{def_eq_U_n_U};
\item $V_{k,n}$ (for $1 \leq k < n \leq |\mathcal{X}|$), $\underline{V}_n$ and $V_n$ (for $1 \leq n \leq |\mathcal{X}|$) in \eqref{def_eq_V_k_n} and \eqref{def_eq_V_k};
\item  $V^*_0$ in Lemma \ref{lem:V_star_dominates}.\end{itemize}
We add to this list one more definition; let 
\begin{equation}\label{def_eq_parameter_of_binomial_p}
 p := p(L, d) = \frac{T^{1-\frac{d}{2}+\varepsilon} + 
 \sum_{k=1}^{\infty} 2^k \left( \frac12 \ell^k L \right)^{2-d +\varepsilon} }{P[V^*_0=0]} .
\end{equation}

\begin{proposition} \label{prop:stoch_dom}
Let $U^* \sim \mathrm{Bin}(|\mathcal{X}|,p)$ and let $V^*$ be independent from
$U^*$, where $V^*$ is the sum of $|\mathcal{X}|$ i.i.d.\ copies of $V^*_0$. Then 
\begin{equation}\label{eq_stoch_dom_U_Ustar_Vstar}
\text{ $U$ is stochastically dominated by 
$U^*+V^*$. }
\end{equation}
\end{proposition}
\begin{remark}
If the $n$\hspace{-.07cm}'th path is not endangered by previous paths then the parameter of the Bernoulli variable $U_n$ is bounded 
by the right-hand side of \eqref{eq:U_n_rarely_one}. The indicators $U_n$, $1 \leq n \leq |\mathcal{X}|$
 are not independent, but we can ``hide'' their correlations
  by slightly increasing the parameters of these indicators (c.f.\ \eqref{eq:U_n_rarely_one} and \eqref{def_eq_parameter_of_binomial_p}) and by adding $V^*$.
  Hence, it can happen that the term of index $n$ contributes to the dominating random variable in \eqref{eq_stoch_dom_U_Ustar_Vstar} even if it ends up being a success, that is, if $U_n = 0$. This justifies the terminology used in the Introduction: we ``throw away'' some paths in order to guarantee independence. 
  Similarly, since we add $V^*$ to the dominating random variable in \eqref{eq_stoch_dom_U_Ustar_Vstar}, we ``throw away'' paths endangered by others and paths which endanger others.
 
  This method resembles the ``sprinkling technique'' which has been successfully applied in the context
of random interlacements (e.g., in \cite[Section 2]{sznitman_decoupling})
 and Gaussian free field (e.g., in \cite[Proposition 2.2]{RS13}).
\end{remark}

The proof of Proposition \ref{prop:stoch_dom} will be carried out in Section \ref{subsection:proof_of_prop:stoch_dom}
using a coupling argument.

\subsection{Proof of Proposition \ref{prop:final_max_bound}}
\label{subsection:completion_of_renorm_proof}

By \eqref{eq:jointc2}  the left-hand side of \eqref{eq:prop_final_max_bound} is a non-decreasing function of $\alpha$, so it is  enough to prove \eqref{eq:prop_final_max_bound} for 
\begin{equation}
\label{eq:assume_alpha}
\alpha = L^{2-d+\frac{1}{4}}.
\end{equation}

\begin{remark}
Let us comment about the choice of $\alpha$. 
For the sake of this heuristic argument let us assume that
 $\mathcal{Y}=\emptyset$ in \eqref{eq:fin_put_together} below, so that 
$|\mathcal{X}|=2 \cdot 2^N$, see \eqref{calX_plus_calY}.
Comparing the combinatorial term $\left(C L^d \right)^{2^N}$ of \eqref{eq:last_union_bound} with
the terms
\[ \alpha^{|\mathcal{X}|} \cdot  e^{|\mathcal{X}|p/\alpha} = \left(\alpha^2 e^{2p/\alpha}\right)^{2^N}
 \]
  in \eqref{eq:good_final_boundEF} below, we
 see that if we want $\mathbb{P}[ B(L_N) \stackrel{*\xi}{\longleftrightarrow} B(2L_N)^c ] \ll 1$ then it is a good idea
  to choose $\alpha$ so that
\begin{equation}\label{alpha_heu_bounds}
  \alpha^2 L^d \ll 1, \qquad p/\alpha = \mathcal{O}(1). 
  \end{equation}
  Now $p$ is not much bigger than $L^{2-d}$ (see \eqref{eq:bounding_order_of_p} below), so if $d\geq 5$,
   then \eqref{eq:assume_alpha} is a good choice if we want $\alpha$ to satisfy the bounds 
   \eqref{alpha_heu_bounds}.
  \end{remark}

 Let us fix $\mathcal{T} \in \Lambda_{N}$ (see Definition \ref{def_proper_embedding_of_trees}) 
 and $(\mathcal{X},\mathcal{Y}) \in \mathcal{P}_{\mathcal{T}} $ (see Definition \ref{def:admissible}).
 We have
\begin{multline}\label{eq:fin_put_together}
 \mathbb{P}\left[
\left(\bigcap_{ \substack{ \{x,z\}\in \mathcal{X}\\ |x- z|=1 }} F_{x,z}\right) \cap 
\left( \bigcap_{y \in \mathcal{Y}} E_y \right) \right]
\stackrel{ \eqref{eq:betaalpha}, \eqref{eq:bound_Zinf0} }{\leq}
\beta^{|\mathcal{Y}|}  E \left[ \alpha^{\mathcal{N}^Z_\infty}\right]
\stackrel{  \eqref{eq:bound_Zinf},  \eqref{eq_stoch_dom_U_Ustar_Vstar} }{\leq}
\\
   \beta^{|\mathcal{Y}|} \cdot \alpha^{|\mathcal{X}|} \cdot E\left[\left(\frac{1}{\alpha}\right)^{U^*}\right] \cdot E\left[\left(\frac{1}{\alpha}\right)^{V^*}\right].
\end{multline}
Now we bound the terms on the right-hand side of \eqref{eq:fin_put_together}.

\begin{equation}
\label{eq:bound_beta}
\beta \stackrel{  \eqref{def_eq_E_x}, \eqref{def_eq_beta} }{=}
\mathbb{P} \left[ \max_{0 \leq t \leq T} | Y^0_t| > \frac{1}{4} L \right]
\stackrel{ \eqref{eq:srw_large_dev_estimate}, \eqref{eq:def_T}   }{\leq}
2d \exp \left(-\frac{1}{8} L \ln \left( 1 + \frac{d}{4} L^{-1+\varepsilon} \right)  \right).
\end{equation}

Recall from Proposition \ref{prop:stoch_dom} that $U^* \sim \mathrm{Bin}(|\mathcal{X}|,p)$, where $p = p(L,d)$ was defined in  \eqref{def_eq_parameter_of_binomial_p}. 
For a random variable $Z \sim \mathrm{Bin}(m,r)$ and $\theta \geq 0$, we have $E[\theta^Z] \leq e^{mr\theta}$, thus
\begin{equation}
\label{eq:fin_moment_U}
E\left[\left(\frac{1}{\alpha}\right)^{U^*}\right] \leq e^{|\mathcal{X}|p/\alpha}.
\end{equation}

Recall from Proposition \ref{prop:stoch_dom} that $V^*$ is the sum of $|\mathcal{X}|$ independent copies of $V^*_0$.

\begin{multline}
\label{eq:V_star_momgen_q_def}
E \left[ \left(\frac{1}{\alpha}\right)^{V^*_0} \right]
\stackrel{ \eqref{eq:def_V_0} }{=}
 \mathsf{p}_{V^*_0}(0) + \frac{e^{-T^{\delta}}}{\alpha^2} + \sum_{k=1}^{\infty} \left(\frac{1}{\alpha}\right)^{2^k+1} 2^k e^{-(\frac12 \ell^{k} L)^{\delta} } 
 \stackrel{\eqref{eq:def_T}, \eqref{eq:assume_alpha} }{=}\\
\mathsf{p}_{V^*_0}(0) +
 L^{2d - \frac{9}{2}} \cdot e^{-L^{(2-\varepsilon)\delta}} + 
\sum_{k=1}^{\infty} 
\exp\left((2^k+1)(d-\frac{9}{4}) \ln(L) + k \ln(2) -\frac{1}{2^{\delta}} (\ell^\delta)^k L^{\delta} \right)
\\=:q \stackrel{(*)}{=}q(L, d),
\end{multline}
where in $(*)$ the parameter $q$ is indeed only a function of $L$ and $d$, because of the definition of  
$\varepsilon$ and $\delta$  in \eqref{eq:def:epsilon_delta} and $\ell$ in \eqref{eq:choice_ell}.
 We can thus bound
\begin{multline}
\label{eq:good_final_boundEF}
\mathbb{P}\left[
\left(\bigcap_{ \substack{ \{x,z\}\in \mathcal{X}\\ |x- z|=1 }} F_{x,z}\right) \cap 
\left( \bigcap_{y \in \mathcal{Y}} E_y \right) \right]
\stackrel{\eqref{eq:fin_put_together}, \eqref{eq:fin_moment_U}, \eqref{eq:V_star_momgen_q_def} }{\leq}
 \beta^{|\mathcal{Y}|} \cdot \alpha^{|\mathcal{X}|}
\cdot  e^{|\mathcal{X}|p/\alpha} 
\cdot q^{|\mathcal{X}|}=\\
 \exp\left\{|\mathcal{Y}|\ln \beta + \frac12|\mathcal{X}|(2\ln q + 2\ln \alpha + 2p/\alpha) \right\} . 
\end{multline}

Recall our definition of $\ell$ from \eqref{eq:choice_ell}.
We will choose $L$ big enough so that it satisfies multiple criteria, as we now discuss. 
By \eqref{eq:choice_L_martingale_dev} we need 
 \begin{equation*}
 L > L_{(1)} := T_0^{\frac{1}{2-\varepsilon}} \vee D_0 .
 \end{equation*}
 
Having already fixed $\varepsilon$, $\delta$ and $\ell$, we  assume that $L$ satisfies
\begin{equation}\label{eq:cond_v_star_zero_geq_half}
L \geq L_{(2)}, \quad \text{so that} \quad  
\exp\left(-L^{(2-\varepsilon)\delta}\right) + 
\sum_{k=1}^{\infty} 2^k \cdot \exp\left(-(\frac12 \ell^{k} L )^{\delta} \right) \leq \frac{1}{2},
\end{equation}
 so that the condition of Lemma \ref{lem:V_star_dominates} is satisfied for $L$. 
 We will also assume
\begin{equation}\label{eq:condition_on_q}
 L \ge L_{(3)}, \quad \text{so that} \quad  q(L,d) \stackrel{(**)}{\leq} 2.
\end{equation}
The inequality $(**)$ can be achieved because $p_{V_0^*}(0) \leq 1$  (see \eqref{eq:def_V_0}) and
by our choice of $\ell$ in \eqref{eq:choice_ell} we have $(\ell^\delta)^k = 3^k$, thus
  the sum of the other terms in the definition \eqref{eq:V_star_momgen_q_def} of 
  $q$ can be made arbitrarily small by making $L$ large.
Next we will show that
\begin{equation}\label{eq:condition_on_q2} 
 p(L,d) \leq 4 L^{2-d+\frac{1}{4}}\stackrel{ \eqref{eq:assume_alpha} }{=}4\alpha 
\quad \text{if} \quad 
   L \geq L_{(2)}.
\end{equation}
To show that this inequality indeed holds, we estimate
\begin{multline}
\label{eq:bounding_order_of_p}
p= p(L,d) 
\stackrel{\eqref{eq:def_T},\eqref{def_eq_parameter_of_binomial_p}, }{=} 
\frac{L^{(2-\varepsilon)(1-\frac{d}{2}+\varepsilon)} + 
 \sum_{k=1}^{\infty} 2^k \left( \frac12 \ell^k L \right)^{2-d +\varepsilon} }{P[V^*_0=0]} \\
  \stackrel{\eqref{eq:def:epsilon_delta} }{\leq} 
  \frac{L^{2-d  + \frac14 } + 
L^{2-d +\varepsilon} \sum_{k=1}^{\infty} 2^k \left( \frac12 \ell^k \right)^{2-d +\varepsilon} }{P[V^*_0=0]}
\stackrel{ \eqref{eq:choice_ell}, \eqref{V_star_0_vanishes_V_k_vanishes}, \eqref{eq:cond_v_star_zero_geq_half} }{\leq} \\
2 \left( L^{2-d +\frac14} + L^{2-d+\varepsilon} \right) 
\stackrel{\eqref{eq:def:epsilon_delta} }{\leq}
4 L^{2-d+\frac14 }.
 \end{multline}
 
We can now bound the expression in the exponential in the right-hand side of \eqref{eq:good_final_boundEF}:
\begin{multline}
\label{eq:good_final_bound2}
 |\mathcal{Y}|\ln(\beta) + 
\frac{|\mathcal{X}|}{2}(2\ln q + \frac{p}{\alpha}+ 2\ln \alpha)
\stackrel{ \eqref{eq:condition_on_q}, \eqref{eq:condition_on_q2}, }{\leq}
|\mathcal{Y}|\ln(\beta) + 
\frac{|\mathcal{X}|}{2}(2\ln 2 + 4+ 2\ln \alpha)
\stackrel{ \eqref{eq:assume_alpha}, \eqref{eq:bound_beta}}{\leq}
 \\
 |\mathcal{Y}|
 \left( \ln(2d) -\frac{1}{8} L \ln \left( 1 + \frac{d}{4} L^{-1+\varepsilon} \right) \right) 
   + 
 \frac{|\mathcal{X}|}{2} \left(\widehat{\mathcal{C}} +\left( 4-2d+\frac12 \right)\ln L\right)
 \stackrel{(*)}{\leq}
 \\
 \left(|\mathcal{Y}|+\frac{|\mathcal{X}|}{2}\right)
 \left( \widehat{\mathcal{C}} + \left( 4-2d+\frac12 \right)\ln L\right)
\stackrel{\eqref{calX_plus_calY}}{=} 
  2^N \cdot \left( \widehat{\mathcal{C}} + \left( 4-2d+\frac12 \right)\ln L\right),
\end{multline}
where  $(*)$ holds for $L \geq L_{(4)}$.
Plugging \eqref{eq:good_final_bound2} back in \eqref{eq:good_final_boundEF}, we obtain that the statement of
Proposition \ref{prop:final_max_bound} holds 
with $\ell$ as in \eqref{eq:choice_ell} and 
$L_{(0)}:=  L_{(1)} \vee L_{(2)} \vee L_{(3)} \vee L_{(4)}.$


\subsection{Proof of Lemmas \ref{lem:V_star_dominates} and \ref{lem:U_n_rarely_one}}
\label{subsection:proof_of_lemmas_lem:V_star_dominates_and_lem:U_n_rarely_one}

We now prove the two lemmas of Section \ref{section:main_stochastic_domination} bounding the probability
that random walk paths endanger (Lemma \ref{lem:V_star_dominates})
 and intersect (Lemma \ref{lem:U_n_rarely_one}) each other. These proofs simply put together results that  have already been  established.  For Lemma \ref{lem:V_star_dominates}, we combine Proposition \ref{prop_near_martingale_deviation} -- which bounds the probability that a path endangers another path that starts at a given distance from it -- with \eqref{eq:count_neighb_X} -- which bounds the number of points of $\mathcal{X}$ that are within a given distance from a fixed point $x \in \mathcal{X}$. Lemma \ref{lem:U_n_rarely_one} is even simpler and follows from a combination of
  \eqref{eq:count_neighb_X} with the definition of ``endangering'' in \eqref{def_eq_V_k_n}.

\begin{proof}[Proof of Lemma \ref{lem:V_star_dominates}]
Fix $n \in \{1,\ldots, |\mathcal{X}|\}$.
 We take a bijection 
 $$\theta: \{0, 1, \ldots, |\mathcal{X}| - n\} \to \{n, n+1, \ldots, |\mathcal{X}|\}$$ 
 with the property that
  $$0 = |x_{\theta(0)} - x_n| \leq |x_{\theta(1)} - x_n| \leq \cdots \leq |x_{\theta(|\mathcal{X}|-n)} - x_n|.$$ 
We have
$$|\{i \geq n: |x_i - x_n| \leq \ell^k L/2 \}| \leq |\mathcal{X} \cap B(x_n,\ell^kL/2 )| 
\stackrel{ \eqref{eq:count_neighb_X} }{\leq} 2^k,\qquad k \geq 1,$$
so that 
\begin{equation*}
 |x_{\theta(i)} - x_n| > \ell^k L/2 \qquad \text{for all } i \geq 2^k,\;k\ge 1. \label{eq:cond_more_n}
\end{equation*}
By our definition of $\ell$ (see  \eqref{eq:def:epsilon_delta},\eqref{eq:choice_ell})
and $L$ (see \eqref{eq:choice_L_martingale_dev}) we have $\ell^kL/2 \geq D_0$, for any $k \geq 1$, moreover
$ T \geq T_0$ (see \eqref{eq:choice_L_martingale_dev}),
 therefore we can use  Proposition \ref{prop_near_martingale_deviation}
 to bound the probability of the event in the indicator $V_{n,\theta(i)}$ (see \eqref{def_eq_V_k_n})   that trajectory $n$ endangers trajectory $\theta(i)$:
 
\begin{equation} 
\label{eq:ellL3}
P\left[V_{n,\theta(i)} = 1\right] \leq \left\{\begin{array}{ll}e^{-T^\delta}&\text{if } i = 1;\\e^{-\left(\ell^kL/2 \right)^\delta}&\text{if } i \geq 2^k,\; k \geq 1. \end{array} \right. \end{equation}

Now, if $i \geq 2$, we have
$$\begin{aligned}
P[V_n \geq i] \leq P[V_n \geq 2^{\lfloor \log_2 i \rfloor}] 
\leq 
\sum_{j\geq 2^{\lfloor \log_2 i \rfloor}} P[V_{n,\theta(j)} = 1] 
\stackrel{ \eqref{eq:ellL3} }{\leq}
 \sum_{k={\lfloor \log_2 i \rfloor}}^\infty 2^k \cdot e^{-\left(\ell^k L/2 \right)^\delta}
\end{aligned}$$ 
and similarly,
$$P[V_n \geq 1] \leq e^{-T^\delta} + \sum_{k=1}^\infty 2^k \cdot e^{-\left(\ell^k L/2 \right)^\delta}.$$

We then obtain
$$\begin{aligned}
&P[V_n + \mathds{1}\{V_n > 0\} \geq 1] = P[V_n + \mathds{1}\{V_n > 0\} \geq 2] \leq e^{-T^\delta} + \sum_{k=1}^\infty 2^k \cdot e^{-\left(\ell^k L/2 \right)^\delta}
\end{aligned}$$
and, for $i > 2$,
$$
P[V_n + \mathds{1}\{V_n > 0\} \geq i] \leq P[V_n \geq i-1] \leq \sum_{k={\lfloor \log_2 (i-1) \rfloor}}^\infty 2^k \cdot e^{-\left(\ell^k L/2 \right)^\delta}.
$$
The statement of the lemma now follows from comparing these inequalities with the definition of the law of
$V_0^*$ in \eqref{eq:def_V_0}.
\end{proof}

\begin{proof}[Proof of Lemma \ref{lem:U_n_rarely_one}]
We have
$$\begin{aligned}
&P\left[ U_n=1 | \,  X^{x_k}_t: 1 \leq k < n,\; t \ge 0 \right]
\\
&\hspace{2.5cm}\stackrel{\eqref{def_eq_U_n_U} }{\leq} \sum_{m=1}^{n-1} P\left[  \exists  s \geq T \, : \,  X^{x_n}_s=X^{x_m}_s | \,  X^{x_k}_t: 1 \leq k < n,\; t \ge 0 \right]\\
&\hspace{2.5cm} =\sum_{m=1}^{n-1} P\left[  \exists  s \geq T \, : \,  X^{x_n}_s=X^{x_m}_s | \,  X^{x_m}_t: \;t \ge 0 \right]
\stackrel{\eqref{def_M_x_y_T_martingale}}{=} 
\sum_{m=1}^{n-1}  M^{x_m,x_n,T}_{\infty},
\end{aligned}$$
so that
$$P\left[U_n = 1\;|\;X^{x_k}_t: 1 \leq k < n,\; t \ge 0\right]\cdot \mathds{1}\left\{\sum_{k=1}^{n-1} V_{k,n} = 0 \right\}
 \leq 
 \sum_{m=1}^{n-1} M^{x_m,x_n,T}_\infty \cdot \mathds{1}\left\{V_{m,n} = 0 \right\}.$$
Now, by \eqref{def_eq_V_k_n},
$$ M^{x_m,x_n,T}_{\infty} \cdot \mathds{1}\left\{V_{m,n} = 0 \right\} 
\leq 
\left\{ \begin{array}{ll}T^{1-\frac{d}{2}+\varepsilon} &\text{if } \; \; \; |x_m - x_n|=1;
\vspace{.3cm}\\|x_m - x_n|^{2-d+\varepsilon}&\text{otherwise}. \end{array} \right.$$

The proof of \eqref{eq:U_n_rarely_one} can now be completed by applying Definition \ref{def:admissible}
and  (\ref{eq:count_neighb_X}) 
as we did in the proof of Lemma \ref{lem:V_star_dominates}; we omit the details.

\end{proof}

\subsection{Proof of Proposition \ref{prop:stoch_dom}}
\label{subsection:proof_of_prop:stoch_dom}

In this section we will prove our stochastic domination result using a coupling argument.
 The key idea lies in the definition of some auxiliary random variables $U_n^*$, $1 \leq n \leq |\mathcal{X}|$, so let us start by explaining this informally (the precise definition is given in \eqref{eq:defu*}). Define the events
\begin{equation}
\label{eq:def_events_A} 
 A_n = \left\{ \sum_{k=1}^{n-1} V_{k,n} =0,\; V_{n}=0 \right\},\quad
 1 \leq n \le |\mathcal{X}|.\end{equation}
In words: $A_n$ is the event that the $n$'th random walk path is not endangered by previous paths and does not endanger upcoming paths. We will specify the key properties of $U^*_n, 1 \leq n \leq |\mathcal{X}|$ using the events $A_n$ in
\eqref{U_star_n_dominates_U_n} and \eqref{U_star_n_iid_Ber_p} below.
Suppose we fix $n$ and we reveal all the paths $\{X^{x_i}_t: i < n,\; t \geq 0\}$, and moreover we reveal the vector $\underline{V}_n$ (defined in \eqref{def_eq_V_k}).
 Given all this information, we are able to determine whether or not $A_n$ has occurred. Now,
\begin{itemize}
\item[(a)] assume $A_n$ has occurred. At this point, we have full knowledge of all the paths with index smaller than $n$, and also some partial knowledge of the $n$'th path: we know $\underline{V}_n$, in fact
 we know that $A_n$ occurred, which implies $V_n = \sum_{k=n+1}^{|\mathcal{X}|}V_{n,k} = 0$.
  In Lemma \ref{lemma:parameter_ineq}, we argue that the conditional probability of $\{U_n=1\}$ given all this information is at most $p$
  (defined in \eqref{def_eq_parameter_of_binomial_p}), a number that is not much larger than the bound we had given in \eqref{eq:U_n_rarely_one} (which did not include the conditioning on $\{V_n = 0\}$). We are thus able to define $U_n^*$ so that $U_n \leq U_n^*$ and $U_n^* \sim \mathrm{Bernoulli}(p)$.
\item[(b)] if $A_n$ has not occurred, we simply prescribe (using extra, auxiliary randomness) that $U_n^*$ is $\mathrm{Bernoulli}(p)$.
\end{itemize}
The sum $\sum_{n=1}^{|\mathcal{X}|} U_n \cdot \mathds{1}_{A_n}$ is then dominated by $\sum_{n=1}^{|\mathcal{X}|} U_n^* \cdot \mathds{1}_{A_n} \leq \sum_{n=1}^{|\mathcal{X}|} U_n^* $, and the sum $\sum_{n=1}^{|\mathcal{X}|} U_n \cdot \mathds{1}_{A_n^c}$ is dominated by $\sum_{n=1}^{|\mathcal{X}|} (V_n + \mathds{1}\{V_n > 0\})$ (see \eqref{eq:final_arg_domination} below),
 which in turn is dominated by $\sum_{n=1}^{|\mathcal{X}|} V_n^*$, a sum of i.i.d.\ random variables distributed
  as $V^*_0$ from Lemma \ref{lem:V_star_dominates}. Finally, the desired independence properties of our construction follow from the fact that the distribution of $U_n^*$ is the same regardless of the conditioning; this is formalized in Lemma \ref{lemma_u_star_iid_dominates}.

\begin{proof}[Proof of Proposition \ref{prop:stoch_dom}]

In a series of lemmas we will construct,
 by extending the probability space of the random walks $(X^{x_n}_t), 1 \leq n \leq |\mathcal{X}|$,
  random variables $U^*_1, \dots, U^*_{|\mathcal{X}|}$ satisfying
\begin{eqnarray}
\label{U_star_n_dominates_U_n} &&U_n \cdot \mathds{1}_{A_n} \leq U^*_n \cdot \mathds{1}_{A_n};\\
 \label{U_star_n_iid_Ber_p} &&U^*_n \sim \mathrm{Ber}( p  )\text{ and is independent of }\left( (U^*_k)_{1 \leq k \leq n-1}, (\underline{V}_{k})_{1\leq k  \leq |\mathcal{X}|}\right).\end{eqnarray}
Here we show how this construction implies  \eqref{eq_stoch_dom_U_Ustar_Vstar}.
 We let $U^* = \sum_{n=1}^{|\mathcal{X}|} U^*_n$. We have
\begin{multline}
U 
\stackrel{ \eqref{def_eq_U_n_U} }{=} 
\sum_{n=1}^{|\mathcal{X}|} U_n 
\stackrel{ \eqref{U_star_n_dominates_U_n} }{\leq}
 \sum_{n=1}^{|\mathcal{X}|} (U^*_n +\mathds{1}_{A_{n}^c})
 \stackrel{ \eqref{eq:def_events_A}  }{\leq} 
 U^* + 
\sum_{n=1}^{|\mathcal{X}|} \left( \mathds{1}_{\{V_n>0\}} + \sum_{k=1}^{n-1} V_{k,n} \right)\\
=
U^*+
\sum_{n=1}^{|\mathcal{X}|}  \mathds{1}_{\{V_n>0\}} + \sum_{k=1}^{|\mathcal{X}|} \sum_{n=k+1}^{|\mathcal{X}|} V_{k,n}  
\stackrel{ \eqref{def_eq_V_k} }{=}
U^*+ \sum_{k=1}^{|\mathcal{X}|} (V_k + \mathds{1}_{\{V_k>0\}}).\label{eq:final_arg_domination}
\end{multline}
 Now (\ref{U_star_n_iid_Ber_p}) implies that $U^* \sim \mathrm{Bin}(|\mathcal{X}|, p)$ 
 and is independent of $V_1, \ldots, V_{|\mathcal{X}|}$, which are also independent by \eqref{V_1_V_2_etc_indep}. 
Putting this together with Lemma \ref{lem:V_star_dominates} 
 we obtain that 
 \[ \text{ $U^* + \sum_{n=1}^{|\mathcal{X}|} (\mathds{1}_{\{V_n > 0\}} + V_n)$ is stochastically dominated by
  $U^* + V^*$,} \] 
where $V^*$ is a sum  of $|\mathcal{X}|$ independent copies of $V_0^*$. This completes the proof of
Proposition \ref{prop:stoch_dom} given (\ref{U_star_n_dominates_U_n}) and (\ref{U_star_n_iid_Ber_p}).
\end{proof}

The rest of this subsection is devoted to the construction of random variables $U^*_1,\ldots, U^*_{|\mathcal{X}|}$ satisfying (\ref{U_star_n_dominates_U_n}) and (\ref{U_star_n_iid_Ber_p}).
We start recalling a few standard facts about conditional expectations.

\begin{lemma}\label{lem:cond_exp_prop}
Let $(\Omega, \mathcal{F}, \mathsf{P})$ be a probability space.
\begin{enumerate}
\item \cite[Section 9.7, Property (k)]{williams} If $X$ is an $\mathcal{F}$-measurable and bounded random variable (r.v.),
 $\mathcal{G},\mathcal{G}' \subset \mathcal{F}$ are sigma-algebras  and $\mathcal{G}'$ is independent of $\sigma(\mathcal{G} \cup \sigma(X))$, then
\begin{equation}\label{throw_away_independent_sigma_algebra}
\mathsf{E}[X\,|\,\mathcal{G}, \mathcal{G}'] = \mathsf{E}[X\,|\, \mathcal{G}].
\end{equation}
\item \cite[Theorem 2.24]{klebaner} 
Let $\mathcal{H} \subseteq \mathcal{F}$ be a sigma-algebra, $Z$ be an  $\mathcal{F}$-measurable r.v.\ independent of $\mathcal{H}$, $Y$ be an $\mathcal{H}$-measurable r.v.,   and $f: \R^2 \to \R$ be Borel-measurable and bounded. 
 If we define $g(y):= \mathsf{E}[f(y,Z)]$ for $y \in \R$  then
\begin{equation}\label{averaging_out_the_independent_stuff}
\mathsf{E}\left[f(Y,Z)\,|\,\mathcal{H}\right] = g(Y).
\end{equation}
\end{enumerate}
\end{lemma}

We now extend the probability space of the walks $X^x_t, t \geq 0, x \in \mathcal{X}$ with an independent collection of
auxiliary random variables 
\begin{equation}\label{def_eq_zeta}
\zeta_k,\; 1 \leq k \leq |\mathcal{X}|, \qquad \text{i.i.d.\ and unifomly distributed on }\; [0,1].
\end{equation}

 For $1 \leq n \leq |\mathcal{X}|$, we introduce the sigma-field
\begin{align}
\nonumber \sigma_{n}:=\; &\sigma\left( \left(
\zeta_k: 1 \leq k \leq n\right),\;\left(\underline{V}_k: 1\leq k  \leq |\mathcal{X}|\right),\;\left( X^{x_k}_t: t \geq 0,\;1\leq  k \leq n\right) \right)\\
\stackrel{\eqref{def_eq_V_k}, \eqref{V_k_n_measurable_wrt_kth_path}}{=} 
&\sigma\left( \left(
\zeta_k: 1 \leq k \leq n\right),\;\left( \underline{V}_k: n< k  \leq |\mathcal{X}|\right),\;\left( X^{x_k}_t: t \geq 0,\;1\leq  k \leq n\right) \right). \label{eq:sigma_sec_def}
\end{align}

Recalling \eqref{def_eq_U_n_U} we  also define the random variable
\begin{equation}
p_n := P\left[ U_n = 1\;|\;\sigma_{n-1}\right] = 
P\left[ U_n = 1\;\left|\;\begin{array}{c} (\zeta_k)_{1\leq k \leq n-1},\;(\underline{V}_{k})_{n\leq k  \leq |\mathcal{X}|}, \\[0.2cm] 
 X^{x_k}_t, t \geq 0, 1\leq k \leq n-1  \end{array}\right.\right].
 \label{eq:def_pn}
 \end{equation}

\begin{lemma}\label{lemma:parameter_ineq}
 The number $p$ defined in (\ref{def_eq_parameter_of_binomial_p}), the event $A_n$, 
$1 \leq n \leq |\mathcal{X}| $
 defined in \eqref{eq:def_events_A} and the random variable $p_n$ defined in (\ref{eq:def_pn})  satisfy
\begin{equation}\label{cond_prob_of_U_n_in_construction}
p_n \cdot  \mathds{1}_{A_{n} }\leq p.
\end{equation}
\end{lemma}

\begin{proof} 
By \eqref{def_eq_U_n_U}, \eqref{V_k_n_measurable_wrt_kth_path},
  \eqref{throw_away_independent_sigma_algebra} 
 and the fact  that $\left( X^{x_k}_t \right), 1 \leq k \leq |\mathcal{X}|$ are independent random walks, 
   we have
$$p_n = P\left[ U_n=1 \, \left| \, 
 \underline{V}_{n}, \;\;
 X^{x_k}_t, t \geq 0,\;  k \leq n-1  
\right.\right].$$
Recall from \eqref{def_eq_V_k} that $V_n$ is the indicator of the event that
 the path $(X^{x_n}_t)$ does not endanger any upcoming paths.
We now claim that
\begin{equation}
\label{eq:claim_p_n} 
p_n \cdot \mathds{1}\left\{V_n = 0 \right\} = \frac{P\left[ U_n=1,\;V_n = 0 \, \left| \, 
 X^{x_k}_t, t \geq 0,\; k \leq n-1  
\right.\right]}{P\left[ V_n=0 \, \left| \, 
 X^{x_k}_t, t \geq 0,\; k \leq n-1  
\right.\right]}\cdot \mathds{1}\left\{V_n = 0 \right\}.
\end{equation}
Before we prove this, let us see how it allows us to conclude. Noting that
 \begin{equation}
 \label{drop_prev_paths_indep}
 P\left[ V_n=0 \, \left| \, 
 X^{x_k}_t, t \geq 0,\; k \leq n-1  
\right.\right]
\stackrel{\eqref{V_k_n_measurable_wrt_kth_path}, \eqref{throw_away_independent_sigma_algebra} }{=}
 P\left[ V_n=0 \right],
 \end{equation}
 we have
\begin{multline*}
p_n \cdot  \mathds{1}_{A_{n} }
\stackrel{ \eqref{eq:def_events_A}  }{=}
p_n \cdot 
\mathds{1}\left\{V_n = 0 \right\}  \mathds{1}\left\{\sum_{k=1}^{n-1}V_{k,n} = 0 \right\} \\
\stackrel{ \eqref{eq:claim_p_n}, \eqref{drop_prev_paths_indep} }{\leq}
\frac{P\left[ U_n=1 \, \left| \, 
 X^{x_k}_t, t \geq 0,\; k \leq n-1  
\right.\right]\cdot
 \mathds{1}\left\{\sum_{k=1}^{n-1}V_{k,n} = 0 \right\}}{P[ V_n = 0]}\cdot \mathds{1}\left\{V_n = 0 \right\}.
\end{multline*}
By applying Lemma \ref{lem:U_n_rarely_one} to the numerator and Lemma \ref{lem:V_star_dominates} to the denominator, we conclude that the right-hand side is smaller than $p$ (see \eqref{def_eq_parameter_of_binomial_p}), thus \eqref{cond_prob_of_U_n_in_construction} holds.

\medskip

It remains to prove \eqref{eq:claim_p_n}. To this end, we abbreviate
$$\mathcal{V} = \sigma(\underline{V}_{n}),\qquad \mathcal{G} = \sigma(X^{x_k}_t:\;t \ge 0,\; 1 \leq k \leq  n-1),$$
thus $p_n=P[U_n = 1|\mathcal{G},\mathcal{V}]$, so
we must then prove that
\begin{equation}\label{rewritten_claim_p_n}
\mathds{1}_{\{V_n = 0\}} \cdot P[V_n = 0|\mathcal{G}] \cdot P[U_n = 1|\mathcal{G},\mathcal{V}] = \mathds{1}_{\{V_n = 0\}} \cdot P[U_n = 1,\;V_n = 0|\mathcal{G}].
\end{equation}
Since $V_n$ is $\mathcal{V}$-measurable and $P[V_n = 0|\mathcal{G}]$ is $\mathcal{G}$-measurable, 
\eqref{rewritten_claim_p_n} is the same as
\begin{equation} \label{eq:cond_claim1}E\left[\mathds{1}_{\{U_n = 1,\;V_n = 0\}} \cdot P[V_n = 0|\mathcal{G}]\;\big|\;\mathcal{G},\mathcal{V}\right] = \mathds{1}_{\{V_n = 0\} }\cdot P[U_n = 1,\;V_n = 0|\mathcal{G}].\end{equation}
We now check that the right-hand side of \eqref{eq:cond_claim1} satisfies the definition of the left-hand side. First, note that $\mathds{1}_{\{V_n = 0\}} \cdot P[U_n = 1,\;V_n = 0|\mathcal{G}]$ is measurable with respect to $\sigma(\mathcal{G},\mathcal{V})$. 
Second, for any event $C \in \sigma(\mathcal{G},\mathcal{V})$, we must check that
\begin{equation} 
\label{eq:cond_claim2}
E\left[\mathds{1}_C \cdot \mathds{1}_{\{U_n = 1,\;V_n = 0\}} \cdot P[V_n = 0|\mathcal{G}]\right] 
= E[\mathds{1}_C \cdot \mathds{1}_{\{V_n = 0\}} \cdot P[U_n = 1,\;V_n = 0|\mathcal{G}]].
\end{equation}
Now $\mathcal{V}$ is an atomic sigma-algebra (since it is generated by finitely many events, see \eqref{def_eq_V_k})
 and $\{V_n=0\}$ is an atom of $\mathcal{V}$, therefore
  the event $C \cap \{V_n = 0\}$ is equal to $G \cap \{V_n = 0\}$ for some $G \in \mathcal{G}$. 
  Using this, \eqref{eq:cond_claim2} is equivalent to  
$$E\left[\mathds{1}_G \cdot \mathds{1}_{\{U_n = 1,\;V_n = 0\}} \cdot P[V_n = 0|\mathcal{G}]\right] = E[ \mathds{1}_{\{V_n = 0\}} \cdot E[\mathds{1}_{G} \cdot \mathds{1}_{\{U_n = 1,\;V_n = 0\}}|\mathcal{G}]].$$
By taking $E[\,\cdot\,|\mathcal{G}]$ inside the expectation, we see that both sides are equal to
$$E\left[E[\mathds{1}_G \cdot \mathds{1}_{\{U_n = 1,\;V_n = 0\}}|\mathcal{G}] \cdot P[V_n = 0|\mathcal{G}]\right] .$$
The proof of Lemma \ref{lemma:parameter_ineq} is complete.
\end{proof}

We are now ready to define
\begin{equation}\label{eq:defu*}
U^*_n := \mathds{1}_{A_{n}}\cdot \left( U_n+ 
(1-U_n) \cdot \mathds{1}\left\{ \zeta_n \leq \frac{p-p_n}{1-p_n}  \right\} \right) + 
\mathds{1}_{A^c_{n}} \cdot \mathds{1}\{\zeta_n \leq p \}.
\end{equation}

\begin{lemma}\label{lemma_u_star_iid_dominates}
$U^*_n$ satisfies \eqref{U_star_n_dominates_U_n} and \eqref{U_star_n_iid_Ber_p}.
\end{lemma}

\begin{proof}
That \eqref{U_star_n_dominates_U_n} is satisfied is obvious, so we turn to \eqref{U_star_n_iid_Ber_p}.

Recalling the definitions of $U_n$ from \eqref{def_eq_U_n_U},
$\underline{V}_{k}$ from \eqref{def_eq_V_k}, $A_n$ from \eqref{eq:def_events_A},
$\zeta_k$ from \eqref{def_eq_zeta}, $\sigma_n$ from \eqref{eq:sigma_sec_def} and $p_n$ from \eqref{eq:def_pn}
we note that
\begin{equation*}
\text{ $A_n,\;U_n,\; \zeta_n$ and $p_n$ are all $\sigma_n$-measurable,  $1 \leq n \leq |\mathcal{X}|$. }
\end{equation*}
 
  Consequently, $U^*_1,\dots, U^*_{n-1}$ are all $\sigma_{n-1}$-measurable.
Since $(\underline{V}_{k})_{n\leq k  \leq |\mathcal{X}|}$ are also   $\sigma_{n-1}$-measurable,
 we see that \eqref{U_star_n_iid_Ber_p} will follow once we show that  
   \begin{equation}
   \label{condexp_of_u_star_n_is_p}
   \mathbb{E}[U^*_n\, | \, \sigma_{n-1}]=p.
\end{equation} 
We start with  
$$\begin{aligned}
&E\left[U^*_n\;|\;\sigma_{n-1}\right] \stackrel{ \eqref{eq:defu*} }{=}
 E\left[U_n \cdot \mathds{1}_{A_{n}}\;|\;\sigma_{n-1}\right]+
  E\left[\left. (1-U_n)\cdot \mathds{1}_{ A_{n}} \cdot \mathds{1}_{\left\{\zeta_n \leq \frac{p-p_n}{1-p_n}\right\}}\;\right|\;\sigma_{n-1}\right]  \\
& \qquad\qquad\qquad\qquad\qquad\qquad\qquad\qquad\qquad\qquad\qquad\qquad\qquad+ P\left[A_{n}^c \cap \{\zeta_n \leq p\}\;|\;\sigma_{n-1}\right]\\
& 
\stackrel{(*)}{=}  
\mathds{1}_{A_{n}} \cdot E\left[U_n\;|\;\sigma_{n-1}\right] + \mathds{1}_{A_{n}} \cdot E\left[\left. (1-U_n) \cdot \mathds{1}_{\{ \zeta_n \leq \frac{p - p_n}{1-p_n}\}}\right|\;\sigma_{n-1}\right] 
  + \mathds{1}_{A_{n}^c} \cdot P\left[\zeta_n \leq p\;|\;\sigma_{n-1}\right] \\
&
\stackrel{ \eqref{eq:def_pn} }{=}
 \mathds{1}_{A_{n}} \cdot p_n + \mathds{1}_{A_{n}} \cdot E\left[\left. (1-U_n) \cdot \mathds{1}_{\{ \zeta_n \leq \frac{p - p_n}{1-p_n}\}}\right|\;\sigma_{n-1}\right] + \mathds{1}_{A_{n}^c} \cdot p,
\end{aligned}$$
where in $(*)$ we used that $A_n \in \sigma_{n-1}$.
The proof of \eqref{condexp_of_u_star_n_is_p} will be complete once we show 
\begin{equation} \label{eq:missing}
 \mathds{1}_{A_{n}} \cdot E\left[\left. (1-U_n) \cdot \mathds{1}_{\{ \zeta_n \leq \frac{p - p_n}{1-p_n}\}}\right|\;\sigma_{n-1}\right] = \mathds{1}_{A_{n}} \cdot (p-p_n).
\end{equation}
To this end, we first calculate
\begin{align} \label{eq:final_coup1}
& \mathds{1}_{A_{n}} \cdot E\left[\left.(1-U_n) \cdot \mathds{1}_{\{\zeta_n \leq \frac{p-p_n}{1-p_n}\}} \right|\; \sigma_{n-1},\; (X^{x_n}_t)_{t \geq 0} \right]  \\
    \nonumber&=\mathds{1}_{A_{n}} \cdot (1-U_n) \cdot E\left[\left. \mathds{1}_{\{\zeta_n \leq \frac{p-p_n}{1-p_n}\}} \right|\; \sigma_{n-1},\; (X^{x_n}_t)_{t \geq 0} \right] \\&
    \stackrel{(**)}{=}
     \mathds{1}_{A_{n}} \cdot (1-U_n) \cdot \frac{p-p_n}{1-p_n},\label{eq:final_coup2}
\end{align}
where $(**)$ follows from \eqref{cond_prob_of_U_n_in_construction} and \eqref{averaging_out_the_independent_stuff},
  which can be applied because $p_n$ is $\sigma_{n-1}$-measurable and $\zeta_n$ is independent of 
  $\sigma(\sigma_{n-1}, (X^{x_n}_t)_{t \geq 0} )$.
  
 To conclude the proof of \eqref{eq:missing}, note that taking $E[\;\cdot\;|\sigma_{n-1}]$ on \eqref{eq:final_coup1} (and again using the fact that $A_n \in \sigma_{n-1}$) gives the left-hand side of \eqref{eq:missing}, whereas taking $E[\;\cdot\;|\sigma_{n-1}]$ on \eqref{eq:final_coup2} (and using \eqref{eq:def_pn}) gives the right-hand side of \eqref{eq:missing}.
The proof of \eqref{condexp_of_u_star_n_is_p} and Lemma \ref{lemma_u_star_iid_dominates} is complete.
\end{proof}


\subsection{Proof of Proposition \ref{prop_near_martingale_deviation} }
\label{section:kallenberg}

The goal of this section is to prove Proposition \ref{prop_near_martingale_deviation}. Recall the definition of $M^{x,y,T}_\infty$ from  \eqref{def_M_x_y_T_martingale}. We generalize this definition by setting, for any $t \in [T,\infty)$,
\begin{equation}\label{def_M_t_simple}
M^{x,y,T}_t = P [ \exists \, u \geq T \; : \; X^y_u=X^x_u \; |\; \mathcal{F}^x_t ], \qquad
\mathcal{F}^x_t=\sigma \left( X^x_u \,:\, 0 \leq u \leq t \right).
\end{equation}
This defines a martingale indexed by $t\in[T,\infty]$. In order to simplify notation, we will omit the superscripts that indicate dependence on $x$, $y$ and $T$.

Let us now outline the strategy of proof of \eqref{eq:martingale_deviation} 
(the proof of \eqref{eq:martingale_deviation_far} will follow as a corollary).
As suggested in Remark \ref{remark_martingale_green}, we have 
$E[M_\infty ]   \leq C T^{1-d/2}$. In fact we have
$M_T \leq C_0 T^{1-d/2 }$ for some deterministic constant $C_0$ (see \eqref{eq:boundC0}),
 because $X^y$ walks independently of $X^x$, so the conditional probability that they meet after $T$ 
 given any possible outcome of $X^x_T$
 is bounded by
 $C_0 T^{1-d/2}$ . Given this bound on $M_T$,
  the event $\{M_\infty > T^{1-\frac{d}{2} + \varepsilon}\}$
  can only occur if the terminal value $M_\infty$ of the martingale deviates too much from $M_T$. 
  This is where Theorem \ref{theorem_kallenberg} comes into play. 
In order to apply this theorem, we will obtain estimates on the size of the jumps of $(M_{t})$ for $t \geq T$ and
 on its predictable quadratic variation $\langle M \rangle_{\infty}-\langle M \rangle_{T}$;
  these estimates are given in  \eqref{bound_on_biggest_jump} and \eqref{bound_on_total_quad_var}. 
 We derive these estimates by first giving a useful equivalent definition of $M_t$ in 
 Claim \ref{lemma_description_M_t} and then
 comparing $M_t$ with  $M_t^{(e)}$, which arises from $M_t$ by artificially forcing the walk
  $(X^x_{s})_{s \geq 0}$ to jump at time $t$ in the direction of the unit vector $e \sim 0$, see Definition \ref{def:M_e_force_jump}.
 Specifically, in Lemma \ref{lemma_corrector_explicit} we show that the jumps of $M$ can be bounded in terms of 
 $|M_{t}-M^{(e)}_{t}|$ and the predictable quadratic variation $\langle M \rangle_{\infty}-\langle M \rangle_{T}$
 can be expressed as an integral of
 $(M_{t} - M^{(e)}_{t})^2$.
  The difference $|M_{t} - M^{(e)}_{t}|$ is bounded in Lemma \ref{lemma_bound_on_jump_at_time_t} 
  using the random walk facts of Section \ref{subsection_random_walk_facts}.

\medskip

Recall that $(M_t)$ is c\`adl\`ag. Denote by 
\begin{equation}\label{delta_M_t_def}
\Delta M_T= \sup_{t \geq T} |M_{t} - M_{t-}|
\end{equation}
the maximal jump size of $M_t$ after time $T$.
 Recall the notion of $\langle M \rangle_t$ from Definition \ref{def_q_variation}.

\begin{lemma}\label{lemma_we_can_apply_kallenberg}
There exist dimension-dependent constants $C_0,C_1,C_2$ such that the following bounds almost surely hold:
\begin{align}
\label{eq:boundC0}
M_T &\leq C_0 T^{1-\frac{d}{2} }, \\
\label{bound_on_biggest_jump}
\Delta M_T &\leq C_1 T^{\frac12 -\frac{d}{2}},\\
\label{bound_on_total_quad_var}
\langle M \rangle_{\infty}-\langle M \rangle_T & \leq C_2 T^{2-d}.
\end{align}
\end{lemma}

Before we prove Lemma \ref{lemma_we_can_apply_kallenberg} we use it
 to prove  Proposition \ref{prop_near_martingale_deviation}.

\begin{proof}[Proof of Proposition \ref{prop_near_martingale_deviation}]
 We first prove \eqref{eq:martingale_deviation}:
 \begin{multline}\label{eq:desired_deviation}
P\left[M_\infty > T^{1-\frac{d}{2}  + \varepsilon}\right] 
\stackrel{\eqref{eq:boundC0}}{\leq}
 P\left[M_\infty - M_T >  T^{1-\frac{d}{2}  + \varepsilon}- C_0 T^{1-\frac{d}{2} } \right]
\stackrel{(*)}{\leq}\\
  P\left[M_\infty - M_T > \frac{1}{2} T^{1-\frac{d}{2}  + \varepsilon}\right]
  \stackrel{ \eqref{eq:ineq_kallengerg}, \eqref{bound_on_biggest_jump}, \eqref{bound_on_total_quad_var} }{\leq}\\
  \exp \left(-\frac12 \frac{ \frac{1}{2} T^{1-\frac{d}{2}  + \varepsilon} }{C_1 T^{\frac12 -\frac{d}{2}}}
   \ln\left(1+\frac{\frac{1}{2} T^{1-\frac{d}{2}  + \varepsilon} C_1 T^{\frac12 -\frac{d}{2}} }
   {C_2 T^{2-d}} \right) \right)= \\
   \exp \left(- \frac{1}{4 C_1} T^{\frac12+\varepsilon}
   \ln\left(1+\frac{C_1}{2C_2} T^{-\frac12 +\varepsilon } \right) \right)
   \stackrel{(*)}{\leq} 
\exp\left(-T^{\varepsilon}\right),
  \end{multline}
where the inequalities marked by $(*)$ hold if $T$ is large enough. 
We have proved that \eqref{eq:martingale_deviation} would hold even if we defined $\delta$ to be equal to $\varepsilon$, so it also holds
if $\delta=\varepsilon/d$ as in \eqref{eq:def:epsilon_delta}.

 We now turn to \eqref{eq:martingale_deviation_far}. 
  We fix a small constant $\sigma \in (0,1)$ (to be chosen later in \eqref {choice_of_sigma} as $\sigma=\varepsilon/4$).
We keep the notation $\sigma$ with the hope that it makes the proof more transparent.   
   Given this $\sigma$ we define 
\begin{equation}
\label{new_def_T_hat}
\widehat{T} = 2^{\frac{1}{d/2 - 1 -\sigma}} \cdot |x-y|^{2-\frac{2(\varepsilon - 2\sigma)}{d-2-2\sigma}}, 
\end{equation}
so that
\begin{equation}
\label{new_def_T_hat2}
\widehat{T}^{1-\frac{d}{2} + \sigma} = \frac12 |x-y|^{2-d+\varepsilon}.
\end{equation}
Note that, since $\frac{2(\varepsilon - 2\sigma)}{d-2-2\sigma} < \frac{2\varepsilon}{d-4} \stackrel{\eqref{eq:def:epsilon_delta}}{<} 1$, we have $2-\frac{2(\varepsilon - 2\sigma)}{d-2-2\sigma} > 1$, so \eqref{new_def_T_hat} implies that
\begin{equation}\label{hat_T_and_x}
\widehat{T} > |x-y| \quad \text{if} \quad x \neq y \in \Z^d.
\end{equation}

Having fixed some $x \neq y \in \Z^d$, we now start to bound the left-hand side of \eqref{eq:martingale_deviation_far}.
\begin{multline}
\label{split_mart_dev_far}
P[ M_{\infty} > | x-y |^{2-d+\varepsilon} ]  
\stackrel{ \eqref{def_M_t_simple} }{\leq} 
P \left[  P [ \exists \, u \geq \widehat{T} \; : \; X^y_u=X^x_u \; |\; \mathcal{F}^x_\infty ]
> \frac12 | x-y |^{2-d+\varepsilon} \right]
+\\
P \left[ P [ \exists \, u \leq \widehat{T} \; : \; X^y_u=X^x_u \; |\; \mathcal{F}^x_\infty ]
> \frac12 | x-y |^{2-d+\varepsilon} \right].
\end{multline}

Assuming that $|x-y|$ is large enough (and hence $\widehat{T}$ is large enough), we bound the first term on the right-hand side of
\eqref{split_mart_dev_far} analogously to \eqref{eq:desired_deviation}, with $\widehat{T}$ in place of $T$
and $\sigma$ in place of $\varepsilon$:
\begin{equation}
 \label{eq:desired_over_2}
P \left[  P [ \exists \, u \geq \widehat{T} \; : \; X^y_u=X^x_u \; |\; \mathcal{F}^x_\infty ]
> \frac12 | x-y |^{2-d+\varepsilon} \right] 
\stackrel{ \eqref{new_def_T_hat2} }{\leq} 
 \exp\left(-\widehat{T}^{\sigma}\right)
 \stackrel{\eqref{hat_T_and_x}}{\leq } e^{-|x-y|^{\sigma} }.
\end{equation}
Now we bound the second term on the right-hand side of
\eqref{split_mart_dev_far} using Markov's inequality:
\begin{equation}\label{eq:markov_for_far_meet}
P \left[ P [ \exists \, u \leq \widehat{T} \; : \; X^y_u=X^x_u \; |\; \mathcal{F}^x_\infty ]
> \frac12 | x-y |^{2-d+\varepsilon} \right] \leq 
\frac{ P [ \exists \, u \leq \widehat{T} \; : \; X^y_u=X^x_u ] }{ \frac12 | x-y |^{2-d+\varepsilon} },
\end{equation}
and 
\begin{multline} 
\label{eq:largedev_for_T_hat_walk}
 P [ \exists \, u \leq \widehat{T} \; : \; X^y_u=X^x_u ] 
 \stackrel{\eqref{eq:difference_of_RWs_is_RW}}{=}
 P [ \exists \, u \leq 2\widehat{T} \; : \; X^{x-y}_u=0 ] 
\leq 
 P \left[ \max_{u \leq 2\widehat{T}} |X^0_u| \geq |x-y|  \right]  
\\
\stackrel{ \eqref{eq:srw_large_dev_estimate} }{\leq} 2d \exp \left( - \frac12 |x-y| \ln \left( 1+ \frac{d \cdot |x-y|}{ 2\widehat{T} } \right) \right).
\end{multline}
The expression on the right-hand side of \eqref{eq:largedev_for_T_hat_walk}
 suggests that $\widehat{T}$ should be much smaller than $|x-y|^2$. 
 With this in mind, and inspecting \eqref{new_def_T_hat}, we set 
\begin{equation}\label{choice_of_sigma}
\sigma = \varepsilon/4.
\end{equation}
 If $|x-y|$ is large enough, \eqref{new_def_T_hat} then implies that 
\begin{equation}
\label{eq:largedev_for_T_hat_walk2}
2d \exp \left( - \frac12 |x-y| \ln \left( 1+ \frac{d \cdot |x-y|}{ 2\widehat{T} } \right) \right)
\leq \exp \left( -|x-y|^{\frac{\varepsilon}{d-2}} \right).
\end{equation} 
Putting the above bounds together we obtain
\begin{multline*}
P[ M_{\infty} > | x-y |^{2-d+\varepsilon} ] 
\stackrel{ \eqref{split_mart_dev_far}, \eqref{eq:desired_over_2}, 
\eqref{eq:markov_for_far_meet}, \eqref{eq:largedev_for_T_hat_walk} ,\eqref{eq:largedev_for_T_hat_walk2}  }{ \leq }\\
 e^{-|x-y|^{\varepsilon/4} } + 
 \frac{ \exp \left( -|x-y|^{\frac{\varepsilon}{d-2}} \right) }{ \frac12 | x-y |^{2-d+\varepsilon} } 
 \stackrel{(*)}{ \leq } 
 \exp \left( -|x-y|^{\varepsilon/d} \right),
\end{multline*}
where $(*)$ holds if $|x-y|$ is large enough.
This completes the proof of \eqref{eq:martingale_deviation_far} with $\delta=\varepsilon/d$, as required by \eqref{eq:def:epsilon_delta}.
The proof of Proposition \ref{prop_near_martingale_deviation} is complete, given 
Lemma \ref{lemma_we_can_apply_kallenberg}.
\end{proof}

Now we prepare the ground for the proof of Lemma \ref{lemma_we_can_apply_kallenberg}.
We begin with stating a useful equivalent formula for the martingale $M_t,\, t \in [T, \infty)$.
\begin{claim}\label{lemma_description_M_t}
For any $t \geq T$,
\begin{equation}\label{eq:property_of_M_t}
M_t = P[\; \exists\, u \geq T: X^y_u = X^x_{u\wedge t}\;|\;\mathcal{F}^x_t].
\end{equation}
\end{claim}
\begin{proof}
Given $t \geq T$  let us define the event
\begin{equation}\label{def_eq_event_A_in_kallengerg}
 A=\{ \, \exists s\in[T,t) \; : \; X^y_s=X^x_s \, \}.
\end{equation}
 The statement follows from
\begin{align*}
&M_t \stackrel{\eqref{def_M_t_simple}}{=} 
P[ A  \; | \; \mathcal{F}^x_t ]\\&\hspace{2cm}+ \sum_{v,w \in \Z^d}
P[ A^c \cap \{ X^y_t = w\}  \; | \; \mathcal{F}^x_t ] \cdot \mathds{1} [ X^x_t = v] \cdot P[\exists s \geq 0:\; X^w_s = X^v_s]\\
&\stackrel{\eqref{eq:difference_of_RWs_is_RW}}{=} P[ A  \; | \; \mathcal{F}^x_t ]+\sum_{v,w \in \Z^d}P[ A^c \cap \{ X^y_t = w\}  \; | \; \mathcal{F}^x_t ] \cdot \mathds{1} [ X^x_t = v] \cdot P[\exists s \geq 0:\; X^w_s = v]\\
&=P[ A  \; | \; \mathcal{F}^x_t ]+
P[ A^c \cap \{\exists s\geq t: X^y_s = X^x_t\} \; | \; \mathcal{F}^x_t  ]= 
P[\, \exists\, s \geq T: X^y_s = X^x_{s\wedge t}\;|\;\mathcal{F}^x_t ].
\end{align*}
\end{proof}

\begin{definition}\label{def:M_e_force_jump}
For any $t \in [T,+\infty)$ let us define for $e \in \Z^d$, $e \sim 0$, the random variable
\begin{equation}\label{def_eq_M_e_t}
M^{(e)}_t = 
 P [\, \exists \,  u \geq T \; : \; X^y_u=X^x_{u\wedge t} + e\cdot \mathds{1}_{\{u \geq t\}} \; | \; \mathcal{F}^x_t ].
\end{equation}
\end{definition}
The $\mathcal{F}^x_t$-measurable random variable $M_t^{(e)}$ is a perturbed version of $M_t$ where we artificially force the walk $(X^x_{s})_{s \geq 0}$ to jump at time $t$ in the direction of the unit vector $e \sim 0$.
Recall that we assume that our random walks and martingales are c\`adl\`ag.

\begin{definition}
Denote by $\tau_1<\tau_2<\dots$ the jump times of the random walk $(X^x_t)$ and let $\tau_0=0$.
For any $n \geq 1$ let $e_n= X^x_{\tau_n} - X^x_{\tau_n-}$ denote the direction of the
 jump of $(X^x_t)$ at time $\tau_n$.
\end{definition}
Note that 
\begin{equation}\label{iid_exp}
\text{ $(\tau_n-\tau_{n-1})_{n \geq 1}$ are i.i.d.\ with $\mathrm{Exp}(1)$  distribution.}
\end{equation}

 The next claim states that 
  $M_t$ only jumps when $X^x_t$ jumps and in between jumps $M_t$ is constant.
\begin{claim}\label{lemma_jumps}
For any $n =1,2,\dots$ we have
\begin{align}
\label{martingale_jump_size}
M_{\tau_n}&=M^{(e_n)}_{\tau_n-},\\
\label{martingale_slope}
M_t & = M_{\tau_{n-1}}, \quad \tau_{n-1}\le t < \tau_n.
\end{align}
\end{claim}
\begin{proof}
Let $\gamma: [0,\infty) \to \mathbb{Z}^d$ be a c\`adl\`ag function with $\gamma(0)=x$. 
This $\gamma$ will play the role
of a possible realization of of $\left(X^x_u\right)_{u \geq 0}$.
 Assume that for some $T \leq s <t$ and $e \in \Z^d$, $e\sim 0$ the trajectory $\gamma$ satisfies
\begin{equation}\label{eq:assume_gamma}
 \gamma(r) = \gamma(s)\text{ for all }r\in [s,t)\quad \text{ and }\quad \gamma(t) = \gamma(t_-)+e.\end{equation}
The two statements of the claim are immediate consequences of \eqref{eq:property_of_M_t},  \eqref{def_eq_M_e_t} and 
\begin{align}
&P[\,\exists u \ge T\;:\; X^y_u = \gamma(u\wedge t)\,] =
 \lim_{r \nearrow t}P[\,\exists u \ge T\;:\; X^y_u = \gamma(u\wedge r) + e\cdot \mathds{1}_{\{u \geq r\}}\,];\label{eq:wedge_limit}\\
&P[\,\exists u \ge T\;: \;X^y_u = \gamma(u\wedge r)\,] =
 P[\,\exists u \ge T\;: \;X^y_u = \gamma(u\wedge s)\,]\text{ for all } r \in [s,t).\label{eq:wedge_inter_time}
\end{align}
\eqref{eq:wedge_inter_time} holds because, by \eqref{eq:assume_gamma}, $\gamma(u \wedge s)= \gamma(u \wedge r)$ for all $u$. To establish \eqref{eq:wedge_limit} we note that, again by  \eqref{eq:assume_gamma}, for fixed $r \in (s,t)$ the symmetric difference of the events 
$$\{ \, \exists u \ge T\;:\; X^y_u = \gamma(u\wedge t)\,\} 
\quad \text{ and }\quad 
\{\, \exists u \ge T\;:\; X^y_u = \gamma(u\wedge r) + e\cdot \mathds{1}_{\{u \geq r\}}\, \} $$ 
is contained in the event that $X^y$ has a jump between times $r$ and $t$.
\end{proof}

\begin{lemma}\label{lemma_corrector_explicit}
We have
\begin{align}
\label{jump_bounded_by_M_e} 
\Delta M_T &\leq  \sup_{t \geq T} \max_{e \sim 0} |M^{(e)}_t-M_t|,\\
\label{variation_bounded_by_M_e} 
\langle M \rangle_{t}-\langle M \rangle_T & = \frac{1}{2d} \sum_{e \sim 0} \int_T^t  (M^{(e)}_s-M_s)^2\, \mathrm{d}s. 
\end{align}
\end{lemma}

\begin{proof}
The inequality \eqref{jump_bounded_by_M_e} immediately follows from \eqref{delta_M_t_def}
 and Claim \ref{lemma_jumps}.

Now we prove \eqref{variation_bounded_by_M_e}. Recall Definition \ref{def_q_variation}.
 The right-hand side of \eqref{variation_bounded_by_M_e} is adapted to $(\mathcal{F}^x_t)$ and continuous in $t$, hence it is predictable (see Definition \ref{def_predictable_process}), thus we only need to check that for any $T \leq s \leq t$ we have
 \begin{equation}\label{mart_corrector_property}
 E[  M^2_t - M^2_s  \,   | \, \mathcal{F}^x_s  ] =
  E\left[  \frac{1}{2d} \sum_{e \sim 0} \int_s^t  (M^{(e)}_u-M_u)^2\, \mathrm{d}u  \,  
   | \,   \mathcal{F}^x_s \right].
 \end{equation}

Let us define for $\delta>0$ and $u \geq T$ the random variable
\begin{equation}\label{def_eq_psi}
 \psi^\delta_u:= 
 \frac{1}{\delta}\; E [ (M_{u+\delta} -M_u)^2  \, | \, \mathcal{F}^x_u ]\stackrel{(*)}{=}
\frac{1}{\delta} \;E [ M_{u+\delta}^2 -M_u^2  \, | \, \mathcal{F}^x_u ],
\end{equation}
where $(*)$ follows from the fact that $M_t$ is a bounded martingale.
Using  \eqref{martingale_jump_size}, \eqref{martingale_slope} and that $(X_t)$ is a continuous-time simple random walk on $\Z^d$ we obtain
\begin{equation}\label{psi_almost_sure_limit}
 \lim_{\delta \to 0_+} \psi^\delta_u= \frac{1}{2d} \sum_{e \sim 0} (M^{(e)}_u-M_u)^2, \qquad \mathbb{P}-\text{a.s.} 
 \end{equation}
It follows from the definition \eqref{def_eq_psi} that for any $\delta>0$ we have
\[ E\left[ \int_s^t \psi^\delta_u \mathrm{d}u \,   | \,  \mathcal{F}^x_s \right] 
= E\left[ \frac{1}{\delta} \int_t^{t+\delta} M^2_u \, \mathrm{d}u -
\frac{1}{\delta} \int_s^{s+\delta} M^2_u \, \mathrm{d}u  \,   | \,  \mathcal{F}^x_s \right] .
\]
From this, \eqref{iid_exp} and Claim \ref{lemma_jumps} it follows that
\begin{equation}\label{psi_integral_limit}
 \lim_{\delta \to 0_+} E\left[ \int_s^t \psi^\delta_u \mathrm{d}u \,   | \, \mathcal{F}^x_s \right] =
E\left[  M^2_t - M^2_s  \,   | \, \mathcal{F}^x_s \right],  \qquad \mathbb{P}-\text{a.s.}
\end{equation}
Now \eqref{mart_corrector_property} will follow from  \eqref{psi_almost_sure_limit} and \eqref{psi_integral_limit}
by dominated convergence as soon as we prove that for any $u \geq T$ and $0<\delta \leq 1$ we have 
$ \psi^\delta_u \leq 1$. This bound follows from  \eqref{iid_exp} and  Claim \ref{lemma_jumps}.
\end{proof}

\begin{lemma} \label{lemma_bound_on_jump_at_time_t}
There exists $C > 0$ such that for any $t \geq T \geq 1$ and $e \sim 0$, 
\begin{equation}\label{eq:mart3}
 |M^{(e)}_t - M_t| \leq C t^{\frac12 -\frac{d}{2} }.
\end{equation}
\end{lemma}

Before we prove Lemma \ref{lemma_bound_on_jump_at_time_t}, let us deduce Lemma \ref{lemma_we_can_apply_kallenberg} from it.

\begin{proof}[Proof of Lemma \ref{lemma_we_can_apply_kallenberg}]
We begin with \eqref{eq:boundC0}. We first observe that, for any $y,z \in \Z^d$,
\begin{equation}
\label{eq:bound_on_M_T_preparations}
\sum_{w \in \Z^d} p_T(y,w)\cdot P\left[\exists t \geq 0: X^w_t = X^z_t\right]
\stackrel{  \eqref{eq:meet_green}, \eqref{eq:chapman_kolmogorov} }{\leq}
\int_T^\infty p_t(y,z) \mathrm{d}t 
\stackrel{ \eqref{green_bounds} }{\leq} C_0T^{1-\frac{d}{2} }.
\end{equation}
With this at hand, we derive \eqref{eq:boundC0}:
$$
\begin{aligned}
M_T & \stackrel{ \eqref{def_M_t_simple} }{=} \sum_{z,w} \mathds{1}\{X^x_T =z\}\cdot p_T(y,w) \cdot P\left[\exists t \geq T: X^y_t = X^x_t\;|\;X^x_T = z,\;X^y_T = w\right]\\
&= \sum_z\mathds{1}\{X^x_T =z\}\cdot \sum_w  p_T(y,w) \cdot P\left[\exists t \geq 0: X^w_t = X^z_t\right] 
\stackrel{ \eqref{eq:bound_on_M_T_preparations} }{\leq} 
C_0 T^{1-\frac{d}{2} }. 
\end{aligned}
$$

The bound \eqref{bound_on_biggest_jump} follows from \eqref{jump_bounded_by_M_e} and \eqref{eq:mart3}. Now we prove 
\eqref{bound_on_total_quad_var}:
\begin{equation*}
\langle M \rangle_{\infty}-\langle M \rangle_T  \stackrel{\eqref{variation_bounded_by_M_e}}{=}
\frac{1}{2d} \sum_{e \sim 0} \int_T^\infty  (M^{(e)}_s-M_s)^2\, \mathrm{d}s 
\stackrel{\eqref{eq:mart3}}{\leq}  \int_T^\infty C s^{1-d } \, \mathrm{d}s = 
 C T^{2-d} .
\end{equation*}

\end{proof}

\begin{proof}[Proof of Lemma \ref{lemma_bound_on_jump_at_time_t}]

Given $t \geq T$ we define the event $A$ by \eqref{def_eq_event_A_in_kallengerg}. We have
\begin{align*}
M_t &\stackrel{\eqref{eq:property_of_M_t},\eqref{eq:hitting_prob_green}}{=} 
P[ A  \; | \; \mathcal{F}^x_t ]+ \sum_{v,w \in \Z^d}
P[ A^c \cap \{ X^y_t = w\}  \; | \; \mathcal{F}^x_t ] \cdot \mathds{1} [ X^x_t = v] \cdot \frac{g(v,w)}{g(0,0)},\\
M_t^{(e)} &\stackrel{\eqref{def_eq_M_e_t} ,\eqref{eq:hitting_prob_green}}{=} 
P [ A  \; | \; \mathcal{F}^x_t ]+ \sum_{v,w \in \Z^d}
P[ A^c \cap \{ X^y_t = w\}  \; | \; \mathcal{F}^x_t ] \cdot \mathds{1} [ X^x_t = v]\cdot  \frac{g(v+e,w)}{g(0,0)},
\end{align*}
thus we obtain \eqref{eq:mart3}:
\begin{multline*}
|M^{(e)}_t - M_t| \leq 
\sum_{v,w \in \Z^d}
P[ A^c \cap \{ X^y_t = w\}  \; | \; \mathcal{F}^x_t ]\cdot 
 \mathds{1} [ X^x_t = v] \cdot\frac{|g(v+e,w)-g(v,w)|}{g(0,0)}
\stackrel{ \eqref{eq:pdf_green_basic_facts} }{\leq} \\
\sum_{v,w \in \Z^d}
P[  X^y_t = w  ]\cdot  \mathds{1} [ X^x_t = v] \cdot |g(v+e,w)-g(v,w)|
\stackrel{ \eqref{eq:weighted_power_1_minus_d} }{\leq}
 C t^{\frac12 -\frac{d}{2} }.
\end{multline*}

\end{proof}


\section{Concluding remarks}
\label{section_remarks}

\begin{remark}\label{remark_conjecture_one_block}
In order to informally explain why $d \geq 5$ is easier than $d=4$ and especially $d=3$ when it comes to proving 
$\alpha_c>0$ for the nearest-neighbour voter model on $\Z^d$, let us introduce a toy model.
Recall the graphical construction \eqref{graphical_walk_poisson} of the coalescing random walks 
 $\left( Y_t^x \right)_{t \geq 0, x \in \Z^d}$ and assume that $R=1$.
  Denote by $\mu^*$ the law of the random element $\left( \xi(x) \right)_{x \in \Z^d}$ of $\{0,1\}^{\Z^d}$
  that we obtain by defining
\[ \xi(x)= \begin{cases} 1 & \text{ if }\;   Y_t^x=Y_t^0 \text{ for some } t \geq 0, \\
0 & \text{ otherwise.}
\end{cases}  \]
In words, the coalescence class of the origin is occupied, and every other vertex of $\Z^d$ is vacant.

As a first step in the direction of \eqref{eq:noperc}, one might first want to show 
\begin{equation}\label{one_block_crossing} 
\lim_{L \to \infty} \mu^* \left[ B(0,L) \stackrel{*\xi}{\longleftrightarrow} B(0,2L)^c \right]=0.
\end{equation}
When $d \geq 5$, this follows from \eqref{eq:noperc} and the fact that
$$ \mu_{\alpha_0}\left[B(0,L) \stackrel{*\xi}{\longleftrightarrow} B(0,2L)^c\right] \geq \alpha_0 \cdot \mu^* \left[ B(0,L) \stackrel{*\xi}{\longleftrightarrow} B(0,2L)^c \right].$$
We believe that \eqref{one_block_crossing} can be proved
in the $d=4$ case using a careful implementation of similar ideas. 
However, the question of  \eqref{one_block_crossing} is to the best of our knowledge open in the $d=3$ case and we think new ideas are needed for the proof. 

We also note that if $d \geq 3$ and $\xi$ has law $\mu^*$, then
 by \cite[Theorem 3]{voter_cluster_scaling_limit_super} the sequence of rescaled random measures
  $\frac{1}{N} \sum_{x \in \Z^d} \xi(x) \delta_{x/\sqrt{N}}$ converge in law with respect to the topology of
  vague
  convergence on the space of Radon measures on $\R^d$, and the limit object is a variant of super-Brownian motion.
  It is also known (see \cite[Section 4]{P95} and \cite[Theorem III.6.3]{P02}) that for $d \geq 4$
   the closed support of super-Brownian
  motion is totally disconnected, but the open problem stated on \cite[page 119]{P02} is still not solved, i.e. the
  closed support of super-Brownian
  motion in $d=3$ may or may not contain non-trivial connected subsets.
   The combination of these facts also indicate that  \eqref{one_block_crossing} may be easier to verify for $d=4$
  than for $d=3$.
  
\end{remark}

\begin{remark}\label{remark_why_no_L_equals_1_in_nearest_neighbour_case}
One reason why the proof of Theorem \ref{thm:spread_out} in Section \ref{section_spread_out} is so short is that
we chose $L=1$ so that crossing an annulus on the bottom level of our renormalization scheme just means that a single
site is of type $1$. Let us explain why this choice is insufficient when it comes to proving Theorem \ref{thm:nearest_neighbour}. In this heuristic argument we will also keep track
of the dependence on $\ell$ of the combinatorial terms and probabilities
 in order to make sure that making $\ell$ large will not be helpful either.

 If $R=1$ and $L=1$, then (similarly to \eqref{eq:binary_union_bound_spread_out}) we obtain
\begin{equation}\label{naive_L_1_nearest_neighbour}
\mu_{\alpha}[ B(L_N-2) \stackrel{*\xi}{\longleftrightarrow} B(2L_N)^c ] 
\stackrel{ \eqref{eq:dualityinf} }{\leq} 
\widehat{C}(\ell)^{2^N} \max_{ \mathcal{T} \in \Lambda_{N} } 
 \mathbb{E}\left[\alpha^{\mathcal{N}_\infty(\mathcal{X}_{\mathcal{T}} )}\right],
\end{equation}
where $\widehat{C}(\ell) \asymp \ell^{2d-2}$. For any $\mathcal{T} \in \Lambda_{N}$, we have 
 $\mathbb{E}\left[\alpha^{\mathcal{N}_\infty(\mathcal{X}_{\mathcal{T}} )}\right] \geq 
 \alpha \mathbb{P}[\, \mathcal{N}_\infty(\mathcal{X}_{\mathcal{T}} )=1  \,]$.
Recalling \eqref{def_eq_tree_leaf_embedded_set}   we can construct a scenario where
 $\mathcal{N}_{\infty}(\mathcal{X})=1$ (i.e., all walkers coalesce)
  by first coalescing the walkers starting from $\mathcal{T}(m_1)$ and $\mathcal{T}(m_2)$ (see \eqref{children_of_m}) for every
  $m \in T_{(N-1)}$, and then coalescing the resulting walkers with their respective ``sibling'', etc.
   Recursively repeating this procedure   from the leaves to the root of the binary tree
   we obtain that
\[  \mathbb{P}[\, \mathcal{N}_\infty(\mathcal{X}_{\mathcal{T}} )=1  \,]
\gtrsim \left(C \ell^{2-d}\right)^{2^{N-1}}\cdot \left(C \ell^{2(2-d)} \right)^{2^{N-2}} 
\dots \left(C \ell^{N(2-d)}\right)^{2^0}
\asymp \left(C\ell^{(2-d)\cdot 2}\right)^{2^N}.
\]
 If we use this to (heuristically) lower bound the right-hand side of \eqref{naive_L_1_nearest_neighbour}, we obtain
 \[ \widehat{C}(\ell)^{2^N} \cdot \alpha \cdot \left(C\ell^{(2-d)\cdot 2}\right)^{2^N} \asymp 
 \left( \widetilde{C} \ell^2 \right)^{2^N}, \]
 which may go to infinity as $N \to \infty$ if the constant $\widetilde{C}=\widetilde{C}(d)$ happens to be too big.
\end{remark}

\begin{remark}\label{remark_why_no_d_3_4}
Let us explain why the method of Section \ref{section_nearest_neighbour} fails to prove \eqref{eq:noperc} if $d=3,4$ and $R=1$ by arguing that
the right-hand side of \eqref{eq:last_union_bound} does not go to zero. Rather than fixing the value of $\ell$
as in \eqref{eq:choice_ell}, in this heuristic argument we will keep track
of the dependence on $\ell$ as well as on $L$ of the terms on the right-hand side of \eqref{eq:last_union_bound}.
 If we assume $\mathcal{Y}=\emptyset$, then by \eqref{calX_plus_calY} we have $|\mathcal{X}|=2 \cdot 2^N$.  
Similarly to Remark \ref{remark_why_no_L_equals_1_in_nearest_neighbour_case},
  we will  bound the probability of the event on the right-hand side of
\eqref{eq:last_union_bound} from below. 
 For any fixed $\alpha>0$ we can bound
\begin{equation}\label{eq:remark_d_3_4_lower}
 \mathbb{P} \left[
\bigcap_{ \substack{ \{x,z\}\in \mathcal{X}\\ |x- z|=1 }} F_{x,z} \right] 
\stackrel{ \eqref{eq:defN}, \eqref{def_eq_F_xy}  }{\geq}
\alpha \mathbb{P} \left[ \mathcal{N}_{\infty}(\mathcal{X})=1, \;
\bigcap_{ \substack{ \{x,z\}\in \mathcal{X}\\ |x- z|=1 }} E_{x,z} \right].
\end{equation}
Now the probability that $E_{x,z}$ occurs and yet $Y^x_t=Y^z_t$ for some $t>T$ is roughly
$\sqrt{T}^{2-d}=L^{(1-\varepsilon/2)(2-d)}$
 by \eqref{green_bounds} and \eqref{eq:def_T}, moreover we can use the binary tree structure of $\mathcal{X}$ to construct a scenario where
 $\mathcal{N}_{\infty}(\mathcal{X})=1$ and 
 give a (heuristic) lower bound on the probability on the right-hand side of \eqref{eq:remark_d_3_4_lower}
 by
 \[ \left(L^{(1-\varepsilon/2)(2-d)}\right)^{2^N} \cdot \prod_{k=1}^N \left( L \ell^k \right)^{(2-d) 2^{N-k}}
 \asymp \left(L^{(2-d)(2-\varepsilon/2) } \cdot \ell^{(2-d)\cdot 2  } \right)^{2^N} . \]
 If we multiply this with the combinatorial  term $\left(L^d \ell^{2d-2}\right)^{2^N}$
 that appears on the right-hand side of \eqref{eq:last_union_bound} then the resulting product goes to infinity as $N \to \infty$.
 \end{remark}

\begin{remark}\label{remark:why_no_annih_in_nearest_neighbour}
 Let us explain why the ``decorrelation via annihilation'' method developed in Section \ref{section_spread_out} cannot be used to prove Theorem \ref{thm:nearest_neighbour}. Let us assume $\mathcal{Y}=\emptyset$ 
(so that by \eqref{calX_plus_calY} we have $|\mathcal{X}|=2 \cdot 2^N$)
and bound the probability of the event of the right-hand side of \eqref{eq:last_union_bound}:
\begin{equation*}
 \mathbb{P} \left[
\bigcap_{ \substack{ \{x,z\}\in \mathcal{X}\\ |x- z|=1 }} F_{x,z} \right] 
\stackrel{ \eqref{eq:betaalpha}, \eqref{eq:bound_Zinf0} }{ \leq} 
E \left[ \alpha^{\mathcal{N}^Z_\infty}\right].
\end{equation*}
 Now we try to bound this using the idea of Lemma \ref{lemma:coalescing_dominates_annihilating}, i.e., we let random walks 
starting from the vertices of $\mathcal{X}$ run independently until time $T$ and then we let them annihilate each other.
Let us denote by $\mathcal{N}^{Z'}_\infty$ the number of walkers that do not get annihilated. 
Similarly to Lemma \ref{lemma:coalescing_dominates_annihilating}, we have
$E \left[ \alpha^{\mathcal{N}^Z_\infty}\right] \leq E \left[ \alpha^{\mathcal{N}^{Z'}_\infty}\right]$, but using an argument similar to the one used in Remark \ref{remark_why_no_d_3_4} we can (non-rigorously) bound
\[E \left[ \alpha^{\mathcal{N}^{Z'}_\infty}\right] \geq P [ \, \mathcal{N}^{Z'}_\infty =0 \, ]
\gtrsim \left(L^{(1-\varepsilon/2)(2-d)}\right)^{2^N}, 
\]
and this term is not small enough to beat the combinatorial term $(C L^d)^{2^N}$ on the right-hand side of
 \eqref{eq:last_union_bound}.
\end{remark}

\bigskip

{\bf Acknowledgements:} We would like to thank Alain-Sol Sznitman for bringing the problem under study here  to our attention and hence initiating this project. We would also like to thank Roberto Imbuzeiro Oliveira for suggesting useful references on martingale concentration inequalities, Tim Hulshof for directing us to the reference \cite{vdHofstad_Hara_Slade}, and Ed Perkins for inspiring discussions about the disconnecedness of the support of super-Brownian motion. We also thank two anonymous
referees for their careful reading of the manuscript and helpful comments.

The work of Bal\'azs R\'ath is partially supported by OTKA (Hungarian
National Research Fund) grant K100473, the Postdoctoral Fellowship of
the Hungarian Academy of Sciences and the Bolyai Research Scholarship
of the Hungarian Academy of Sciences.

\end{document}